\documentclass[12pt]{article}%
\usepackage{amsmath}
\usepackage{amsfonts}
\usepackage{amssymb}
\usepackage{graphicx}
\usepackage{comment}
\usepackage[margin=1 in]{geometry}
\usepackage{setspace}%
\providecommand{\U}[1]{\protect\rule{.1in}{.1in}}
\numberwithin{equation}{section}
\providecommand{\U}[1]{\protect\rule{.1in}{.1in}}
\providecommand{\U}[1]{\protect\rule{.1in}{.1in}}
\providecommand{\U}[1]{\protect\rule{.1in}{.1in}}
\newtheorem{theorem}{Theorem}
\newtheorem{acknowledgement}[theorem]{Acknowledgement}

\newtheorem{condition}[theorem]{Condition}

\newtheorem{definition}[theorem]{Definition}
\newtheorem{example}[theorem]{Example}

\newtheorem{lemma}[theorem]{Lemma}

\newtheorem{proposition}[theorem]{Proposition}

\newenvironment{proof}[1][Proof]{\noindent\textbf{#1.} }{\ \rule{0.5em}{0.5em}}
\begin{document}

\title{Frames arising from irreducible solvable actions Part I}
\author{Vignon Oussa}
\maketitle

\begin{abstract}
Let $G$ be a simply connected, connected completely solvable Lie group with
Lie algebra $\mathfrak{g}=\mathfrak{p}+\mathfrak{m}.$ Next, let $\pi$ be an
infinite-dimensional unitary irreducible representation of $G$ obtained by
inducing a character from a closed normal subgroup $P=\exp\mathfrak{p}$ of
$G.$ Additionally, we assume that $G=P\rtimes M,$ $M=\exp\mathfrak{m}$ is a
closed subgroup of $G,$ $d\mu_{M}$ is a fixed Haar measure on the solvable Lie
group $M$ and there exists a linear functional $\lambda\in\mathfrak{p}^{\ast}$
such that the representation $\pi=\pi_{\lambda}=\mathrm{ind}_{P}^{G}\left(
\chi_{\lambda}\right)  $ is realized as acting in $L^{2}\left(  M,d\mu
_{M}\right)  .$ Making no assumption on the integrability of $\pi_{\lambda}$,
we describe explicitly a discrete subgroup $\Gamma\subset G$ and a vector
$\mathbf{f}\in L^{2}\left(  M,d\mu_{M}\right)  $ such that $\pi_{\lambda
}\left(  \Gamma\right)  \mathbf{f}$ is a tight frame for $L^{2}\left(
M,d\mu_{M}\right)  .$ We also construct compactly supported smooth functions
$\mathbf{s}$ and discrete subsets $\Gamma\subset G$ such that $\pi_{\lambda
}\left(  \Gamma\right)  \mathbf{s}$ is a frame for $L^{2}\left(  M,d\mu
_{M}\right)  .$

\end{abstract}

\section{Introduction and preliminaries}

Let $G$ be a locally compact group, and let $\pi$ be a strongly continuous
unitary irreducible representation of $G$ acting in an infinite-dimensional
Hilbert space $\mathcal{H}_{\pi}.$ Next, let $\Gamma$ be a discrete subset of
$G$ and fix $\mathbf{f}\in\mathcal{H}_{\pi}.$ We say that $\left\{  \pi\left(
\gamma\right)  \mathbf{f}:\gamma\in\Gamma\right\}  $ is a \textbf{frame} for
$\mathcal{H}_{\pi}$ if there exist positive constants $a\leq b$ (frame bounds)
such that
\[
a\left\Vert \mathbf{h}\right\Vert _{\mathcal{H}_{\pi}}^{2}\leq\sum_{\gamma
\in\Gamma}\left\vert \left\langle \mathbf{h,}\pi(\gamma)\mathbf{f}%
\right\rangle _{\mathcal{H}_{\pi}}\right\vert ^{2}\leq b\left\Vert
\mathbf{h}\right\Vert _{\mathcal{H}_{\pi}}^{2}%
\]
for any vector $\mathbf{h}$ in $\mathcal{H}_{\pi}.$ The frame operator $S$ is
defined as
\[
S\mathbf{h=}\sum_{\gamma\in\Gamma}\left\langle \mathbf{h,}\pi(\gamma
)\mathbf{f}\right\rangle _{\mathcal{H}_{\pi}}\pi(\gamma)\mathbf{f}\text{
\ (}\mathbf{h}\in\mathcal{H}_{\pi}\text{)}%
\]
and if $\left\{  \pi\left(  \gamma\right)  \mathbf{f}:\gamma\in\Gamma\right\}
$ is a frame for $\mathcal{H}_{\pi}$ then $S$ is invertible and every vector
$\mathbf{h}$ in $\mathcal{H}_{\pi}$ admits the expansion
\[
\mathbf{h=}\sum_{\gamma\in\Gamma}\left\langle \mathbf{h},S^{-1}\pi\left(
\gamma\right)  \mathbf{f}\right\rangle _{\mathcal{H}_{\pi}}\pi\left(
\gamma\right)  \mathbf{f}%
\]
with convergence in the norm of $\mathcal{H}_{\pi}$. If $a=b$ then $\left\{
\pi\left(  \gamma\right)  \mathbf{f}:\gamma\in\Gamma\right\}  $ is called a
\textbf{tight frame} and every vector $\mathbf{h}\in\mathcal{H}_{\pi}$ admits
the simpler series expansion
\[
\mathbf{h}=\sum_{\gamma\in\Gamma}\left\langle \mathbf{h,}\pi\left(
\gamma\right)  \frac{\mathbf{f}}{\sqrt{a}}\right\rangle _{\mathcal{H}_{\pi}%
}\pi\left(  \gamma\right)  \frac{\mathbf{f}}{\sqrt{a}}\mathbf{.}%
\]
We are interested in finding conditions under which there exist a discrete
subset $\Gamma\subset G$ and a vector $\mathbf{f}\in\mathcal{H}_{\pi}$ such
that the collection $\left\{  \pi\left(  \gamma\right)  \mathbf{f}:\gamma
\in\Gamma\right\}  $ is a (tight) frame for $\mathcal{H}_{\pi}.$ Moreover, if
there exist frames generated by $\pi,$ we would also like to present an
explicit procedure for the construction of $\Gamma$ and $\mathbf{f}$ such that
$\left\{  \pi\left(  \gamma\right)  \mathbf{f}:\gamma\in\Gamma\right\}  $ is a
frame for $\mathcal{H}_{\pi}.$ If $\pi$ is an irreducible and integrable
representation, then the coorbit theory \cite{coorbit,coorbit1,coorbit2} which
was developed by Feichtinger and Gr\"{o}chenig provides a powerful and
flexible discretization scheme. For example, the theory of coorbit has proved
to be quite successful in the context of shearlets \cite{shear1, shear2}. In a
more general direction, the work contained in \cite{vanishing,coorbit}
addresses the case where $G=\mathbb{R}^{d}\rtimes H$ (where $H$ is a closed
subgroup of $\mathrm{GL}\left(  \mathbb{R}^{d}\right)  $) is the semi-direct
product group with multiplication law given by
\[
\left(  x,M\right)  \left(  x^{\prime},M^{\prime}\right)  =\left(
x+Mx^{\prime},MM^{\prime}\right)
\]
and $\pi$ is the quasiregular representation of $G$ which is realized as
acting in $L^{2}\left(  \mathbb{R}^{d}\right)  $ such that
\begin{equation}
\left[  \pi\left(  x,id_{H}\right)  \mathbf{f}\right]  \left(  t\right)
=\mathbf{f}\left(  t-x\right)  \text{ and }\left[  \pi\left(  0,M\right)
\mathbf{f}\right]  \left(  t\right)  =\left\vert \det M\right\vert
^{-1/2}\mathbf{f}\left(  M^{-1}t\right)  .\label{quasi}%
\end{equation}
It is well-known that if the quasiregular representation $\pi$ is irreducible
then it is integrable as well \cite{coorbit}. As such, the coorbit machinery
of Feichtinger and Gr\"{o}chenig can be applied to construct frames for a
large class of Banach spaces. In fact, it is proved in \cite{Tousi} that if
$\mathbf{f}\in L^{2}\left(  \mathbb{R}^{d}\right)  $ satisfies some decay,
smoothness and vanishing moments conditions, then there exists a discrete
subset $\Gamma\subset\mathbb{R}^{d}\rtimes H$ such that $\left\{  \pi\left(
\gamma\right)  \mathbf{f}:\gamma\in\Gamma\right\}  $ is a wavelet frame for
the Hilbert space on which the representation $\pi$ is acting \cite{Tousi}.
Unfortunately, the theory of coorbit heavily depends on the integrability of
$\pi.$ Thus, a large class of irreducible representations are automatically
excluded because they do not fit within the scope of the theory. \vskip0.1cm

The main objective of the present work is to provide a method for constructing
frames arising from the action of irreducible representations of some solvable
Lie groups which does not depend on any type of integrability condition on the
representations of interest. To be more specific, let $G=P\rtimes M$ be a
simply connected, connected completely solvable Lie group such that
$P=\exp\mathfrak{p}$ and $M=\exp\mathfrak{m}$ are closed solvable Lie
subgroups of $G.$ Moreover, we shall assume that there exists a unitary
character $\chi$ of $P$ such that
\[
\pi=\mathrm{ind}_{P}^{G}\left(  \chi\right)
\]
is an irreducible representation of $G$ realized as acting on functions
defined on the conormal subgroup $M.$ The space on which the representation
$\pi$ is acting consists of measurable functions on $M$ which are
square-integrable with respect to a fixed Haar measure on $M.$ Without
imposing any additional assumption on the integrability of the representation
$\pi,$ we shall provide a unified and explicit procedure which can be
exploited to construct discrete tight frames, and smooth compactly supported
frames generated by the action of $\pi.$

\subsection{Notation}

\begin{itemize}
\item Let $T$ be a linear operator acting on a vector space spanned by an
ordered basis $\mathfrak{B}.$ The matrix representation of the linear operator
$T$ is denoted $\left[  T\right]  _{\mathfrak{B}}$ and $T^{\ast}$ stands for
the adjoint of the linear operator $T.$

\item The transpose of a matrix $M$ is denoted $M^{T}.$ Moreover, the inverse
transpose of $M$ is written as $M^{\top}.$

\item Let $Q$ be a linear operator acting on an $n$-dimensional real vector
space $V.$ The norm of the matrix $Q$ induced by the max-norm of the vector
space $V$ is given by
\[
\left\Vert Q\right\Vert _{\infty}=\sup\left\{  \left\Vert Qv\right\Vert
_{\max}:v\in V\text{ and }\left\Vert v\right\Vert _{\max}=1\right\}
\]
and the max-norm of an arbitrary vector is
\[
\left\Vert v\right\Vert _{\max}=\max\left\{  \left\vert v_{k}\right\vert
:1\leq k\leq n\right\}  .
\]

\item Let $A$ be a Lebesgue measurable subset of $%
\mathbb{R}
^{d}.$ The Lebesgue measure of $A$ is denoted $\left\vert A\right\vert .$

\item Let $\mathfrak{s}$ be Lie algebra, and let $\lambda$ be a linear
functional in the dual of $\mathfrak{s}$ which we denote by $\mathfrak{s}%
^{\ast}.$ A subalgebra (or ideal) $\mathfrak{p}$ of $\mathfrak{s}$ is said to
be subordinated to the linear functional $\lambda$ if
\[
\left[  \mathfrak{p,p}\right]  =%
\mathbb{R}
\text{-span}\left\{  \left[  X,Y\right]  :X,Y\in\mathfrak{p}\right\}
\]
is contained in the kernel of the linear functional $\lambda.$
\end{itemize}

\subsection{Completely solvable Lie groups}

Let $\mathfrak{g}$ be a finite-dimensional Lie algebra over $\mathbb{R}.$
Given subsets $\mathfrak{r,s\subseteq g,}$ we define $\left[  \mathfrak{r,s}%
\right]  $ as the linear span of vectors
\[
\left[  X,Y\right]  =XY-YX
\]
such that $X\in\mathfrak{r}$ and $Y\in\mathfrak{s}.$ Put $\mathfrak{g}%
_{\left(  0\right)  }=\mathfrak{g}$ and define in a recursive fashion
\[
\mathfrak{g}_{\left(  k\right)  }=\left[  \mathfrak{g}_{\left(  k-1\right)
},\mathfrak{g}_{\left(  k-1\right)  }\right]  .
\]
The sequence
\[
\mathfrak{g}_{\left(  0\right)  }\supseteq\mathfrak{g}_{\left(  1\right)
}\supseteq\mathfrak{g}_{\left(  2\right)  }\cdots
\]
is called the \textbf{derived series} of the Lie algebra $\mathfrak{g}.$ In a
similar fashion, we define the \textbf{descending central series} of
$\mathfrak{g}$ inductively as follows
\[
\mathfrak{g}^{\left(  0\right)  }=\mathfrak{g,g}^{\left(  k\right)  }=\left[
\mathfrak{g}^{\left(  k-1\right)  },\mathfrak{g}\right]  .
\]
A Lie algebra is \textbf{solvable} if there exists a natural number $n$ such
that $\dim\mathfrak{g}_{\left(  n\right)  }=0$. Additionally, a Lie algebra
$\mathfrak{g}$ is called \textbf{nilpotent} if $\dim\mathfrak{g}^{\left(
n\right)  }=0$ for some natural number $n.$ From the definitions provided
above, it is clear that $\left[  \mathfrak{g}_{\left(  k-1\right)
},\mathfrak{g}_{\left(  k-1\right)  }\right]  \subseteq\left[  \mathfrak{g}%
^{\left(  k-1\right)  },\mathfrak{g}\right]  $ for all $k.$ As such, every
nilpotent Lie algebra is necessarily solvable. However, there exist solvable
Lie algebras which are not nilpotent. Additionally, a solvable Lie algebra
$\mathfrak{g}$ is called \textbf{completely solvable} if for any
$Z\in\mathfrak{g},$ the spectrum of the linear operator
\[
\mathfrak{g}\backepsilon X\mapsto\left[  Z,X\right]  =ZX-XZ
\]
is a subset of the reals. If $\mathfrak{g}$ is nilpotent then for any given
$Z\in\mathfrak{g,}$ the spectrum of the linear operator $X\mapsto\left[
Z,X\right]  $ must coincide with $\left\{  0\right\}  .$ Furthermore, it is
well-known that if $\mathfrak{g}$ is a completely solvable Lie algebra then
there exists an ordered basis $\mathfrak{B}=\left(  Z_{1},\cdots,Z_{d}\right)
$ of $\mathfrak{g}$ such that each $\mathfrak{g}_{i}=%
\mathbb{R}
$-span $\left\{  Z_{1},\cdots,Z_{i}\right\}  $ is an ideal in $\mathfrak{g.}$
Such a basis $\mathfrak{B}$ is called a \textbf{strong Malcev basis} for the
Lie algebra. If $G$ is a simply connected, connected Lie group with a
completely solvable finite-dimensional real Lie algebra $\mathfrak{g}$ then,
$G$ is called a \textbf{completely solvable} Lie group. As it is well-known,
if $G$ is completely solvable then the exponential map defines a bi-analytic
bijection between $\mathfrak{g}$ and $G,$ and the inverse of the exponential
map is denoted $\log$. If $\mathfrak{B}=\left(  Z_{1},\cdots,Z_{d}\right)  $
is a strong Malcev basis for $\mathfrak{g}$ then the map
\[
\left(  t_{1},\cdots,t_{d}\right)  \mapsto\exp\left(  t_{1}Z_{1}\right)
\exp\left(  t_{2}Z_{2}\right)  \cdots\exp\left(  t_{d}Z_{d}\right)
\]
defines an analytic diffeomorphism from $%
\mathbb{R}
^{d}$ onto $G$ (see \cite{Vara}, Theorem $3.18.11.$) This map induces a system
of coordinates on $G$ called the \textbf{canonical coordinates of the second
kind}. Let $\mathfrak{p}$ be a subalgebra of $\mathfrak{g.}$ Next, let
$P=\exp\mathfrak{p.}$ If $\mathfrak{p}$ is subordinated to the linear
functional $\lambda,$ then
\[
\chi_{\lambda}\left(  \exp X\right)  =e^{2\pi i\left\langle \lambda
,X\right\rangle }=e^{2\pi i\lambda\left(  X\right)  }%
\]
defines a continuous one-dimensional representation of $P$, and $\chi
_{\lambda}$ is called a \textbf{unitary character} of $P.$ Given an element
$x=\exp\left(  X\right)  \in G$, we define the linear maps $\mathrm{ad}\left(
X\right)  ,$ $Ad\left(  x\right)  $ acting on the Lie algebra of $G$ as
follows: $\mathrm{ad}\left(  X\right)  \left(  Y\right)  =\left[  X,Y\right]
$ and $Ad\left(  x\right)  =e^{\mathrm{ad}\left(  X\right)  }.$ Let
\[
\mathbf{O}_{\lambda}=\left\{  Ad\left(  x^{-1}\right)  ^{\ast}\lambda:x\in
G\right\}
\]
be the \textbf{coadjoint orbit} of the linear functional $\lambda,$ and let
$\mathfrak{p}^{\bot}$ be the orthogonal complement of $\mathfrak{p}$ in the
dual vector space $\mathfrak{g}^{\ast}.$ That is,
\[
\mathfrak{p}^{\bot}=\left\{  \ell\in\mathfrak{g}^{\ast}:\ell\left(  X\right)
=0\text{ for all }X\in\mathfrak{p}\right\}  .
\]
It is well-known that the induced representation
\[
\pi_{\lambda}=\mathrm{ind}_{P}^{G}\left(  \chi_{\lambda}\right)
\]
is irreducible if and only if \cite{Pukansky}

\begin{enumerate}
\item $\dim\left(  \mathfrak{p}\right)  =\dim\left(  \mathfrak{g}\right)
-\frac{1}{2}\dim\left(  \mathbf{O}_{\lambda}\right)  $

\item Pukansky's condition holds. That is, $\lambda+\mathfrak{p}^{\bot
}\subseteq\mathbf{O}_{\lambda}$.
\end{enumerate}

\noindent Moreover, for every irreducible unitary representation of $G,$ there
exists a linear functional $\ell\in\mathfrak{g}^{\ast}$ and a subalgebra
$\mathfrak{p}_{\ell}$ subordinated to $\ell$ such that the given
representation is unitarily equivalent to the induced representation
$\pi_{\ell}=\mathrm{ind}_{\exp\mathfrak{p}_{\ell}}^{G}\left(  \chi_{\ell
}\right)  .$ Let us now suppose that $G$ is a solvable Lie group satisfying
the following.

\begin{condition}
\label{cond} $G$ is connected, simply connected completely solvable Lie group
of the type $G=P\rtimes M$ where $P=\exp\mathfrak{p},M=\exp\mathfrak{m}$ are
closed subgroups of $G.$ Moreover, there exists a linear functional $\lambda$
in $\mathfrak{p}^{\ast}$ such that the induced representation $\pi_{\lambda
}=\mathrm{ind}_{P}^{G}\left(  \chi_{\lambda}\right)  $ which is realized as
acting in $L^{2}\left(  M,d\mu_{M}\right)  $ is an irreducible representation
of $G$ ($d\mu_{M}$ is a fixed Haar measure on the solvable subgroup $M.$)
\end{condition}

Let $\left(  A_{1},A_{2},\cdots,A_{n_{2}}\right)  $ be a fixed strong Malcev
basis for the Lie algebra $\mathfrak{m.}$ For a fixed $m\in M,$ there exists a
unique element $a=\left(  a_{1},a_{2},\cdots,a_{n_{2}}\right)  \in
\mathbb{R}^{n_{2}}$ such that
\[
m=\exp\left(  a_{1}A_{1}\right)  \cdots\exp\left(  a_{n_{2}}A_{n_{2}}\right)
.
\]
Letting $dA$ be the Lebesgue measure on the Lie algebra of $M,$ a left Haar
measure on $M$ is up to multiplication by a constant uniquely determined as
follows (see \cite{Faraut} Page $90$)%
\[
d\mu_{M}\left(  \exp\left(  \sum_{k-1}^{n_{2}}a_{k}A_{k}\right)  \right)
=d\mu_{M}\left(  \exp\left(  A\right)  \right)  =\left\vert \det\left(
\frac{id-e^{-\mathrm{ad}\left(  A\right)  }}{\mathrm{ad}\left(  A\right)
}\right)  \right\vert dA.
\]
Put
\[
w\left(  A\right)  =\left\vert \det\left(  \frac{id-e^{-\mathrm{ad}\left(
A\right)  }}{\mathrm{ad}\left(  A\right)  }\right)  \right\vert .
\]
The function $w$ is a non-vanishing smooth positive function on the Lie
algebra $\mathfrak{m}$ satisfying $w\left(  0\right)  =1.$ Moreover, the
modular function of $M$ is given by
\[
\Delta_{M}\left(  \exp\left(  a_{1}A_{1}\right)  \cdots\exp\left(  a_{n_{2}%
}A_{n_{2}}\right)  \right)  =\left\vert \det Ad\left(  \exp\left(  a_{1}%
A_{1}\right)  \cdots\exp\left(  a_{n_{2}}A_{n_{2}}\right)  \right)
\right\vert ^{-1}%
\]
and if $\mathfrak{m}$ is a nilpotent algebra then $\mathrm{ad}\left(
A\right)  $ is nilpotent and $w\left(  A\right)  =1.$ Put $\dim\left(
P\right)  =n_{1}$ and $\dim\left(  M\right)  =n_{2}.$ The mapping
\[
\left(  x,a\right)  =\left(  x_{1},x_{2},\cdots,x_{n_{1}},a_{1},a_{2}%
,\cdots,a_{n_{2}}\right)  \mapsto\exp\left(  \sum_{k=1}^{n_{1}}x_{k}%
X_{k}\right)  \exp\left(  a_{1}A_{1}\right)  \cdots\exp\left(  a_{n_{2}%
}A_{n_{2}}\right)
\]
defines an analytic diffeomorphism between the Lie algebra $\mathfrak{g}$ and
its Lie group $G$. This diffeormorphism induces a system of coordinates on the
Lie group $G.$ Since every element $g\in G$ is uniquely written as
\[
g=\left(  p,n\right)  \in P\times M
\]
it follows that
\[
\left[  \pi_{\lambda}\left(  g\right)  \mathbf{f}\right]  \left(  m\right)
=\left[  \pi_{\lambda}\left(  p,n\right)  \mathbf{f}\right]  \left(  m\right)
=e^{2\pi i\left\langle \lambda,\log\left(  m^{-1}pm\right)  \right\rangle
}\mathbf{f}\left(  n^{-1}m\right)  .
\]
If $p=\exp\left(  X\right)  $ for some $X=\sum_{k=1}^{n_{1}}x_{k}X_{k}%
\in\mathfrak{p}$ then there exist some real numbers $a_{k}$ such that
\[
\left[  \pi_{\lambda}\left(  p,n\right)  \mathbf{f}\right]  \left(  m\right)
=e^{2\pi i\left\langle \lambda,e^{-\mathrm{ad}\left(  a_{n_{2}}A_{n_{2}%
}\right)  }\cdots e^{-\mathrm{ad}\left(  a_{2}A_{2}\right)  }e^{-\mathrm{ad}%
\left(  a_{1}A_{1}\right)  }X\right\rangle }\mathbf{f}\left(  n^{-1}m\right)
.
\]
Next, define the linear map $\mathbf{C}\left(  a\right)  :\mathfrak{p}%
\rightarrow\mathfrak{p}$ such that
\begin{equation}
\mathbf{C}\left(  a\right)  =\mathbf{C}\left(  a_{1},a_{2},\cdots,a_{n_{2}%
}\right)  =\left.  \left(  e^{-\mathrm{ad}\left(  a_{n_{2}}A_{n_{2}}\right)
}\cdots e^{-\mathrm{ad}\left(  a_{2}A_{2}\right)  }e^{-\mathrm{ad}\left(
a_{1}A_{1}\right)  }\right)  \right\vert _{\mathfrak{p}}.
\end{equation}
Given $r=\left(  r_{1},\cdots,r_{n_{2}}\right)  \in%
\mathbb{R}
^{n_{2}},$ let
\[
\mathbf{Q}\left(  a,\lambda,X\right)  =\left\langle \mathbf{C}\left(
a\right)  ^{\ast}\lambda,X\right\rangle
\]
and
\[
\mathbf{e}\left(  r\right)  =\exp\left(  r_{1}A_{1}\right)  \cdots\exp\left(
r_{n_{2}}A_{n_{2}}\right)  .
\]
Then
\begin{equation}
\left[  \pi_{\lambda}\left(  \exp\left(  \sum_{k=1}^{n_{1}}x_{k}X_{k}\right)
\mathbf{e}\left(  t\right)  \right)  \mathbf{f}\right]  \left(  \mathbf{e}%
\left(  a\right)  \right)  =e^{2\pi i\left(  \mathbf{Q}\left(  a,\lambda
,x\right)  \right)  }\mathbf{f}\left(  \mathbf{e}\left(  t\right)
^{-1}\mathbf{e}\left(  a\right)  \right)  .\label{irre}%
\end{equation}

\subsection{Wavelet theory and time-frequency analysis}

In order to convince the reader that the class of groups under investigation
is relevant to wavelet theory and time-frequency analysis experts, we shall
present a few examples belonging to the class of groups under consideration.

\begin{itemize}
\item (\textbf{Affine group}) Let $G$ be the \textbf{ax+b} group with Lie
algebra $\mathfrak{g}$ spanned by $X_{1},A_{1}$ such that
\[
\left[  A_{1},X_{1}\right]  =X_{1}.
\]
Given a linear functional $\lambda=\lambda_{1}X_{1}^{\ast}$ such that
$\lambda_{1}$ is a non-zero real number, $\pi_{\lambda}$ is realized as acting
on the Hilbert space $L^{2}\left(  M,d\mu_{M}\right)  $ as follows. For a
square-integrable function $\mathbf{f}$ with respect to the fixed Haar measure
$d\mu_{M},$
\[
\left[  \pi_{\lambda}\left(  \exp\left(  xX_{1}\right)  \exp\left(
tA_{1}\right)  \right)  \mathbf{f}\right]  \left(  \exp\left(  aA_{1}\right)
\right)  =e^{2\pi ixe^{-a}\lambda_{1}}\mathbf{f}\left(  \exp\left(  \left(
a-t\right)  A_{1}\right)  \right)  .
\]
Since $G$ is isomorphic to $\mathbb{R}\rtimes e^{\mathbb{R}}$ with
multiplication law
\[
\left(  x,e^{t}\right)  \left(  y,e^{s}\right)  =\left(  x+e^{t}%
y,e^{t+s}\right)  ,
\]
we may in fact remodel the representation $\pi_{\lambda}$ as acting on
$L^{2}\left(  \left(  0,\infty\right)  ,\frac{dh}{h}\right)  $ such that
\[
\left[  \pi_{\lambda}\left(  \exp\left(  xX_{1}\right)  \exp\left(
tA_{1}\right)  \right)  \mathbf{f}\right]  \left(  h\right)  =e^{\frac{2\pi
ix\lambda_{1}}{h}}\mathbf{f}\left(  \frac{h}{e^{t}}\right)  .
\]

\item (\textbf{Toeplitz shearlet groups}) Let us consider the Lie algebra
spanned by
\[
X_{1},X_{2},\cdots,X_{n_{1}},A_{1},\cdots,A_{n_{2}-1},A_{n_{2}}%
\]
where $n_{2}=n_{1}.$ The vector space generated by $\mathfrak{B}%
_{\mathfrak{p}}=\left(  X_{1},X_{2},\cdots,X_{n_{1}}\right)  $ is a
commutative ideal, the vector space generated by $\left(  A_{1},\cdots
,A_{n_{2}-1},A_{n_{2}}\right)  $ is commutative and the matrix representation
of $\mathrm{ad}\left(  \sum_{k=1}^{n_{2}}t_{k}A_{k}\right)  $ restricted to
$\mathfrak{p}$ with respect to the ordered basis $\mathfrak{B}_{\mathfrak{p}}$
is
\begin{equation}
N\left(  t\right)  =\left[  \left.  \mathrm{ad}\left(  \sum_{k=1}^{n_{2}}%
t_{k}A_{k}\right)  \right\vert _{\mathfrak{p}}\right]  _{\mathfrak{B}%
_{\mathfrak{p}}}=\left[
\begin{array}
[c]{cccccc}%
t_{n_{2}} & t_{1} & t_{2} & \cdots & \cdots & t_{n_{2}-1}\\
0 & t_{n_{2}} & t_{1} & t_{2} & \cdots & \vdots\\
\vdots & 0 & t_{n_{2}} & t_{1} & \ddots & \\
0 & \ddots & \ddots & t_{n_{2}} & \ddots & t_{2}\\
0 & \ddots & 0 & 0 & \ddots & t_{1}\\
0 & \cdots & 0 & 0 & 0 & t_{n_{2}}%
\end{array}
\right]  . \label{At}%
\end{equation}
Given a linear functional $\lambda=\sum_{k=1}^{n_{1}}\lambda_{k}X_{k}^{\ast}$
such that $\lambda_{1}\neq0,$ $\pi_{\lambda}$ is an irreducible representation
of $G$ modeled as follows. For a fixed Haar measure \ $d\mu$ on $M,$
$\pi_{\lambda}$ acts on $L^{2}\left(  M,d\mu\right)  $ as follows
\begin{equation}
\left[  \pi_{\lambda}\left(  \exp\left(  \sum_{k=1}^{n_{2}}x_{k}X_{k}\right)
\exp\left(  \sum_{k=1}^{n_{2}}t_{k}A_{k}\right)  \right)  \mathbf{f}\right]
\left(  \mathbf{e}\left(  a\right)  \right)  =e^{2\pi i\left\langle
\exp\left(  N\left(  -a\right)  ^{\ast}\right)  \lambda,x\right\rangle
}\mathbf{f}\left(  \mathbf{e}\left(  t\right)  ^{-1}\mathbf{e}\left(
a\right)  \right)  .
\end{equation}

\item (\textbf{Heisenberg groups and generalizations}) Let $G=P\rtimes M$ be a
\textbf{step-two (}that is,\textbf{ }$\left[  \mathfrak{g,g}\right]  $ is
non-trivial and is contained in the center of $\mathfrak{g}$) nilpotent Lie
group with Lie algebra spanned by
\[
\left\{  \underset{\mathfrak{z}\left(  \mathfrak{g}\right)  }{\underbrace
{X_{1},\cdots,X_{r}}},X_{r+1},\cdots,X_{r+n_{2}},A_{1},\cdots,A_{n_{2}%
}\right\}
\]
such that
\[
P=\exp\left(  \sum_{k=1}^{r+n_{2}}%
\mathbb{R}
X_{k}\right)  ,M=\exp\left(  \sum_{k=1}^{n_{2}}%
\mathbb{R}
A_{k}\right)
\]
are commutative, and $\exp\mathfrak{z}\left(  \mathfrak{g}\right)  $ is the
center of $G.$ Moreover, let $\lambda\in\mathfrak{g}^{\ast}$ such that
\[
\det B\left(  \lambda\right)  =\det\left[
\begin{array}
[c]{ccc}%
\left\langle \lambda,\left[  A_{1},X_{r+1}\right]  \right\rangle  & \cdots &
\left\langle \lambda,\left[  A_{1},X_{r+n_{2}}\right]  \right\rangle \\
\vdots & \ddots & \vdots\\
\left\langle \lambda,\left[  A_{n_{2}},X_{r+n_{2}}\right]  \right\rangle  &
\cdots & \left\langle \lambda,\left[  A_{n_{2}},X_{r+n_{2}}\right]
\right\rangle
\end{array}
\right]  \neq0.
\]
Then $\pi_{\lambda}$ is irreducible and is realized as acting on the Hilbert
space $L^{2}\left(
\mathbb{R}
^{n_{2}}\right)  $ as follows
\[
\left[  \pi_{\lambda}\left(  \exp\left(  \sum_{k=1}^{r+n_{2}}x_{k}%
X_{k}\right)  \exp\left(  \sum_{k=1}^{n_{2}}t_{k}A_{k}\right)  \right)
\mathbf{f}\right]  \left(  a\right)  =e^{2\pi i\left\langle e^{-\left(
\left.  \sum_{k=1}^{n_{2}}a_{k}\mathrm{ad}\left(  A_{k}\right)  \right\vert
\mathfrak{p}\right)  ^{\ast}}\lambda,x\right\rangle }\mathbf{f}\left(
a-t\right)  .
\]

\item (\textbf{Higher order time-frequency groups} \cite{metabelian}) Let
$G=P\rtimes M$ be a nilpotent Lie group with Lie algebra spanned by
$X_{1},\cdots,X_{n_{2}+1},A_{1},\cdots,A_{n_{2}}$ such that $P$ and $M$ are
commutative closed subgroups and
\[
N\left(  t\right)  =\left[  \left.  \sum_{k=1}^{n_{2}}t_{k}\mathrm{ad}\left(
A_{k}\right)  \right\vert _{\mathfrak{p}}\right]  _{\left(  X_{1}%
,\cdots,X_{n_{2}+1}\right)  }=\left[
\begin{array}
[c]{ccccc}%
0 & t_{1} & t_{2} & \cdots & t_{n_{2}}\\
& 0 & t_{1} & \ddots & \vdots\\
&  & \ddots & \ddots & t_{2}\\
&  &  & 0 & t_{1}\\
&  &  &  & 0
\end{array}
\right]  .
\]
Then $%
\mathbb{R}
X_{1}=\mathfrak{z}\left(  \mathfrak{g}\right)  $ is the central ideal of the
Lie algebra $\mathfrak{g}.$ Let $\lambda=\lambda_{1}X_{1}^{\ast}%
\in\mathfrak{p}^{\ast}$ be a linear functional satisfying $\lambda_{1}\neq0.$
The corresponding irreducible unitary representation $\pi_{\lambda}$ is
realized as acting on the Hilbert space $L^{2}\left(
\mathbb{R}
^{n_{2}}\right)  $ as follows. Given $a=\left(  a_{1},\cdots,a_{n_{1}}\right)
\in%
\mathbb{R}
^{n_{1}}$ and $t=\left(  t_{1},\cdots,t_{n_{1}}\right)  \in%
\mathbb{R}
^{n_{2}},$ we have%
\[
\left[  \pi_{\lambda}\left(  \exp\left(  \sum_{k=1}^{n_{2}+1}x_{k}%
X_{k}\right)  \exp\left(  \sum_{k=1}^{n_{2}}t_{k}A_{k}\right)  \right)
\mathbf{f}\right]  \left(  a\right)  =e^{2\pi i\left(  \left\langle N\left(
-a\right)  ^{T}\lambda,x\right\rangle \right)  }\mathbf{f}\left(  a-t\right)
.
\]

\item (\textbf{Solvable extensions of} $\mathbb{R}^{n_{1}})$ Let $A$ be a
matrix of order $n_{1}$ in its Jordan canonical form, and let $G=\mathbb{R}%
^{n_{1}}\rtimes\exp\left(
\mathbb{R}
A\right)  $ with multiplication law
\[
\left(  v,t\right)  \left(  w,s\right)  =\left(  v+\exp\left(  tA\right)
w,t+s\right)  .
\]
Next, let $\pi_{\lambda}$ be an irreducible representation of $G$ acting on
$L^{2}\left(
\mathbb{R}
\right)  $ such that
\[
\left[  \pi_{\lambda}\left(  v,s\right)  \mathbf{f}\right]  \left(  t\right)
=e^{2\pi i\left(  \left\langle \exp\left(  -tA\right)  ^{T}\lambda
,v\right\rangle \right)  }\mathbf{f}\left(  t-s\right)
\]
where $\lambda\in\left(  \mathbb{R}^{n_{1}}\right)  ^{\ast}$ such that
\[
\left\{  \exp\left(  tA\right)  ^{T}\lambda=\lambda:t\in%
\mathbb{R}
\right\}
\]
is trivial. Notice that $\pi_{\lambda}$ is neither generally a
square-integrable representation nor an integrable representation of $G$.
However, this class of representations fits within the scope of our method.
This example is naturally generalized as follows. Let $\mathfrak{m}=\sum
_{k=1}^{n_{2}}\mathbb{R}A_{k}$ be a solvable Lie algebra of upper-triangular
matrices of order $n_{1}$ consisting of matrices whose spectrum is contained
in $\mathbb{R}$ (see \cite{vanishing,coorbit,Generalized Calderon}.) Next, let
$G=\mathbb{R}^{n_{1}}\rtimes\exp\left(  \mathfrak{m}\right)  $ be the
semidirect product group with multiplication law given by
\[
\left(  v,\exp\left(  A\right)  \right)  \left(  v^{\prime},\exp\left(
A^{\prime}\right)  \right)  =\left(  v+\exp\left(  A\right)  v^{\prime}%
,\exp\left(  A\right)  \exp\left(  A^{\prime}\right)  \right)  .
\]
$G$ is a completely solvable Lie group with Lie algebra $%
\mathbb{R}
^{n_{1}}\oplus\mathfrak{m.}$ Moreover, let $\chi_{\lambda}$ be a character of
$%
\mathbb{R}
^{n_{1}}$ defined by
\[
\chi_{\lambda}\left(  v\right)  =e^{2\pi i\left\langle \lambda,v\right\rangle
}%
\]
where $\lambda\in\left(
\mathbb{R}
^{n_{1}}\right)  ^{\ast}.$ Furthermore, we assume that
\[
G_{\lambda}=\left\{  \exp A\in\exp\mathfrak{m}:\exp\left(  A\right)
^{T}\lambda=\lambda\right\}
\]
is the trivial subgroup of $\exp\mathfrak{m.}$ Next, the representation
$\pi_{\lambda}=\mathrm{ind}_{%
\mathbb{R}
^{n_{1}}}^{G}\left(  \chi_{\lambda}\right)  $ is an irreducible representation
of $G$ acting on the Hilbert space $L^{2}\left(  \exp\left(  \mathfrak{m}%
\right)  ,d\mu_{\exp\left(  \mathfrak{m}\right)  }\right)  $ as follows%
\begin{equation}
\left[  \pi_{\lambda}\left(  \mathbf{x}\right)  \mathbf{f}\right]  \left(
\exp A\right)  =\left\{
\begin{array}
[c]{c}%
\mathbf{f}\left(  \exp\left(  -A^{\prime}\right)  \exp\left(  A\right)
\right)  \text{ if }\mathbf{x}=\left(  0,\exp A^{\prime}\right)  \\
e^{2\pi i\left\langle \exp\left(  -A\right)  ^{T}\lambda,v\right\rangle
}\mathbf{f}\left(  \exp\left(  A\right)  \right)  \text{ if }\mathbf{x}%
=\left(  v,id\right)
\end{array}
\right.  .
\end{equation}

\item (\textbf{Extensions of non-commutative nilpotent Lie groups}) Let
$G=P\rtimes M$ be a completely solvable Lie group where $P$ is a
non-commuative nilpotent Lie normal subgroup and $M$ is isomorphic to a
subgroup of the automorphism group of $P$ \cite{monomial,wavelets on
nilpotent}. Moreover, let us suppose that there exists a character
$\chi_{\lambda}$ of $P$ such that the stabilizer of the coadjoint action of
$M$ on the linear functional $\lambda$ is trivial. Then $\pi_{\lambda
}=\mathrm{ind}_{P}^{P\rtimes M}\left(  \chi_{\lambda}\right)  $ is irreducible
and given $\mathbf{f}\in L^{2}\left(  M,d\mu_{M}\right)  $,%
\[
\left[  \pi_{\lambda}\left(  p,n\right)  \mathbf{f}\right]  \left(  m\right)
=e^{2\pi i\left\langle Ad\left(  m^{-1}\right)  ^{\ast}\lambda,\log\left(
p\right)  \right\rangle }\mathbf{f}\left(  n^{-1}m\right)
\]
as described in (\ref{irre})
\end{itemize}

\subsection{Short overview of the main results}

The main objective of the present paper is to establish the following. Let $G$
and $\pi_{\lambda}$ be as defined in Condition \ref{cond}. There exist a
discrete subset $\Gamma$ of $G$ and a function $\mathbf{f}\in L^{2}\left(
M,d\mu_{M}\right)  $ such that $\left\{  \pi_{\lambda}\left(  \gamma\right)
\mathbf{f}:\gamma\in\Gamma\right\}  $ is a frame for $L^{2}\left(  M,d\mu
_{M}\right)  .$ Moreover, the function $\mathbf{f}$ can be chosen to be
infinitely smooth and compactly supported on $M.$ In contrast to other
discretization schemes such as the coorbit theory, we insist that \textbf{no
assumption} is being made about the integrability of $\pi_{\lambda}.$

\subsubsection{A unified procedure for the construction of tight frames}

In this subsection, we shall present a scheme which is systematically
exploited to construct pairs $\left(  \Gamma,\mathbf{f}\right)  $ such that
$\left\{  \pi_{\lambda}\left(  \gamma\right)  \mathbf{f}:\gamma\in
\Gamma\right\}  $ is a tight frame for $L^{2}\left(  M,d\mu_{M}\right)  .$ Our
procedure is outlined as follows.

\begin{enumerate}
\item Fix a linear functional $\lambda$ with corresponding unitary irreducible
representation
\[
\pi_{\lambda}=\mathrm{ind}_{P}^{P\rtimes M}\left(  \chi_{\lambda}\right)  .
\]
Given $a=\left(  a_{1},a_{2},\cdots,a_{n_{2}}\right)  \in%
\mathbb{R}
^{m},$ let
\[
A\left(  a_{1},a_{2},\cdots,a_{n_{2}}\right)  =A\left(  a\right)  =%
{\displaystyle\sum\limits_{k=1}^{n_{2}}}
a_{k}A_{k}=A.
\]
Next, let $\theta_{\lambda}:\mathfrak{m}\rightarrow\mathfrak{p}^{\ast}$ be a
smooth function defined as follows
\begin{equation}
\theta_{\lambda}\left(
{\displaystyle\sum\limits_{k=1}^{n_{2}}}
a_{k}A_{k}\right)  =\mathbf{C}\left(  a\right)  ^{\ast}\lambda=\left(  \left.
\left(  e^{-\mathrm{ad}\left(  a_{n_{2}}A_{n_{2}}\right)  }\cdots
e^{-\mathrm{ad}\left(  a_{2}A_{2}\right)  }e^{-\mathrm{ad}\left(  a_{1}%
A_{1}\right)  }\right)  \right\vert _{\mathfrak{p}}\right)  ^{\ast}\lambda.
\end{equation}
Then $\theta_{\lambda}$ represents the coadjoint action of the conormal
subgroup $M$ on the linear functional $\lambda.$ Moreover, under the
assumptions stated above, $\theta_{\lambda}$ defines an immersion of
$\mathfrak{m}$ into $\mathfrak{p}^{\ast}$ (see Lemma \ref{immersion}). In
other words, the differential of $\theta_{\lambda}$ is injective at each point
$A,$ and $\theta_{\lambda}$ behaves locally like an injective function. Let
$D_{\theta_{\lambda}}$ be the differential of $\theta_{\lambda}$ at the zero
element in $\mathfrak{m}$. To be more explicity, if
\begin{equation}
\mathrm{Jac}_{\theta_{\lambda}}\left(  a\right)  =\left[
\begin{array}
[c]{ccc}%
\dfrac{\partial\left[  \left(  \mathbf{C}\left(  a\right)  ^{\ast}%
\lambda\right)  _{1}\right]  }{\partial a_{1}} & \cdots & \dfrac
{\partial\left[  \left(  \mathbf{C}\left(  a\right)  ^{\ast}\lambda\right)
_{1}\right]  }{\partial a_{n_{2}}}\\
\vdots & \ddots & \vdots\\
\dfrac{\partial\left[  \left(  \mathbf{C}\left(  a\right)  ^{\ast}%
\lambda\right)  _{n_{1}}\right]  }{\partial a_{1}} & \cdots & \dfrac
{\partial\left[  \left(  \mathbf{C}\left(  a\right)  ^{\ast}\lambda\right)
_{n_{1}}\right]  }{\partial a_{n_{2}}}%
\end{array}
\right]  \label{Jac}%
\end{equation}
then
\begin{equation}
D_{\theta_{\lambda}}=\mathrm{Jac}_{\theta_{\lambda}}\left(  0\right)
.\label{Jaczero}%
\end{equation}

\item Next, let $D_{\theta_{\lambda}}\left(  j_{1},\cdots,j_{n_{2}}\right)  $
be the submatrix of $D_{\theta_{\lambda}}$ obtained by retaining the
$j_{1}^{th}$-row$,\cdots,j_{n_{2}}^{th}$-row of the matrix $\left[
D_{\theta_{\lambda}}\right]  _{\mathfrak{B}_{\mathfrak{p}}}$. Put
\[
\mathcal{T}=\left\{  \boldsymbol{I}:\boldsymbol{I}=\left(  j_{1}%
,\cdots,j_{n_{2}}\right)  \text{ and }1\leq j_{1}<\cdots<j_{n_{2}}\leq
n_{1}\right\}
\]
and define
\[
\mathcal{A}=\left\{  D_{\theta_{\lambda}}\left(  \boldsymbol{I}\right)
:\boldsymbol{I}\in\mathcal{T}\text{ and }\mathrm{\det}\left(  D_{\theta
_{\lambda}}\left(  \boldsymbol{I}\right)  \right)  \neq0\right\}  .
\]
Note that if $\left[  D_{\theta_{\lambda}}\right]  _{\mathfrak{B}%
_{\mathfrak{p}}}$ is an invertible matrix of order $n_{2}$ (this is not
generally the case) then $\mathcal{A}$ is necessarily a singleton. Fix
\begin{equation}
\boldsymbol{J}\left(  \lambda\right)  =\left(  j_{1},\cdots,j_{n_{2}}\right)
\in\mathcal{T}\label{setJ}%
\end{equation}
such that $D_{\theta_{\lambda}}\left(  \boldsymbol{J}\left(  \lambda\right)
\right)  \in\mathcal{A}$ and define $\mathbf{P}_{\boldsymbol{J}\left(
\lambda\right)  }:\mathfrak{p}\rightarrow\mathfrak{p}$ such that
\begin{equation}
\mathbf{P}_{\boldsymbol{J}\left(  \lambda\right)  }\left(  X_{k}\right)
=\left\{
\begin{array}
[c]{c}%
X_{k}\text{ if }k\in\boldsymbol{J}\left(  \lambda\right)  \\
0\text{ if }k\notin\boldsymbol{J}\left(  \lambda\right)
\end{array}
\right.  .\label{DJ}%
\end{equation}
Clearly, $\mathbf{P}_{\boldsymbol{J}\left(  \lambda\right)  }$ is a linear
map, and the matrix representation of the linear map $\mathbf{P}%
_{\boldsymbol{J}\left(  \lambda\right)  }$ is a diagonal matrix whose spectrum
is contained in the discrete set $\left\{  0,1\right\}  .$ Evidently,
$\mathbf{P}_{\boldsymbol{J}\left(  \lambda\right)  }$ defines an orthogonal
projection of rank $n_{2}$ on the Lie algebra $\mathfrak{p.}$ Furthermore, if
$X$ belongs to the range of $\mathbf{P}_{\boldsymbol{J}\left(  \lambda\right)
}$, we obtain
\begin{equation}
\left[  \pi_{\lambda}\left(  \exp X\right)  \mathbf{f}\right]  \left(
\exp\left(  a_{1}A_{1}\right)  \cdots\exp\left(  a_{n_{2}}A_{n_{2}}\right)
\right)  =e^{2\pi i\mathbf{Q}\left(  a,\lambda,X\right)  }\mathbf{f}\left(
\exp\left(  a_{1}A_{1}\right)  \cdots\exp\left(  a_{n_{2}}A_{n_{2}}\right)
\right)
\end{equation}
where
\[
\mathbf{Q}\left(  a,\lambda,X\right)  =\left\langle \mathbf{P}_{\boldsymbol{J}%
\left(  \lambda\right)  }^{\ast}\left(  \mathbf{C}\left(  a_{1},a_{2}%
,\cdots,a_{n_{2}}\right)  \right)  ^{\ast}\lambda,X\right\rangle .
\]

\item Define $\beta_{\boldsymbol{J}\left(  \lambda\right)  }:\mathfrak{m}%
\rightarrow\mathbf{P}_{\boldsymbol{J}\left(  \lambda\right)  }^{\ast}\left(
\mathfrak{p}^{\ast}\right)  $ such that
\begin{equation}
\beta_{\boldsymbol{J}\left(  \lambda\right)  }\left(
{\displaystyle\sum\limits_{k=1}^{n_{2}}}
a_{k}A_{k}\right)  =\mathbf{P}_{\boldsymbol{J}\left(  \lambda\right)  }^{\ast
}\left(  \mathbf{C}\left(  a_{1},a_{2},\cdots,a_{n_{2}}\right)  \right)
^{\ast}\lambda.
\end{equation}
According to Lemma \ref{U}, $\beta_{\boldsymbol{J}\left(  \lambda\right)  }$
is a local diffeomorphism at the zero element in $\mathfrak{m}$. By the
Inverse Function Theorem (\cite{Lee}, Theorem $5.11$) there exists a connected
open subset $\mathcal{O}$ around the zero element of $\mathfrak{m}$ such that
the restriction of $\beta_{\boldsymbol{J}\left(  \lambda\right)  }$ to
$\mathcal{O}$ is a diffeomorphism. We shall coin the collection of maps
\begin{equation}
\mathfrak{Data}_{\left(  \pi_{\lambda},G\right)  }=\left\{  \beta
_{\boldsymbol{J}\left(  \lambda\right)  }:D_{\theta_{\lambda}}\left(
\boldsymbol{J}\left(  \lambda\right)  \right)  \in\mathcal{A}\right\}  ,
\end{equation}
the \textbf{orbital data corresponding to the linear functional} $\lambda.$
Next, let
\begin{equation}
\mathbf{L}=\left\{  s\in\left(  0,\infty\right)  :\sum_{k=1}^{n_{2}}\left[
-\frac{s}{2},\frac{s}{2}\right]  A_{k}\subset\mathcal{O}\right\}  .
\end{equation}
Clearly, $\mathbf{L}$ is a non-empty subset of $\left(  0,\infty\right)  .$
Fix $\epsilon\in\mathbf{L}$, define a relatively compact subset $\Omega
_{\epsilon}\subset M$ and a discrete subset $\Gamma_{M\text{ }}^{\epsilon}$ of
$M$ as follows.
\begin{equation}
\Omega_{\epsilon}=\exp\left(  \left[  -\frac{\epsilon}{2},\frac{\epsilon}%
{2}\right)  A_{1}\right)  \cdots\exp\left(  \left[  -\frac{\epsilon}{2}%
,\frac{\epsilon}{2}\right)  A_{n_{2}}\right)  ,
\end{equation}
and
\[
\Gamma_{M\text{ }}^{\epsilon}=\exp\left(  \epsilon%
\mathbb{Z}
A_{1}\right)  \cdots\exp\left(  \epsilon%
\mathbb{Z}
A_{n_{2}}\right)
\]
respectively. Appealing to Proposition \ref{tilesM}, the collection
\[
\left\{  \gamma^{-1}\Omega_{\epsilon}:\gamma\in\Gamma_{M\text{ }}^{\epsilon
}\right\}
\]
is a measurable partition of $M.$ Next, define
\begin{equation}
\mathcal{O}_{\epsilon}=\sum_{k=1}^{n_{2}}\left(  -\frac{\epsilon}{2}%
,\frac{\epsilon}{2}\right)  A_{k}\text{ and }\digamma_{\epsilon}=\sum
_{k=1}^{n_{2}}\left[  -\frac{\epsilon}{2},\frac{\epsilon}{2}\right)
A_{k}\text{.}%
\end{equation}
Note that $\mathcal{O}_{\epsilon}$ is just the interior of the set
$\digamma_{\epsilon},$ and the map $\beta_{\boldsymbol{J}\left(
\lambda\right)  }$ is continuous on $\mathcal{O}$ and uniformly continuous on
$\overline{\mathcal{O}_{\epsilon}}$ (the topological closure of $\mathcal{O}%
_{\epsilon}.$) Moreover, $\beta_{\boldsymbol{J}\left(  \lambda\right)
}\left(  \overline{\mathcal{O}_{\epsilon}}\right)  $ is a compact subset of
\[
\mathbf{P}_{\boldsymbol{J}\left(  \lambda\right)  }^{\ast}\left(
\mathfrak{p}^{\ast}\right)  =\sum_{j\in\boldsymbol{J}\left(  \lambda\right)  }%
\mathbb{R}
X_{j}^{\ast}.
\]
Thus, $\beta_{\boldsymbol{J}\left(  \lambda\right)  }\left(  \overline
{\mathcal{O}_{\epsilon}}\right)  $ is a bounded subset, has positive Lebesgue
measure on $\mathbf{P}_{\boldsymbol{J}\left(  \lambda\right)  }^{\ast}\left(
\mathfrak{p}^{\ast}\right)  ,$ and the set $\beta_{\boldsymbol{J}\left(
\lambda\right)  }\left(  \digamma_{\epsilon}\right)  $ is contained in a
fundamental domain of a full-rank lattice of the vector space $\mathbf{P}%
_{\boldsymbol{J}\left(  \lambda\right)  }^{\ast}\left(  \mathfrak{p}^{\ast
}\right)  .$

\item Let
\begin{equation}
\mathbf{T}^{\epsilon}=\left\{
\begin{array}
[c]{c}%
\mathcal{T}:\mathcal{T}\text{ is a full-rank lattice of }\mathbf{P}%
_{\boldsymbol{J}\left(  \lambda\right)  }^{\ast}\left(  \mathfrak{p}^{\ast
}\right)  \text{ and }\\
\sum_{\kappa\in\mathcal{T}}\mathbf{1}_{\beta_{\boldsymbol{J}\left(
\lambda\right)  }\left(  \digamma_{\epsilon}\right)  }\left(  \xi
+\kappa\right)  \leq1\text{ for every }\xi\in\mathbf{P}_{\boldsymbol{J}\left(
\lambda\right)  }^{\ast}\left(  \mathfrak{p}^{\ast}\right)
\end{array}
\right\}  .\label{T}%
\end{equation}
In other words, $\mathcal{L}\in\mathbf{T}^{\epsilon}$ if and only if
$\beta_{\boldsymbol{J}\left(  \lambda\right)  }\left(  \digamma_{\epsilon
}\right)  $ is contained in a fundamental domain of $\beta_{\boldsymbol{J}%
\left(  \lambda\right)  }\left(  \digamma_{\epsilon}\right)  .$ It is shown in
Lemma \ref{GammaP} that $\mathbf{T}^{\epsilon}$ is a non-empty set. To be more
specific, Lemma \ref{GammaP} describes explicitly an invertible linear map
\begin{equation}
\mathcal{L}^{\epsilon}:\sum_{j\in\boldsymbol{J}\left(  \lambda\right)  }%
\mathbb{R}
X_{j}\rightarrow\sum_{j\in\boldsymbol{J}\left(  \lambda\right)  }%
\mathbb{R}
X_{j}%
\end{equation}
such that given
\[
\Gamma_{P}^{\epsilon}=\exp\left(  \Lambda_{\epsilon}\right)  \subset P\text{
\ where }\Lambda_{\epsilon}=%
\mathbb{Z}
\text{-span}\left\{  \mathcal{L}^{\epsilon}X_{j}:j\in\boldsymbol{J}\left(
\lambda\right)  \right\}
\]
the following holds true. The set
\[
\beta_{\boldsymbol{J}\left(  \lambda\right)  }\left(  \digamma_{\epsilon
}\right)  \subset\sum_{j\in\boldsymbol{J}\left(  \lambda\right)  }%
\mathbb{R}
X_{j}^{\ast}%
\]
is contained in a Lebesgue measurable fundamental domain of
\[
\Lambda_{\epsilon}^{\star}=%
\mathbb{Z}
\text{-span}\left\{  \left(  \mathcal{L}^{\epsilon}\right)  ^{\top}X_{j}%
^{\ast}:j\in\boldsymbol{J}\left(  \lambda\right)  \right\}  \subset
\mathfrak{p}^{\ast}%
\]
where $\left(  \mathcal{L}^{\epsilon}\right)  ^{\top}$ is the transpose
inverse of $\mathcal{L}^{\epsilon}$.

\item For a positive measurable function $r$ defined on $M,$ we define
$\mathbf{f}_{r,\epsilon}=r\times\mathbf{1}_{\Omega_{\epsilon}}.$ Next, let
\[
\mathbf{e}\left(
{\displaystyle\sum\limits_{k=1}^{n_{2}}}
a_{k}A_{k}\right)  =\exp\left(  a_{1}A_{1}\right)  \cdots\exp\left(  a_{n_{2}%
}A_{n_{2}}\right)  .
\]
Clearly $\mathbf{e}:\mathfrak{m}\rightarrow M$ is a diffeomorphism.
Furthermore, let $\rho$ be the Radon-Nikodym derivative given by%
\begin{equation}
d\mu_{M}\left(  \mathbf{e}\left(
{\displaystyle\sum\limits_{k=1}^{n_{2}}}
a_{k}A_{k}\right)  \right)  =d\mu_{M}\left(  \exp\left(  a_{1}A_{1}\right)
\cdots\exp\left(  a_{n_{2}}A_{n_{2}}\right)  \right)  =\rho\left(  A\right)
dA
\end{equation}
where $d\mu_{M}$ is a left Haar measure on the solvable group $M$ and $dA$ is
the Lebesgue measure on $\mathfrak{m}=\mathbb{R}^{n_{2}}$. The function $\rho$
is an analytic function which is explicitly computed as follows. Let $\nu$ be
a smooth bijection defined on the Lie algebra $\mathfrak{m}$ such that
\begin{equation}
\exp\left(  a_{1}A_{1}\right)  \cdots\exp\left(  a_{n_{2}}A_{n_{2}}\right)
=\exp\left(  \nu\left(
{\displaystyle\sum\limits_{k=1}^{n_{2}}}
a_{k}A_{k}\right)  \right)  =\exp\left(
{\displaystyle\sum\limits_{k=1}^{n_{2}}}
\nu_{k}\left(  a\right)  A_{k}\right)  .
\end{equation}
For example, if $\mathfrak{m}$ is commutative then $\nu$ is the identity map,
and if $\mathfrak{m}$ is a nilpotent algebra, then $\nu$ is a polynomial. In
general, $\exp\left(  a_{1}A_{1}\right)  \cdots\exp\left(  a_{n_{2}}A_{n_{2}%
}\right)  $ is computed by applying the Campbell-Baker-Hausdorff formula
iteratively as follows. Defining
\begin{align*}
X\ast Y &  =\sum_{n>0}\frac{\left(  -1\right)  ^{n+1}}{n}\sum_{p_{i}%
+q_{i}>0,1\leq i\leq n}\frac{\left(  \sum_{i=1}^{n}\left(  p_{i}+q_{i}\right)
\right)  ^{-1}}{p_{1}!q_{1}!\cdots p_{n}!q_{n}!}\\
&  \left(  adX\right)  ^{p_{1}}\left(  adY\right)  ^{q_{1}}\cdots\left(
adX\right)  ^{p_{n}}\left(  adY\right)  ^{q_{n}-1}Y\\
&  =\nu\left(  X+Y\right)  ,
\end{align*}
then
\[
\exp\left(  a_{1}A_{1}\right)  \cdots\exp\left(  a_{n_{2}}A_{n_{2}}\right)
=\exp\left(  a_{1}A_{1}\ast\cdots\ast a_{n_{2}}A_{n_{2}}\right)  .
\]
In other words,
\[
\nu\left(
{\displaystyle\sum\limits_{k=1}^{n_{2}}}
a_{k}A_{k}\right)  =a_{1}A_{1}\ast\cdots\ast a_{n_{2}}A_{n_{2}}%
\]
and for a positive function $\mathbf{f}\in L^{1}\left(  M,d\mu_{M}\right)  ,$%
\begin{align*}
\int_{M}\mathbf{f}\left(  m\right)  d\mu_{M}\left(  m\right)   &
=\int_{\mathfrak{m}}\mathbf{f}\left(  \exp A\right)  d\mu_{M}\left(  \exp
A\right)  \\
&  =\int_{\mathfrak{m}}\mathbf{f}\left(  \exp A\right)  \underset{=w\left(
A\right)  }{\underbrace{\left\vert \det\left(  \frac{id-e^{-\mathrm{ad}\left(
A\right)  }}{\mathrm{ad}\left(  A\right)  }\right)  \right\vert }}dA\\
&  =\int_{\mathfrak{m}}\mathbf{f}\left(  \exp A\right)  w\left(  A\right)
dA\\
&  =\int_{\mathfrak{m}}\mathbf{f}\left(  \exp\nu\left(  A\right)  \right)
\overset{=\rho\left(  A\right)  dA}{\overbrace{w\left(  \nu\left(  A\right)
\right)  d\left(  \nu\left(  A\right)  \right)  }}\\
&  =\int_{\mathfrak{m}}\mathbf{f}\left(  \exp\nu\left(  A\right)  \right)
\rho\left(  A\right)  dA.
\end{align*}

\item Fix $\Lambda_{\epsilon}^{\star}\in\mathbf{T}^{\epsilon}.$ Let
\begin{align*}
\Gamma^{\epsilon}  &  =\left(  \Gamma_{M\text{ }}^{\epsilon}\right)
^{-1}\Gamma_{P}^{\epsilon}=\left(  \exp\left(  \epsilon%
\mathbb{Z}
A_{1}\right)  \cdots\exp\left(  \epsilon%
\mathbb{Z}
A_{n_{2}}\right)  \right)  ^{-1}\exp\left(  \Lambda_{\epsilon}\right) \\
&  =\exp\left(  \epsilon\mathbb{Z}A_{n_{2}}\right)  \cdots\exp\left(
\epsilon\mathbb{Z}A_{1}\right)  \exp\left(  \sum_{j\in\boldsymbol{J}\left(
\lambda\right)  }\mathcal{L}^{\epsilon}%
\mathbb{Z}
X_{j}\right)  .
\end{align*}
Define the system
\begin{equation}
\mathcal{S}\left(  \mathbf{f}_{r,\epsilon},\Gamma^{\epsilon}\right)  =\left\{
\pi_{\lambda}\left(  \kappa\right)  \mathbf{f}_{r,\epsilon}:\kappa\in
\Gamma^{\epsilon}\right\}  \label{system}%
\end{equation}
and let $\Theta_{\lambda}\left(  \xi\right)  $ be the absolute value of the
determinant of the Jacobian of $\left[  \beta_{\boldsymbol{J}\left(
\lambda\right)  }|_{\mathcal{O}}\right]  ^{-1}.$
\end{enumerate}

\begin{theorem}
\label{Main} If $r\left(  \mathbf{e}\left(  {A}\right)  \right)  =\left(
\rho\left(  {A}\right)  \times\Theta_{\lambda}\left(  \beta_{\boldsymbol{J}%
\left(  \lambda\right)  }\left(  {A}\right)  \right)  \right)  ^{-1/2}$then
$\mathbf{f}_{r,\epsilon}$ is square-integrable with respect to the Haar
measure $d\mu_{M}.$ Moreover, the system $\mathcal{S}\left(  \mathbf{f}%
_{r,\epsilon},\Gamma^{\epsilon}\right)  $ is a tight frame for $L^{2}\left(
M,d\mu_{M}\right)  $ with frame bound $\left\vert \det\left(  \mathcal{L}%
^{\epsilon}\right)  \right\vert ^{-1}.$ Consequently, $\mathcal{S}\left(
\left\vert \det\left(  \mathcal{L}^{\epsilon}\right)  \right\vert
^{1/2}\mathbf{f}_{r,\epsilon},\Gamma^{\epsilon}\right)  $ is a Parseval frame
for $L^{2}\left(  M,d\mu_{M}\right)  .$
\end{theorem}

\subsubsection{Constructions of smooth frames of compact supports}

We shall now present an explicit construction of a frame $\mathcal{S}\left(
\mathbf{s},\Gamma\right)  $ such that $\mathbf{s}$ is smooth and compactly
supported. To this end, we proceed as follows.

\begin{enumerate}
\item We fix $\epsilon\in\mathbf{L}$. Next, define $\Omega_{\epsilon}^{\circ}$
to be an open subset of $M$ such that
\begin{equation}
\Omega_{\epsilon}^{\circ}=\exp\left(  \left(  -\frac{\epsilon}{2}%
,\frac{\epsilon}{2}\right)  A_{1}\right)  \cdots\exp\left(  \left(
-\frac{\epsilon}{2},\frac{\epsilon}{2}\right)  A_{n_{2}}\right)  .
\end{equation}
As observed above, the restriction of $\beta_{\boldsymbol{J}\left(
\lambda\right)  }$ to the open set
\begin{equation}
\mathcal{O}_{\epsilon}=\left(  -\frac{\epsilon}{2},\frac{\epsilon}{2}\right)
A_{1}+\cdots+\left(  -\frac{\epsilon}{2},\frac{\epsilon}{2}\right)  A_{n_{2}%
}\subset\mathfrak{m}%
\end{equation}
defines a diffeomorphism between $\mathcal{O}_{\epsilon}$ and $\beta
_{\boldsymbol{J}\left(  \lambda\right)  }\left(  \mathcal{O}_{\epsilon
}\right)  .$

\item We define the map
\[
\Phi_{\boldsymbol{J}\left(  \lambda\right)  }^{_{\epsilon}}:\Omega_{\epsilon
}^{\circ}\rightarrow\beta_{\boldsymbol{J}\left(  \lambda\right)  }\left(
\mathcal{O}_{\epsilon}\right)
\]
such that
\begin{equation}
\Phi_{\boldsymbol{J}\left(  \lambda\right)  }^{_{\epsilon}}\left(
\mathbf{e}\left(  a_{1}A_{1}+\cdots+a_{n_{2}}A_{n_{2}}\right)  \right)
=\beta_{\boldsymbol{J}\left(  \lambda\right)  }\left(  a_{1}A_{1}%
+\cdots+a_{n_{2}}A_{n_{2}}\right)  .
\end{equation}

\item Let $\mathbf{s}\in C_{c}^{\infty}\left(  M\right)  $ such that the
support of $\mathbf{s}$ is a compact subset of $\Omega_{\epsilon}^{\circ}.$
Then the support of the function $\mathbf{s}\circ\left[  \Phi_{\boldsymbol{J}%
\left(  \lambda\right)  }^{_{\epsilon}}\right]  ^{-1}$ is a compact subset
$\Sigma_{\mathbf{s}}$ of $\beta_{\boldsymbol{J}\left(  \lambda\right)
}\left(  \mathcal{O}_{\epsilon}\right)  .$ Put
\[
\Lambda_{\epsilon}=\sum_{j\in\boldsymbol{J}\left(  \lambda\right)  }%
\mathbb{Z}
\mathcal{L}^{\epsilon}X_{j}%
\]
and let
\begin{equation}
\Gamma_{P}^{\epsilon}=\exp\left(  \Lambda_{\epsilon}\right)  =\exp\left(
\sum_{j\in\boldsymbol{J}\left(  \lambda\right)  }%
\mathbb{Z}
\mathcal{L}^{\epsilon}X_{j}\right)
\end{equation}
be a discrete subset of $P$ such that $\Sigma_{\mathbf{s}}$ is contained in a
fundamental domain of a lattice
\begin{equation}
\Lambda_{\epsilon}^{\star}=\sum_{j\in\boldsymbol{J}\left(  \lambda\right)  }%
\mathbb{Z}
\left(  \mathcal{L}^{\epsilon}\right)  ^{\top}X_{j}^{\ast}.
\end{equation}

\item Let $d\xi$ be the canonical Lebesgue measure defined on $\beta
_{\boldsymbol{J}\left(  \lambda\right)  }\left(  \mathcal{O}_{\epsilon
}\right)  .$ Define the Radon-Nikodym derivative $\Psi_{\boldsymbol{J}\left(
\lambda\right)  }^{\epsilon}$ such that
\begin{equation}
\Psi_{\boldsymbol{J}\left(  \lambda\right)  }^{\epsilon}\left(  \xi\right)
d\xi=d\mu_{M}\left(  \left[  \Phi_{\boldsymbol{J}\left(  \lambda\right)
}^{_{\epsilon}}\right]  ^{-1}\left(  \xi\right)  \right)  .
\end{equation}
Then $\Psi_{\boldsymbol{J}\left(  \lambda\right)  }^{\epsilon}\left(
\xi\right)  $ is a positive smooth function defined on $\beta_{\boldsymbol{J}%
\left(  \lambda\right)  }\left(  \mathcal{O}_{\epsilon}\right)  $. Moreover,
it is proved in Proposition \ref{integrable} that the function
\[
m\mapsto\left(  \Psi_{\boldsymbol{J}\left(  \lambda\right)  }^{\epsilon
}\left(  \Phi_{\boldsymbol{J}\left(  \lambda\right)  }^{_{\epsilon}}\left(
m\right)  \right)  \right)  ^{-1}%
\]
is integrable with respect to the Haar measure on $M$ on any compact subset of
$\Omega_{\epsilon}^{\circ}.$ Furthermore, the auxiliary function $\Upsilon$
given by
\begin{equation}
\Upsilon\left(  m\right)  =\sqrt{\Psi_{\boldsymbol{J}\left(  \lambda\right)
}^{\epsilon}\left(  \Phi_{\boldsymbol{J}\left(  \lambda\right)  }^{_{\epsilon
}}\left(  m\right)  \right)  \left\vert \mathrm{\det}\text{ }\left(
\mathcal{L}^{\epsilon}\right)  ^{\top}\right\vert }\mathbf{s}\left(  m\right)
\end{equation}
is a smooth function which is compactly supported on $M.$ Consequently, there
exists a sufficiently dense discrete subset%
\begin{equation}
\Gamma_{M\text{ }}^{\epsilon}\left(  \mathbf{s}\right)  =\Gamma_{M\text{ }%
}^{\epsilon}\subset M\label{GammaM}%
\end{equation}
depending on the function $\mathbf{s}$ such that
\begin{equation}
\inf\left\{
{\displaystyle\sum\limits_{\gamma\in\Gamma_{M\text{ }}^{\epsilon}}}
\left\vert \Upsilon\left(  \gamma m\right)  \right\vert ^{2}:m\in M\right\}
=A_{\mathbf{s,}\Gamma_{M\text{ }}^{\epsilon},\Lambda^{\star}}>0
\end{equation}
and
\begin{equation}
\sup\left\{
{\displaystyle\sum\limits_{\gamma\in\Gamma_{M\text{ }}^{\epsilon}}}
\left\vert \Upsilon\left(  \gamma m\right)  \right\vert ^{2}:m\in M\right\}
=B_{\mathbf{s,}\Gamma_{M\text{ }}^{\epsilon},\Lambda^{\star}}<\infty.
\end{equation}
Finally, let
\[
\Gamma^{\epsilon}=\left(  \Gamma_{M\text{ }}^{\epsilon}\right)  ^{-1}%
\Gamma_{P}^{\epsilon}.
\]

\end{enumerate}

\begin{theorem}
\label{smooth frames}Let $\mathbf{s}$ and $\Gamma^{\epsilon}$ be as defined
above. Then the system $\mathcal{S}\left(  \mathbf{s},\Gamma^{\epsilon
}\right)  $ is a frame for $L^{2}\left(  M,\mu_{M}\right)  $ with optimal
lower and upper frame bounds $A_{\mathbf{s,}\Gamma_{M\text{ }}^{\epsilon
},\Lambda^{\star}}$ and $B_{\mathbf{s,}\Gamma_{M\text{ }}^{\epsilon}%
,\Lambda^{\star}}$ respectively.
\end{theorem}

\section{A toy example}

In order to set the stage for the generalization to come, we shall present a
toy example which illustrates the core ideas of our scheme. Let $G=P\rtimes M$
be a simply connected, connected completely solvable Lie group with Lie
algebra spanned by $\left(  X_{1},X_{2},A_{1}\right)  $ such that
$P=\exp\left(
\mathbb{R}
X_{1}+%
\mathbb{R}
X_{2}\right)  $ and $M=\exp%
\mathbb{R}
A_{1}$ are commutative closed subgroups of $G$ and
\[
\left[  A_{1},X_{2}\right]  =X_{1},\left[  A_{1},X_{2}\right]  =X_{2}+X_{1}.
\]
Thus,
\begin{equation}
\left[  ad\left(  aA_{1}\right)  \right]  _{\left(  X_{1},X_{2}\right)
}=\left[
\begin{array}
[c]{cc}%
a & a\\
0 & a
\end{array}
\right]  \text{ and }e^{\left[  ad\left(  aA_{1}\right)  \right]  _{\left(
X_{1},X_{2}\right)  }}=\left[
\begin{array}
[c]{cc}%
e^{a} & ae^{a}\\
0 & e^{a}%
\end{array}
\right]  .
\end{equation}
Fixing $\lambda=X_{1}^{\ast}$, the orbital data corresponding to $\lambda$ is
equal to
\[
\mathfrak{Data}_{\left(  \pi_{\lambda},G\right)  }=\left\{  a\mapsto
\beta_{\left(  1\right)  }\left(  a\right)  =e^{-a},a\mapsto\beta_{\left(
2\right)  }\left(  a\right)  =-ae^{-a}\right\}  .
\]
In order to simplify our presentation, we identify $G$ with the semi-direct
product group $%
\mathbb{R}
^{2}\rtimes%
\mathbb{R}
$ with multiplication law given by
\[
\left(  x_{1},x_{2},a\right)  \left(  y_{1},y_{2},b\right)  =\left(
x_{1}+e^{a}\left(  y_{1}+ay_{2}\right)  ,x_{2}+e^{a}y_{2},a+b\right)  .
\]
Indeed, the identification above is valid since the mapping
\[
\left(  x_{1},x_{2},a\right)  \mapsto\exp\left(  x_{1}X_{1}+x_{2}X_{2}\right)
\exp\left(  aA_{1}\right)
\]
is a Lie group isomorphism. We realize the representation $\pi_{\lambda}$ as
acting in $L^{2}\left(
\mathbb{R}
\right)  $ as follows
\begin{equation}
\left[  \pi_{\lambda}\left(  x_{1},x_{2},a\right)  \mathbf{f}\right]  \left(
t\right)  =\exp\left(  2\pi i\left\langle \left[
\begin{array}
[c]{c}%
e^{-t}\\
-te^{-t}%
\end{array}
\right]  ,\left[
\begin{array}
[c]{c}%
x_{1}\\
x_{2}%
\end{array}
\right]  \right\rangle \right)  \mathbf{f}\left(  t-a\right)  .
\end{equation}
Referring back to the orbital data described previously, we fix
$\boldsymbol{J}\left(  \lambda\right)  =\left(  2\right)  $ and we define
\[
\Omega=\left[  -\frac{1}{2},\frac{1}{2}\right)  .
\]
Then%
\[
\left[  \pi_{\lambda}\left(  0,x_{2},a\right)  \mathbf{f}\right]  \left(
t\right)  =\exp\left(  2\pi i\beta_{\left(  2\right)  }\left(  t\right)
x_{2}\right)  \mathbf{f}\left(  t-a\right)  \text{ and }\beta_{\left(
2\right)  }\left(  t\right)  =-te^{-t}.
\]
The reader might find the following remarks instructive. First, the function
$\beta_{\left(  2\right)  }\left(  t\right)  =-te^{-t}$ does not define a
global diffeomorphism between $\mathbb{R}$ and its range. Indeed, from the
graph below, it is clear that the map $t\mapsto\beta_{\left(  2\right)
}\left(  t\right)  $ is not injective%

\[
\fbox{\text{\includegraphics[scale=0.5]{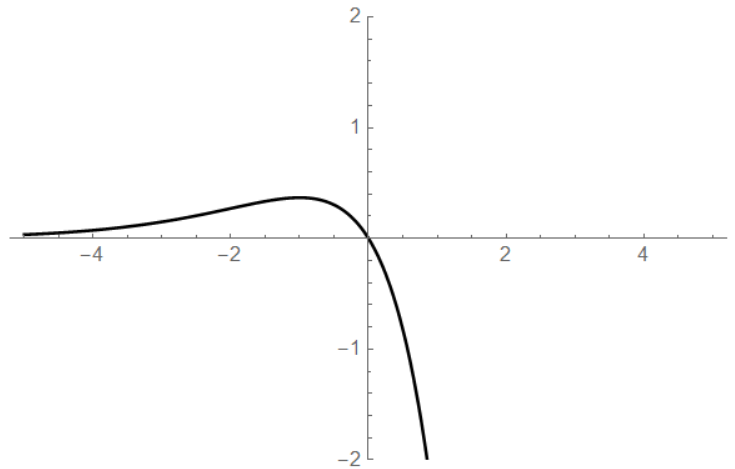}}}%
\]
However, since the derivative of $\beta_{\left(  2\right)  }\left(  t\right)
$ does not vanish at zero, there exists an open set $\mathcal{O}$ around zero
such that the restriction of $\beta_{\left(  2\right)  }$ to $\mathcal{O}$
defines a diffeomorphism between $\mathcal{O}$ and $\beta_{\left(  2\right)
}\left(  \mathcal{O}\right)  .$ Indeed,
\[
\beta_{\left(  2\right)  }|_{\left(  -\frac{1}{2},\frac{1}{2}\right)
}:\left(  -\frac{1}{2},\frac{1}{2}\right)  \rightarrow\beta_{\left(  2\right)
}\left(  \left(  -\frac{1}{2},\frac{1}{2}\right)  \right)
\]
is a diffeomorphism%
\[
\fbox{\text{\includegraphics[scale=0.5]{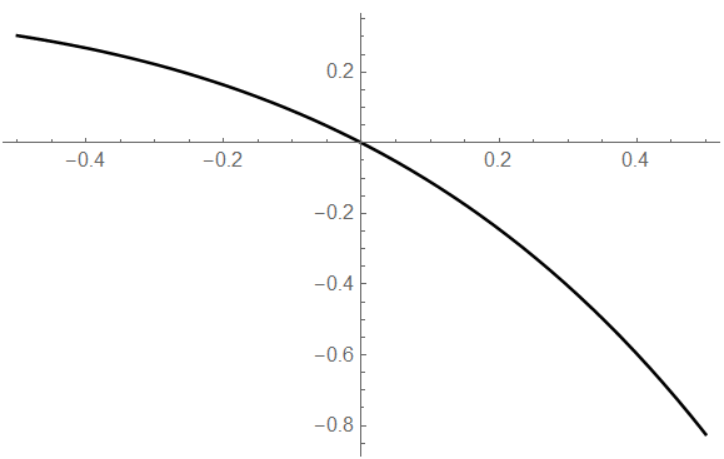}}}%
\]
Let $\mathbf{A}=\left(  -\frac{1}{2},\frac{1}{2}\right)  $ be the interior of
$\Omega.$ Secondly, let $\Theta_{\lambda}$ be defined as follows%
\begin{equation}
\Theta_{\lambda}\left(  \xi\right)  =\left\vert \left(  \frac{d\left[
\beta_{\left(  2\right)  }|_{\mathbf{A}}\right]  ^{-1}\left(  \xi\right)
}{d\xi}\right)  \left(  \xi\right)  \right\vert =\left\vert \frac{\left[
\beta_{\left(  2\right)  }|_{\mathbf{A}}\right]  ^{-1}\left(  \xi\right)
}{\xi\left(  1+\left[  \beta_{\left(  2\right)  }|_{\mathbf{A}}\right]
^{-1}\left(  \xi\right)  \right)  }\right\vert .
\end{equation}
It can be shown numerically that
\[
t\mapsto\sqrt{\dfrac{1}{\Theta_{\lambda}\left(  -te^{-t}\right)  }}\in
L^{2}\left(  \mathbf{A},dt\right)  .
\]
Note that
\[
\mathbf{B}=\beta_{\left(  2\right)  }\left(  \mathbf{A}\right)  =\left(
-\frac{1}{2\sqrt{e}},\frac{\sqrt{e}}{2}\right)
\]
and the Lebesgue measure of $\mathbf{B}$ which we denote by $\left\vert
\mathbf{B}\right\vert $ is equal to
\[
\frac{\sqrt{e}}{2}+\frac{1}{2\sqrt{e}}=\frac{1+e}{2\sqrt{e}}.
\]
Consequently, $\mathbf{B}$ is up to a null set a fundamental domain of the
lattice $\frac{1+e}{2\sqrt{e}}%
\mathbb{Z}
.$ Thus, the trigonometric system
\[
\left\{  \xi\mapsto\frac{\exp\left(  2\pi i\xi k\right)  }{\left(  \frac
{1+e}{2\sqrt{e}}\right)  ^{1/2}}:k\in\frac{2\sqrt{e}}{1+e}%
\mathbb{Z}
\right\}
\]
is an orthonormal basis for the Hilbert space $L^{2}\left(  \mathbf{B,}%
d\xi\right)  .$ Put $c=\frac{2\sqrt{e}}{1+e}$ and
\[
\mathbf{f}\left(  t\right)  =\frac{\mathbf{1}_{\Omega}\left(  t\right)
}{\left(  \Theta_{\lambda}\left(  \beta_{\left(  2\right)  }\left(  t\right)
\right)  \right)  ^{1/2}}=\frac{\mathbf{1}_{\Omega}\left(  t\right)  }%
{\sqrt{\Theta_{\lambda}\left(  -te^{-t}\right)  }}.
\]
Next, we shall prove that the system%
\[
\left\{  \frac{\mathbf{1}_{\Omega}\left(  t\right)  e^{2\pi i\left(
e^{-t}t\right)  k}}{\sqrt{\Theta_{\lambda}\left(  -te^{-t}\right)  }}:k\in c%
\mathbb{Z}
\right\}  =\left\{  \pi_{\lambda}\left(  \gamma\right)  \mathbf{f}:\gamma
\in\Gamma_{M\text{ }}\right\}
\]
is a tight frame for $L^{2}\left(  \Omega\right)  $ with frame bounds
$c^{-1}.$ Indeed, given $\mathbf{h}\in L^{2}\left(  \Omega\right)  $ we have
\[%
{\displaystyle\sum\limits_{k\in c\mathbb{Z}}}
\left\vert \left\langle \mathbf{h,}\pi_{\lambda}\left(  0,k,0\right)
\mathbf{f}\right\rangle _{L^{2}\left(  \Omega\right)  }\right\vert ^{2}=%
{\displaystyle\sum\limits_{k\in c\mathbb{Z}}}
\left\vert \int_{\mathbf{A}}\frac{\mathbf{h}\left(  t\right)  e^{2\pi i\left(
e^{-t}t\right)  k}}{\left(  \Theta_{\lambda}\left(  \beta_{\left(  2\right)
}\left(  t\right)  \right)  \right)  ^{1/2}}dt\right\vert ^{2}.
\]
The change of variable $t=\left(  \beta_{\boldsymbol{J}\left(  \lambda\right)
}\right)  ^{-1}\left(  \xi\right)  $ allows us to proceed as follows
\begin{align*}
&
{\displaystyle\sum\limits_{k\in c\mathbb{Z}}}
\left\vert \left\langle \mathbf{h,}\pi_{\lambda}\left(  0,k,0\right)
\mathbf{f}\right\rangle _{L^{2}\left(  \mathbf{A}\right)  }\right\vert ^{2}\\
&  =%
{\displaystyle\sum\limits_{k\in c\mathbb{Z}}}
\left\vert \int_{\mathbf{B}}\left[  \frac{\mathbf{h}\left(  \beta
_{\boldsymbol{J}\left(  \lambda\right)  }^{-1}\left(  \xi\right)  \right)
e^{-2\pi i\xi k}}{\left(  \Theta_{\lambda}\left(  \beta_{\left(  2\right)
}\left(  \beta_{\left(  2\right)  }^{-1}\left(  \xi\right)  \right)  \right)
\right)  ^{1/2}}\right]  d\left(  \left(  \beta_{\left(  2\right)  }\right)
^{-1}\left(  \xi\right)  \right)  \right\vert ^{2}\\
&  =%
{\displaystyle\sum\limits_{k\in c\mathbb{Z}}}
\left\vert \int_{\mathbf{B}}\left[  \frac{\mathbf{h}\left(  \beta_{\left(
2\right)  }^{-1}\left(  \xi\right)  \right)  e^{-2\pi i\xi k}\Theta_{\lambda
}\left(  \xi\right)  }{\left(  \Theta_{\lambda}\left(  \xi\right)  \right)
^{1/2}}\right]  d\xi\right\vert ^{2}\\
&  =%
{\displaystyle\sum\limits_{k\in c\mathbb{Z}}}
\left\vert \int_{\mathbf{B}}\left[  \frac{\mathbf{h}\left(  \beta_{\left(
2\right)  }^{-1}\left(  \xi\right)  \right)  \Theta_{\lambda}\left(
\xi\right)  }{\left(  \Theta_{\lambda}\left(  \xi\right)  \right)  ^{1/2}%
}\right]  e^{-2\pi i\xi k}d\xi\right\vert ^{2}\\
&  =%
{\displaystyle\sum\limits_{k\in c\mathbb{Z}}}
\left\vert \int_{\mathbf{B}}\left[  \frac{\left(  \frac{1+e}{2\sqrt{e}%
}\right)  ^{1/2}\mathbf{h}\left(  \beta_{\left(  2\right)  }^{-1}\left(
\xi\right)  \right)  \Theta_{\lambda}\left(  \xi\right)  }{\left(
\Theta_{\lambda}\left(  \xi\right)  \right)  ^{1/2}}\right]  \frac{e^{-2\pi
i\left(  \xi k\right)  }}{\left(  \frac{1+e}{2\sqrt{e}}\right)  ^{1/2}}%
d\xi\right\vert ^{2}.
\end{align*}
Put
\begin{equation}
\mathbf{H}\left(  \xi\right)  =\frac{\left(  \frac{1+e}{2\sqrt{e}}\right)
^{1/2}\mathbf{h}\left(  \beta_{\left(  2\right)  }^{-1}\left(  \xi\right)
\right)  \Theta_{\lambda}\left(  \xi\right)  }{\left(  \Theta_{\lambda}\left(
\xi\right)  \right)  ^{1/2}}.
\end{equation}
Then
\[%
{\displaystyle\sum\limits_{k\in c\mathbb{Z}}}
\left\vert \left\langle \mathbf{h,}\pi_{\lambda}\left(  0,k,0\right)
\mathbf{f}\right\rangle _{L^{2}\left(  \Omega\right)  }\right\vert ^{2}=%
{\displaystyle\sum\limits_{k\in c\mathbb{Z}}}
\left\vert \int_{\mathbf{B}}\mathbf{H}\left(  \xi\right)  \left(
\frac{e^{-2\pi i\xi k}}{\left(  \frac{1+e}{2\sqrt{e}}\right)  ^{1/2}}\right)
d\xi\right\vert ^{2}.
\]
Since
\[
\left\{  \xi\mapsto\frac{e^{-2\pi i\left(  \xi k\right)  }\mathbf{1}_{\left(
-\frac{1}{2\sqrt{e}},\frac{\sqrt{e}}{2}\right)  }}{\left(  \frac{1+e}%
{2\sqrt{e}}\right)  ^{1/2}}:k\in c%
\mathbb{Z}
\right\}
\]
is an orthonormal basis for $L^{2}\left(  \mathbf{B}\right)  $ it follows
that
\begin{align*}
&
{\displaystyle\sum\limits_{k\in c\mathbb{Z}}}
\left\vert \left\langle \mathbf{h,}\pi_{\lambda}\left(  0,k,0\right)
\mathbf{f}\right\rangle _{L^{2}\left(  \Omega\right)  }\right\vert ^{2}%
=\int_{\mathbf{B}}\left\vert \mathbf{H}\left(  \xi\right)  \right\vert
^{2}d\xi\\
&  =\int_{\mathbf{B}}\left\vert \frac{\left(  \frac{1+e}{2\sqrt{e}}\right)
^{1/2}\mathbf{h}\left(  \beta_{\left(  2\right)  }^{-1}\left(  \xi\right)
\right)  \Theta_{\lambda}\left(  \xi\right)  }{\left(  \Theta_{\lambda}\left(
\xi\right)  \right)  ^{1/2}}\right\vert ^{2}d\xi=\left(  \ast\right)
\end{align*}
and
\begin{align*}
\left(  \ast\right)   &  =c^{-1}\int_{\mathbf{B}}\left\vert \mathbf{h}\left(
\beta_{\left(  2\right)  }^{-1}\left(  \xi\right)  \right)  \left(
\Theta_{\lambda}\left(  \xi\right)  \right)  ^{1/2}\right\vert ^{2}d\xi\\
&  =c^{-1}\int_{\mathbf{B}}\left\vert \mathbf{h}\left(  \beta_{\left(
2\right)  }^{-1}\left(  \xi\right)  \right)  \right\vert ^{2}\underset
{=d\left(  \beta_{\left(  2\right)  }^{-1}\left(  \xi\right)  \right)
}{\underbrace{\Theta_{\lambda}\left(  \xi\right)  d\xi}}\\
&  =c^{-1}\int_{\mathbf{B}}\left\vert \mathbf{h}\left(  \beta_{\left(
2\right)  }^{-1}\left(  \xi\right)  \right)  \right\vert ^{2}d\left(
\beta_{\left(  2\right)  }^{-1}\left(  \xi\right)  \right)  .
\end{align*}
Next, the change of variable $\xi=\beta_{\boldsymbol{J}\left(  \lambda\right)
}\left(  t\right)  $ yields
\begin{align*}
&
{\displaystyle\sum\limits_{k\in c\mathbb{Z}}}
\left\vert \left\langle \mathbf{h,}\pi_{\lambda}\left(  0,k,0\right)
\mathbf{f}\right\rangle _{L^{2}\left(  \mathbf{A}\right)  }\right\vert ^{2}\\
&  =c^{-1}\int_{\mathbf{A}}\left\vert \mathbf{h}\left(  \beta_{\left(
2\right)  }^{-1}\left(  \beta_{\left(  2\right)  }\left(  t\right)  \right)
\right)  \right\vert ^{2}d\left(  \beta_{\left(  2\right)  }^{-1}\left(
\beta_{\left(  2\right)  }\left(  t\right)  \right)  \right)  \\
&  =c^{-1}\int_{\mathbf{A}}\left\vert \mathbf{h}\left(  t\right)  \right\vert
^{2}dt.
\end{align*}
As such,
\[
\left\{  \frac{\mathbf{1}_{\Omega}\left(  t\right)  e^{2\pi i\left(
e^{-t}t\right)  k}}{\left(  \Theta_{\lambda}\left(  \beta_{\left(  2\right)
}\left(  t\right)  \right)  \right)  ^{1/2}}:k\in c%
\mathbb{Z}
\right\}
\]
is a tight frame for $L^{2}\left(  \mathbf{A}\right)  $ with frame bounds
$c^{-1}=\left(  1+e\right)  \left(  2\sqrt{e}\right)  ^{-1}.$ Next, we observe
that
\[
\left\{  \left[  \frac{1}{2}-\ell,\frac{1}{2}-\ell\right)  :\ell\in%
\mathbb{Z}
\right\}
\]
is a measurable partition of the real line. Thus, the collection
\[
\left\{  t\mapsto\pi_{\lambda}\left(  \left(  0,0,\ell\right)  \left(
0,k,0\right)  \right)  \frac{\mathbf{1}_{\Omega}\left(  t\right)  }{\left(
\Theta_{\lambda}\left(  \beta_{\left(  2\right)  }\left(  t\right)  \right)
\right)  ^{1/2}}:k\in\frac{2\sqrt{e}}{1+e}%
\mathbb{Z}
,\ell\in\mathbb{Z}\right\}
\]
is a tight frame with frame bound $c^{-1}.$ In conclusion, the family of
vectors
\begin{equation}
\pi_{\lambda}\left(  \exp\left(
\mathbb{Z}
A_{1}\right)  \exp\left(  \left(  1+e\right)  ^{-1}\left(  2\sqrt{e}\right)
\mathbb{Z}
X_{2}\right)  \right)  \frac{\mathbf{1}_{\Omega}\left(  \cdot\right)
}{\left(  \Theta_{\lambda}\left(  \beta_{\left(  2\right)  }\left(
\cdot\right)  \right)  \right)  ^{1/2}}%
\end{equation}
is a tight frame with frame bound $\left(  1+e\right)  \left(  2\sqrt
{e}\right)  ^{-1}$ for $L^{2}\left(
\mathbb{R}
\right)  .$ 

\section{Intermediate results}

In order to help the reader keep track of the various objects introduced in
our scheme, we provide the following tables.
\[%
\begin{tabular}
[c]{|l|}\hline
$\theta_{\lambda}\left(
{\displaystyle\sum\limits_{k=1}^{n_{2}}}
a_{k}A_{k}\right)  =\mathbf{C}\left(  a\right)  ^{\ast}\lambda=\left(  \left.
\left(  e^{-\mathrm{ad}\left(  a_{n_{2}}A_{n_{2}}\right)  }\cdots
e^{-\mathrm{ad}\left(  a_{2}A_{2}\right)  }e^{-\mathrm{ad}\left(  a_{1}%
A_{1}\right)  }\right)  \right\vert _{\mathfrak{p}}\right)  ^{\ast}\lambda
$\\\hline
$D_{\theta_{\lambda}}=\mathrm{Jac}_{\theta_{\lambda}}\left(  0\right)  ,$
$\mathcal{T}=\left\{  \boldsymbol{I}:\boldsymbol{I}=\left(  j_{1}%
,\cdots,j_{n_{2}}\right)  \text{ and }1\leq j_{1}<\cdots<j_{n_{2}}\leq
n_{1}\right\}  $\\\hline
$\mathcal{A}=\left\{  D_{\theta_{\lambda}}\left(  \boldsymbol{I}\right)
:\boldsymbol{I}\in\mathcal{T}\text{ and }\mathrm{\det}\left(  D_{\theta
_{\lambda}}\left(  \boldsymbol{I}\right)  \right)  \neq0\right\}  $\\\hline
Fixing $\boldsymbol{J}\left(  \lambda\right)  =\left(  j_{1},\cdots,j_{n_{2}%
}\right)  \in\mathcal{T}$, \ $\mathbf{P}_{\boldsymbol{J}\left(  \lambda
\right)  }\left(  X_{k}\right)  =\left\{
\begin{array}
[c]{c}%
X_{k}\text{ if }k\in\boldsymbol{J}\left(  \lambda\right) \\
0\text{ if }k\notin\boldsymbol{J}\left(  \lambda\right)
\end{array}
\right.  $\\\hline
$\beta_{\boldsymbol{J}\left(  \lambda\right)  }\left(
{\displaystyle\sum\limits_{k=1}^{n_{2}}}
a_{k}A_{k}\right)  =\mathbf{P}_{\boldsymbol{J}\left(  \lambda\right)  }^{\ast
}\left(  \mathbf{C}\left(  a_{1},a_{2},\cdots,a_{n_{2}}\right)  \right)
^{\ast}\lambda$\\\hline
$\mathfrak{Data}_{\left(  \pi_{\lambda},G\right)  }=\left\{  \beta
_{\boldsymbol{J}\left(  \lambda\right)  }:D_{\theta_{\lambda}}\left(
\boldsymbol{J}\left(  \lambda\right)  \right)  \in\mathcal{A}\right\}  ,$
$\mathbf{L}=\left\{  s\in\left(  0,\infty\right)  :\sum_{k=1}^{n_{2}}\left[
-\frac{s}{2},\frac{s}{2}\right]  A_{k}\subset\mathcal{O}\right\}  $\\\hline
$\epsilon$ is a fixed positive number in $\mathbf{L,}$ $\Omega_{\epsilon}%
=\exp\left(  \left[  -\frac{\epsilon}{2},\frac{\epsilon}{2}\right)
A_{1}\right)  \cdots\exp\left(  \left[  -\frac{\epsilon}{2},\frac{\epsilon}%
{2}\right)  A_{n_{2}}\right)  $\\\hline
$\Gamma_{M\text{ }}^{\epsilon}=\exp\left(  \epsilon%
\mathbb{Z}
A_{1}\right)  \cdots\exp\left(  \epsilon%
\mathbb{Z}
A_{n_{2}}\right)  ,$ $\mathcal{O}_{\epsilon}=\sum_{k=1}^{n_{2}}\left(
-\frac{\epsilon}{2},\frac{\epsilon}{2}\right)  A_{k}\text{ and }\digamma
=\sum_{k=1}^{n_{2}}\left[  -\frac{\epsilon}{2},\frac{\epsilon}{2}\right)
A_{k}$\\\hline
$\mathbf{f}_{r,\epsilon}=r\times\mathbf{1}_{\Omega_{\epsilon}},$
$\mathbf{T}^{\epsilon}=\left\{
\begin{array}
[c]{c}%
\mathcal{T}:\mathcal{T}\text{ is a full-rank lattice of }\mathbf{P}%
_{\boldsymbol{J}\left(  \lambda\right)  }^{\ast}\left(  \mathfrak{p}^{\ast
}\right)  \text{ and }\\
\sum_{\kappa\in\mathcal{T}}\mathbf{1}_{\beta_{\boldsymbol{J}\left(
\lambda\right)  }\left(  \digamma_{\epsilon}\right)  }\left(  \xi
+\kappa\right)  \leq1\text{ for every }\xi\in\mathbf{P}_{\boldsymbol{J}\left(
\lambda\right)  }^{\ast}\left(  \mathfrak{p}^{\ast}\right)
\end{array}
\right\}  $\\\hline
$\left\vert \det\left(  \frac{id-e^{-\mathrm{ad}\left(  A\right)  }%
}{\mathrm{ad}\left(  A\right)  }\right)  \right\vert dA=w\left(  \nu\left(
A\right)  \right)  d\left(  \nu\left(  A\right)  \right)  =\rho\left(
A\right)  dA$\\\hline
$\Gamma_{P}^{\epsilon}=\exp\left(
\mathbb{Z}
\text{-span}\left\{  \mathcal{L}^{\epsilon}X_{j}:j\in\boldsymbol{J}\left(
\lambda\right)  \right\}  \right)  ,$ $\Lambda_{\epsilon}^{\star}=%
\mathbb{Z}
\text{-span}\left\{  \left(  \mathcal{L}^{\epsilon}\right)  ^{\top}X_{j}%
^{\ast}:j\in\boldsymbol{J}\left(  \lambda\right)  \right\}  \subset
\mathfrak{p}^{\ast}$\\\hline
$\Gamma^{\epsilon}=\left(  \Gamma_{M\text{ }}^{\epsilon}\right)  ^{-1}%
\Gamma_{P}^{\epsilon}=\exp\left(  \epsilon%
\mathbb{Z}
A_{n_{1}}\right)  \cdots\exp\left(  \epsilon%
\mathbb{Z}
A_{1}\right)  \exp\left(  \sum_{j\in\boldsymbol{J}\left(  \lambda\right)  }%
\mathbb{Z}
\mathcal{L}^{\epsilon}X_{j}\right)  ,$ $\Theta_{\lambda}\left(  \xi\right)
=\left\vert \det\mathrm{Jac}\left[  \beta_{\boldsymbol{J}\left(
\lambda\right)  }|_{\mathcal{O}}\right]  ^{-1}\right\vert $\\\hline
\end{tabular}
\ \ \ \ \text{ }%
\]
and
\[%
\begin{tabular}
[c]{|l|}\hline
Fix $\epsilon\in\mathbf{L,}$ $\Omega_{\epsilon}^{\circ}=\exp\left(  \left(
-\frac{\epsilon}{2},\frac{\epsilon}{2}\right)  A_{1}\right)  \cdots\exp\left(
\left(  -\frac{\epsilon}{2},\frac{\epsilon}{2}\right)  A_{n_{2}}\right)
$\\\hline
$\mathcal{O}_{\epsilon}=\left(  -\frac{\epsilon}{2},\frac{\epsilon}{2}\right)
A_{1}+\cdots+\left(  -\frac{\epsilon}{2},\frac{\epsilon}{2}\right)  A_{n_{2}%
}\subset\mathfrak{m,}$\\\hline
$\Phi_{\boldsymbol{J}\left(  \lambda\right)  }^{_{\epsilon}}\left(
\mathbf{e}\left(  a_{1}A_{1}+\cdots+a_{n_{2}}A_{n_{2}}\right)  \right)
=\beta_{\boldsymbol{J}\left(  \lambda\right)  }\left(  a_{1}A_{1}%
+\cdots+a_{n_{2}}A_{n_{2}}\right)  $\\\hline
$\mathbf{s}\in C_{c}^{\infty}\left(  M\right)  $ and supp$\left(
\mathbf{s}\circ\left[  \Phi_{\boldsymbol{J}\left(  \lambda\right)
}^{_{\epsilon}}\right]  ^{-1}\right)  $ is compact and contained in
$\beta_{\boldsymbol{J}\left(  \lambda\right)  }\left(  \mathcal{O}_{\epsilon
}\right)  $\\\hline
$\Lambda_{\epsilon}=\sum_{j\in\boldsymbol{J}\left(  \lambda\right)  }%
\mathbb{Z}
\mathcal{L}^{\epsilon}X_{j}$ and $\Lambda_{\epsilon}^{\star}=\sum
_{j\in\boldsymbol{J}\left(  \lambda\right)  }%
\mathbb{Z}
\left(  \mathcal{L}^{\epsilon}\right)  ^{\top}X_{j}^{\ast}$\\\hline
$\Gamma_{P}^{\epsilon}=\exp\left(  \Lambda_{\epsilon}\right)  =\exp\left(
\sum_{j\in\boldsymbol{J}\left(  \lambda\right)  }%
\mathbb{Z}
\mathcal{L}^{\epsilon}X_{j}\right)  ,$ $\Psi_{\boldsymbol{J}\left(
\lambda\right)  }^{\epsilon}\left(  \xi\right)  d\xi=d\mu_{M}\left(  \left[
\Phi_{\boldsymbol{J}\left(  \lambda\right)  }^{_{\epsilon}}\right]
^{-1}\left(  \xi\right)  \right)  $\\\hline
$\Upsilon\left(  m\right)  =\sqrt{\Psi_{\boldsymbol{J}\left(  \lambda\right)
}^{\epsilon}\left(  \Phi_{\boldsymbol{J}\left(  \lambda\right)  }^{_{\epsilon
}}\left(  m\right)  \right)  \left\vert \mathrm{\det}\text{ }\left(
\mathcal{L}^{\epsilon}\right)  ^{\top}\right\vert }\mathbf{s}\left(  m\right)
$\\\hline
$\Gamma_{M\text{ }}^{\epsilon}\left(  \mathbf{s}\right)  =\Gamma_{M\text{ }%
}^{\epsilon}\subset M$, with $\inf\left\{
{\displaystyle\sum\limits_{\gamma\in\Gamma_{M\text{ }}^{\epsilon}}}
\left\vert \Upsilon\left(  \gamma m\right)  \right\vert ^{2}:m\in M\right\}
>0$, $\sup\left\{
{\displaystyle\sum\limits_{\gamma\in\Gamma_{M\text{ }}^{\epsilon}}}
\left\vert \Upsilon\left(  \gamma m\right)  \right\vert ^{2}:m\in M\right\}
<\infty$\\\hline
\end{tabular}
\]

We recall that $G$ is a simply connected, connected semi-direct product group
of the type $G=P\rtimes M$ where $P,M$ are closed subgroups of $G.$ Moreover,
it is assumed that there exists a linear functional $\lambda$ in
$\mathfrak{p}^{\ast}$ such that the induced representation $\pi_{\lambda
}=\mathrm{ind}_{P}^{G}\left(  \chi_{\lambda}\right)  $ which is realized as
acting in $L^{2}\left(  M,d\mu_{M}\right)  $ is an irreducible representation
of $G.$ Next, appealing to Pukansky's condition, we know that $\pi_{\lambda}$
is irreducible if and only if

\begin{enumerate}
\item $\dim\mathfrak{p}=\dim\mathfrak{g-}\frac{1}{2}\dim\left(  \mathbf{O}%
_{\lambda}\right)  $

\item and the linear variety $\lambda+\mathfrak{p}^{\bot}$ is contained in
$\mathbf{O}_{\lambda}$ (the coadjoint orbit of $\lambda.$)
\end{enumerate}

\begin{definition}
A polarization algebra (or an ideal) in $\mathfrak{g}$ subordinated to a
linear functional $\lambda$ is a subalgebra (or an ideal) $\mathfrak{p}$ of
$\mathfrak{g}$ such that $\mathfrak{p}$ is subordinated to $\lambda$ and
\[
\dim\mathfrak{p}=\dim\mathfrak{g}-\frac{\dim\left(  \mathbf{O}_{\lambda
}\right)  }{2}.
\]

\end{definition}

\begin{lemma}
If $\pi_{\lambda}=\mathrm{ind}_{P}^{G}\left(  \chi_{\lambda}\right)  $ is
irreducible then $\mathfrak{p}$ is a polarization algebra which is an ideal
subordinated to the linear functional $\lambda.$
\end{lemma}

\begin{proof}
Let us assume that $\pi_{\lambda}$ is irreducible. Appealing to Result
$\left(  32\right)  $ on Page $38,$ \cite{Heptin}) together with the fact that
$P$ is assumed to be normal, it must be the case that $\mathfrak{p}%
=\log\left(  P\right)  $ is a polarization ideal for the linear functional
$\lambda.$
\end{proof}

\begin{lemma}
\label{Leptin}(\cite{Heptin} Page $43$) Assume that $\mathfrak{p}$ is a
polarization ideal subordinated to the linear functional $\lambda$. Then the
following are equivalent.

\begin{enumerate}
\item $\lambda+\mathfrak{p}^{\bot}\subset\mathbf{O}_{\lambda},$

\item $Ad^{\ast}P\left(  \lambda\right)  =\lambda+\mathfrak{p}^{\bot}%
=\lambda+\mathfrak{m}^{\ast},$

\item $\mathfrak{p}$ is a polarization ideal for all linear functionals
$\lambda+\lambda^{\prime}$ where $\lambda^{\prime}\in\mathfrak{m}^{\ast}.$
\end{enumerate}
\end{lemma}

\begin{lemma}
The irreducibility of $\pi_{\lambda}$ implies that $M$ acts freely on the
linear functional $\lambda.$
\end{lemma}

\begin{proof}
Since $\pi_{\lambda}$ is irreducible, then $\lambda+\mathfrak{p}^{\bot}%
\subset\mathbf{O}_{\lambda}.$ Next,
\[
\dim\mathfrak{p}=\dim\mathfrak{g}-\frac{\dim\left(  \mathbf{O}_{\lambda
}\right)  }{2}.
\]
Note that
\begin{align*}
\left.  \dim\mathfrak{p}=\dim\mathfrak{g}-\frac{\dim\left(  \mathbf{O}%
_{\lambda}\right)  }{2}\right.   &  \Rightarrow-\frac{\dim\left(
\mathbf{O}_{\lambda}\right)  }{2}=\dim\mathfrak{p}-\dim\mathfrak{g}\\
&  \Rightarrow\frac{1}{2}\dim\left(  \mathbf{O}_{\lambda}\right)
=\dim\mathfrak{g}-\dim\mathfrak{p}\\
&  \Rightarrow\frac{1}{2}\dim\left(  \mathbf{O}_{\lambda}\right)
=\dim\mathfrak{m}\\
&  \Rightarrow\dim\mathbf{O}_{\lambda}=2\dim\mathfrak{m.}%
\end{align*}
Thus, the coadjoint orbit of $\lambda$ is a $2\times\dim\mathfrak{m}%
$-dimensional manifold. Since $\dim\left(  Ad^{\ast}P\left(  \lambda\right)
\right)  =\dim\mathfrak{m}$ it is now clear that
\[
\dim\left(  Ad^{\ast}M\left(  \lambda\right)  \right)  =\dim\left(
\mathfrak{m}\right)  =\dim\left(  \theta_{\lambda}\left(  \mathfrak{m}\right)
\right)  .
\]
Next, since $M$ is a completely solvable Lie group, there exists a basis for
$\mathfrak{p}$ such that a matrix representation of $Ad^{\ast}\left(
m\right)  $ (which we denote by $\left[  Ad^{\ast}\left(  m\right)  \right]
$) is a lower triangular matrix with only real eigenvalues. As such,
\[
\left\{  a\in%
\mathbb{R}
^{n_{2}}:\left(  \left.  \left(  e^{-\mathrm{ad}\left(  a_{n_{2}}A_{n_{2}%
}\right)  }\cdots e^{-\mathrm{ad}\left(  a_{2}A_{2}\right)  }e^{-\mathrm{ad}%
\left(  a_{1}A_{1}\right)  }\right)  \right\vert _{\mathfrak{p}}\right)
^{\ast}\lambda=\lambda\right\}
\]
cannot be a non-trivial subgroup discrete subgroup (or a lattice) of $%
\mathbb{R}
^{n_{2}}.$ Finally, since
\[
\dim\left(  Ad^{\ast}P\left(  \lambda\right)  \right)  =\dim\left(
\mathfrak{m}\right)
\]
it must then be the case that $M$ acts freely on $\lambda.$
\end{proof}

Define the map $\Phi_{\lambda}:M\rightarrow Ad\left(  M\right)  ^{\ast}%
\lambda\subset\mathfrak{p}^{\ast}$such that
\[
\Phi_{\lambda}\left(  m\right)  =Ad\left(  m\right)  ^{\ast}\lambda.
\]

\begin{lemma}
$\Phi_{\lambda}$ is a smooth map that has constant rank. Moreover,
$m\mapsto\Phi_{\lambda}\left(  m\right)  $ is an equivariant diffeomorphism
and the rank of $\Phi_{\lambda}$ is equal to the dimension of $M.$
\end{lemma}

\begin{proof}
For any given $n\in M,$ we have
\[
\Phi_{\lambda}\left(  nm\right)  =Ad\left(  nm\right)  ^{\ast}\lambda
=Ad\left(  n\right)  ^{\ast}Ad\left(  m\right)  ^{\ast}\lambda=Ad\left(
n\right)  ^{\ast}\Phi_{\lambda}\left(  m\right)  .
\]
Thus, the map $\Phi_{\lambda}$ is equivariant with respect to the
multiplication action of $M$ on itself and the coadjoint action of $M$ on
$Ad\left(  M\right)  ^{\ast}\lambda\mathfrak{.}$ Additionally, since the
multiplicative action of $M$ on itself is transitive, according to the
Equivariant Rank Theorem (see \cite{Lee}, Theorem $7.5$), $\Phi_{\lambda}$ is
a smooth map that has constant rank. Next, appealing to \cite{Lee}, Theorem
$7.19,$ since the isotropy group of $\lambda$ is trivial by assumption, the
map $m\mapsto\Phi_{\lambda}\left(  m\right)  $ is an equivariant
diffeomorphism and the rank of this map must be equal to the dimension of $M.$
\end{proof}

\vskip0.5cm \noindent Next, we define the map $\mathbf{e}:\mathfrak{m}%
\rightarrow M$ by
\begin{equation}
\mathbf{e}\left(
{\displaystyle\sum\limits_{k=1}^{n_{2}}}
a_{k}A_{k}\right)  =\exp\left(  a_{1}A_{1}\right)  \cdots\exp\left(  a_{n_{2}%
}A_{n_{2}}\right)  =\exp\left(  \nu\left(
{\displaystyle\sum\limits_{k=1}^{n_{2}}}
a_{k}A_{k}\right)  \right)
\end{equation}

Then $\mathbf{e}$ is bijective and bi-analytic. Next, its inverse defines a
global coordinate system on $G$ (see Chapter $1,$ \cite{Heptin})

\begin{lemma}
$\theta_{\lambda}=\Phi_{\lambda}\circ\mathbf{e.}$
\end{lemma}

\begin{proof}
Given $a_{1},\cdots,a_{n_{2}}\in%
\mathbb{R}
,$%
\begin{align*}
\Phi_{\lambda}\left(  \mathbf{e}\left(
{\displaystyle\sum\limits_{k=1}^{n_{2}}}
a_{k}A_{k}\right)  \right)   &  =\Phi_{\lambda}\left(  \exp\left(  a_{1}%
A_{1}\right)  \cdots\exp\left(  a_{n_{2}}A_{n_{2}}\right)  \right) \\
&  =\mathbf{C}\left(  a_{1},\cdots,a_{n_{2}}\right)  ^{\ast}\lambda
=\theta_{\lambda}\left(
{\displaystyle\sum\limits_{k=1}^{n_{2}}}
a_{k}A_{k}\right)  .
\end{align*}

\end{proof}

\begin{lemma}
\label{immersion}$\theta_{\lambda}$ defines an immersion of $\mathfrak{m}$
into $\mathfrak{p}^{\ast}$ of constant rank $\dim\mathfrak{m}$.
\end{lemma}

\begin{proof}
Since $\theta_{\lambda}=\Phi_{\lambda}\circ\mathbf{e}$, it follows that the
rank of the smooth map $\theta_{\lambda}$ is equal to $\dim\mathfrak{m}$ at
every $A\in\mathfrak{m}$.
\end{proof}

\vskip0.5cm \noindent Next, we recall that $D_{\theta_{\lambda}}%
=\mathrm{Jac}_{\theta_{\lambda}}\left(  0\right)  $ is the Jacobian of the map
$\theta_{\lambda}$ at the zero element in $\mathfrak{m}$. Let
\[
\mathfrak{B}_{\mathfrak{p}}=\left(  X_{1},\cdots,X_{n_{1}}\right)
\]
be a strong Malcev basis for $\mathfrak{p.}$ Fix
\begin{equation}
\boldsymbol{J}\left(  \lambda\right)  =\left(  j_{1},\cdots,j_{n_{2}}\right)
\in\mathcal{T} \label{J}%
\end{equation}
such that $D_{\theta_{\lambda}}\left(  j_{1},\cdots,j_{n_{2}}\right)
\in\mathcal{A}$. Define the map $\mathbf{P}_{\boldsymbol{J}\left(
\lambda\right)  }:\mathfrak{p}\rightarrow\mathfrak{p}$ such that
\[
\mathbf{P}_{\boldsymbol{J}\left(  \lambda\right)  }\left(  X_{k}\right)
=\left\{
\begin{array}
[c]{c}%
X_{k}\text{ if }k\in\boldsymbol{J}\left(  \lambda\right) \\
0\text{ if }k\notin\boldsymbol{J}\left(  \lambda\right)
\end{array}
\right.  .
\]
Put $a=\left(  a_{1},\cdots,a_{n_{2}}\right)  ,$ and
\[
A\left(  a\right)  =A=%
{\displaystyle\sum\limits_{k=1}^{n_{2}}}
a_{k}A_{k}.
\]
From the definition of $\mathbf{P}_{\boldsymbol{J}\left(  \lambda\right)  },$
it is not hard to verify that $\mathbf{P}_{\boldsymbol{J}\left(
\lambda\right)  }$ is a linear map. Moreover, $\mathbf{P}_{\boldsymbol{J}%
\left(  \lambda\right)  }^{2}=\mathbf{P}_{\boldsymbol{J}\left(  \lambda
\right)  }$ and the nullspace and range of $\mathbf{P}_{\boldsymbol{J}\left(
\lambda\right)  }$ are orthogonal to each other with respect to the natural
dot product on $\mathfrak{p.}$ Therefore, $\mathbf{P}_{\boldsymbol{J}\left(
\lambda\right)  }$ defines an orthogonal projection of rank $n_{2}$ on
$\mathfrak{p.}$ $\mathbf{P}_{\boldsymbol{J}\left(  \lambda\right)  }$ is an
orthogonal projection of rank $n_{2}$ on the Lie algebra $\mathfrak{p.}$

\begin{lemma}
Let $X$ be an element of the range of $\mathbf{P}_{\boldsymbol{J}\left(
\lambda\right)  }$. Then given $\left(  a_{1},\cdots,a_{n_{2}}\right)  \in%
\mathbb{R}
^{n_{2}},$
\[
\left\langle \mathbf{C}\left(  a_{1},\cdots,a_{n_{2}}\right)  ^{\ast}%
\lambda,X\right\rangle =\left\langle \mathbf{P}_{\boldsymbol{J}\left(
\lambda\right)  }^{\ast}\mathbf{C}\left(  a_{1},\cdots,a_{n_{2}}\right)
^{\ast}\lambda,X\right\rangle .
\]

\end{lemma}

\begin{proof}
If $X$ belongs to the range of $\mathbf{P}_{\boldsymbol{J}\left(
\lambda\right)  }$ then $X=\mathbf{P}_{\boldsymbol{J}\left(  \lambda\right)
}X^{\prime}$ for some element $X^{\prime}$ in the Lie algebra $\mathfrak{p.}$
Consequently,
\begin{align*}
\left\langle \mathbf{C}\left(  a_{1},\cdots,a_{n_{2}}\right)  ^{\ast}%
\lambda,X\right\rangle  &  =\left\langle \mathbf{C}\left(  a_{1}%
,\cdots,a_{n_{2}}\right)  ^{\ast}\lambda,\mathbf{P}_{\boldsymbol{J}\left(
\lambda\right)  }X^{\prime}\right\rangle \\
&  =\left\langle \mathbf{C}\left(  a_{1},\cdots,a_{n_{2}}\right)  ^{\ast
}\lambda,\mathbf{P}_{\boldsymbol{J}\left(  \lambda\right)  }\mathbf{P}%
_{\boldsymbol{J}\left(  \lambda\right)  }X^{\prime}\right\rangle \\
&  =\left\langle \mathbf{P}_{\boldsymbol{J}\left(  \lambda\right)  }^{\ast
}\mathbf{C}\left(  a_{1},\cdots,a_{n_{2}}\right)  ^{\ast}\lambda
,\mathbf{P}_{\boldsymbol{J}\left(  \lambda\right)  }X^{\prime}\right\rangle \\
&  =\left\langle \mathbf{P}_{\boldsymbol{J}\left(  \lambda\right)  }^{\ast
}\mathbf{C}\left(  a_{1},\cdots,a_{n_{2}}\right)  ^{\ast}\lambda
,X\right\rangle .
\end{align*}

\end{proof}

Appealing to the preceding lemma, it is clear that if $X$ is an element of the
range of $\mathbf{P}_{\boldsymbol{J}\left(  \lambda\right)  }$ then
\[
\left[  \pi_{\lambda}\left(  \exp X\right)  \mathbf{f}\right]  \left(
\mathbf{e}\left(  A\right)  \right)  =e^{2\pi i\left\langle \mathbf{P}%
_{\boldsymbol{J}\left(  \lambda\right)  }^{\ast}\mathbf{C}\left(  a\right)
^{\ast}\lambda,X\right\rangle }\mathbf{f}\left(  \mathbf{e}\left(  A\right)
\right)  .
\]

Next, let
\begin{equation}
\beta_{\boldsymbol{J}\left(  \lambda\right)  }:\mathfrak{m}\rightarrow
\sum_{k\in\boldsymbol{J}\left(  \lambda\right)  }%
\mathbb{R}
X_{k}^{\ast}=\mathbf{P}_{\boldsymbol{J}\left(  \lambda\right)  }^{\ast}\left(
\mathfrak{p}^{\ast}\right)  \label{betalamda}%
\end{equation}
such that
\[
\beta_{\boldsymbol{J}\left(  \lambda\right)  }\left(  A\right)  =\mathbf{P}%
_{\boldsymbol{J}\left(  \lambda\right)  }^{\ast}\mathbf{C}\left(  a\right)
^{\ast}\lambda=\mathbf{P}_{\boldsymbol{J}\left(  \lambda\right)  }^{\ast
}\theta_{\lambda}\left(  A\right)  .
\]

\begin{lemma}
\label{U}There exists an open set $\mathcal{O}$ around the neutral element in
$\mathfrak{m}$ such that the restriction of $\beta_{\boldsymbol{J}\left(
\lambda\right)  }$ to $\mathcal{O}$ defines a diffeomorphism between
$\mathcal{O}$ and $\beta_{\boldsymbol{J}\left(  \lambda\right)  }\left(
\mathcal{O}\right)  .$ Moreover, there exists $\epsilon>0$ such that given
\begin{equation}
\mathcal{O}_{\epsilon}=\left(  -\frac{\epsilon}{2},\frac{\epsilon}{2}\right)
A_{1}+\cdots+\left(  -\frac{\epsilon}{2},\frac{\epsilon}{2}\right)  A_{n_{2}%
}\subset\mathfrak{m}%
\end{equation}
the restriction of $\beta_{\boldsymbol{J}\left(  \lambda\right)  }$ to
$\mathcal{O}_{\epsilon}$ determines a diffeomorphism between $\mathcal{O}%
_{\epsilon}$ and $\beta_{\boldsymbol{J}\left(  \lambda\right)  }\left(
\mathcal{O}_{\epsilon}\right)  .$
\end{lemma}

\begin{proof}
Since the Jacobian of the map $\beta_{\boldsymbol{J}\left(  \lambda\right)  }$
at zero is invertible, there exists a neighborhood of the zero element such
that the restriction of $\beta_{\boldsymbol{J}\left(  \lambda\right)  }$ to
such a neighborhood has constant rank. Thus, appealing to Proposition $5.16$,
\cite{Lee}, there exists an open set $\mathcal{O}$ around the neutral element
in $\mathfrak{m}$ such that the restriction of $\beta_{\boldsymbol{J}\left(
\lambda\right)  }$ to $\mathcal{O}$ defines a diffeomorphism between
$\mathcal{O}$ and its range $\beta_{\boldsymbol{J}\left(  \lambda\right)
}\left(  \mathcal{O}\right)  .$ The second part of the lemma is trivial and we
shall omit its proof.
\end{proof}

\begin{proposition}
\label{tilesM}Let $\epsilon>0$ and set
\[
\Omega_{\epsilon}=\exp\left(  \left[  -\frac{\epsilon}{2},\frac{\epsilon}%
{2}\right)  A_{1}\right)  \exp\left(  \left[  -\frac{\epsilon}{2}%
,\frac{\epsilon}{2}\right)  A_{2}\right)  \cdots\exp\left(  \left[
\frac{\epsilon}{2},\frac{\epsilon}{2}\right)  A_{n_{2}}\right)  .
\]
Then
\[
\left\{  \gamma^{-1}\Omega_{\epsilon}:\gamma\in\exp\left(  \epsilon%
\mathbb{Z}
A_{1}\right)  \exp\left(  \epsilon%
\mathbb{Z}
A_{2}\right)  \cdots\exp\left(  \epsilon%
\mathbb{Z}
A_{n_{2}}\right)  \right\}
\]
is a measurable partition of $M.$
\end{proposition}

\begin{proof}
We shall prove this claim by induction on the dimension of $M.$ Clearly for
$n_{2}=1$ the base case holds. Next, let
\[
M_{1}=\exp\left(
\mathbb{R}
A_{1}\right)  \cdots\exp\left(
\mathbb{R}
A_{n_{2}-1}\right)  \text{ and }H_{1}=\exp\left(
\mathbb{R}
A_{n_{2}}\right)
\]
Put $M=M_{1}H_{1}$,
\[
\Omega_{1}=\exp\left[  -\frac{\epsilon}{2},\frac{\epsilon}{2}\right)
A_{1}\cdots\exp\left[  -\frac{\epsilon}{2},\frac{\epsilon}{2}\right)
A_{n_{2}-1}%
\]
and $\Gamma_{M\text{ }}^{1}=\exp\left(  \epsilon%
\mathbb{Z}
A_{1}\right)  \cdots\exp\left(  \epsilon%
\mathbb{Z}
A_{n_{2}-1}\right)  .$ Next define
\[
\Omega=\exp\left[  -\frac{\epsilon}{2},\frac{\epsilon}{2}\right)  A_{1}%
\cdots\exp\left[  -\frac{\epsilon}{2},\frac{\epsilon}{2}\right)  A_{n_{2}}%
\]
and $\Gamma_{M\text{ }}=\exp\left(  \epsilon\mathbb{Z}A_{1}\right)  \cdots
\exp\left(  \epsilon\mathbb{Z}A_{n_{2}}\right)  .$ Let $m\in M.$ Then, there
exists a unique pair $\left(  m_{n_{2}-1},m_{n_{2}}\right)  \in M_{1}\times
H_{1}$ such that $m=m_{n_{2}-1}m_{n_{2}}.$ Appealing to the inductive
hypothesis, $m=\left(  \gamma_{n_{2}-1}^{-1}\omega_{n_{2}-1}\right)  \left(
\gamma_{n_{2}}^{-1}\omega_{n_{2}}\right)  $ where
\[
\gamma_{n_{2}-1}\in\Gamma_{M\text{ }}^{1},\omega_{n_{2}-1}\in\Omega_{1}%
,\gamma_{n_{2}}\in\exp\left(  \epsilon%
\mathbb{Z}
A_{n_{2}}\right)  ,\omega_{n_{2}}\in H_{1}.
\]
Next, let $e$ be the neutral element in $\Gamma_{M\text{ }}.$ Then
\begin{align*}
m  &  =e\left(  \gamma_{n_{2}-1}^{-1}\omega_{n_{2}-1}\right)  \left(
\gamma_{n_{2}}^{-1}\omega_{n_{2}}\right) \\
&  =\left(  \gamma_{n_{2}}^{-1}\gamma_{n_{2}}\right)  \left(  \gamma_{n_{2}%
-1}^{-1}\omega_{n_{2}-1}\right)  \left(  \gamma_{n_{2}}^{-1}\omega_{n_{2}%
}\right) \\
&  =\gamma_{n_{2}}^{-1}\left(  \gamma_{n_{2}}\left(  \gamma_{n_{2}-1}%
^{-1}\omega_{n_{2}-1}\right)  \gamma_{n_{2}}^{-1}\right)  \omega_{n_{2}}.
\end{align*}
Now, since $M_{1}$ is a normal subgroup of $M,$ $\gamma_{n_{2}-1}^{-1}%
\omega_{n_{2}-1}\in M_{1}$ it follows that $\gamma_{n_{2}}\left(
\gamma_{n_{2}-1}^{-1}\omega_{n_{2}-1}\right)  \gamma_{n_{2}}^{-1}\in M_{1}.$
Thus, there exists a unique pair $\left(  \widetilde{\gamma}_{n_{2}%
-1},\widetilde{\omega}_{n_{2}-1}\right)  \in\Gamma_{M\text{ }}^{1}\times
\Omega_{1}$ such that
\[
\gamma_{n_{2}}\left(  \gamma_{n_{2}-1}^{-1}\omega_{n_{2}-1}\right)
\gamma_{n_{2}}^{-1}=\widetilde{\gamma}_{n_{2}-1}^{-1}\widetilde{\omega}%
_{n_{2}-1}.
\]
Consequently,
\begin{align*}
m  &  =\gamma_{n_{2}}^{-1}\underset{=\widetilde{\gamma}_{n_{2}-1}%
^{-1}\widetilde{\omega}_{n_{2}-1}}{\underbrace{\left(  \gamma_{n_{2}}\left(
\gamma_{n_{2}-1}^{-1}\omega_{n_{2}-1}\right)  \gamma_{n_{2}}^{-1}\right)  }%
}\omega_{n_{2}}\\
&  =\gamma_{n_{2}}^{-1}\left(  \widetilde{\gamma}_{n_{2}-1}^{-1}%
\widetilde{\omega}_{n_{2}-1}\right)  \omega_{n_{2}}\\
&  =\left(  \gamma_{n_{2}}^{-1}\widetilde{\gamma}_{n_{2}-1}^{-1}\right)
\left(  \widetilde{\omega}_{n_{2}-1}\omega_{n_{2}}\right) \\
&  =\left(  \widetilde{\gamma}_{n_{2}-1}\gamma_{n_{2}}\right)  ^{-1}\left(
\widetilde{\omega}_{n_{2}-1}\omega_{n_{2}}\right)  .
\end{align*}
Since the factorization above is unique,
\[
\left\{  \gamma^{-1}\Omega_{\epsilon}:\gamma\in\exp\left(  \epsilon%
\mathbb{Z}
A_{1}\right)  \exp\left(  \epsilon%
\mathbb{Z}
A_{2}\right)  \cdots\exp\left(  \epsilon%
\mathbb{Z}
A_{n_{2}}\right)  \right\}
\]
forms a measurable partition of $M.$
\end{proof}

\noindent Fix $\epsilon\in\mathbf{L,}$%
\begin{equation}
\mathbf{L}=\left\{  s>0:\sum_{k=1}^{n_{2}}\left[  -\frac{s}{2},\frac{s}%
{2}\right]  A_{k}\subset\mathcal{O}\right\}  \subset%
\mathbb{R}
.
\end{equation}
Next, we define
\begin{equation}
\Omega_{\epsilon}=\exp\left[  -\frac{\epsilon}{2},\frac{\epsilon}{2}\right)
A_{1}\exp\left[  -\frac{\epsilon}{2},\frac{\epsilon}{2}\right)  A_{2}%
\cdots\exp\left[  \frac{\epsilon}{2},\frac{\epsilon}{2}\right)  A_{n_{2}}%
\end{equation}
such that
\begin{equation}
\mathcal{O}_{\epsilon}=\left(  -\frac{\epsilon}{2},\frac{\epsilon}{2}\right)
A_{1}+\cdots+\left(  -\frac{\epsilon}{2},\frac{\epsilon}{2}\right)  A_{n_{2}%
}\subset\mathfrak{m}%
\end{equation}
is an open set around the neutral element in $\mathfrak{m}$ such that the
restriction of $\beta_{\boldsymbol{J}\left(  \lambda\right)  }$ to
$\mathcal{O}_{\epsilon}$ defines a diffeomorphism between $\mathcal{O}%
_{\epsilon}$ and its image set $\beta_{\boldsymbol{J}\left(  \lambda\right)
}\left(  \mathcal{O}_{\epsilon}\right)  .$ Fix a discrete set $\Gamma_{M\text{
}}\subset M,$
\[
\Gamma_{M\text{ }}=\Gamma_{M\text{ }}^{\epsilon}=\exp\left(  \epsilon%
\mathbb{Z}
A_{1}\right)  \exp\left(  \epsilon%
\mathbb{Z}
A_{2}\right)  \cdots\exp\left(  \epsilon%
\mathbb{Z}
A_{n_{2}}\right)
\]
such that
\[
\left\{  \gamma^{-1}\Omega_{\epsilon}:\gamma\in\Gamma_{M\text{ }}\right\}
\]
is a $d\mu_{M}$-measurable partition of $M.$ Next, put
\[
\digamma_{\epsilon}=%
{\displaystyle\sum\limits_{k=1}^{n_{2}}}
\left[  -\frac{\epsilon}{2},\frac{\epsilon}{2}\right)  A_{k}\subset
\mathfrak{m}%
\]
such that $\mathbf{e}\left(  \digamma_{\epsilon}\right)  =\Omega_{\epsilon}.$
Fix $\boldsymbol{J}\left(  \lambda\right)  =\left(  j_{1}<\cdots<j_{m}\right)
$ such that $D_{\theta_{\lambda}}\left(  \boldsymbol{J}\left(  \lambda\right)
\right)  \in\mathcal{A}$ where
\[
\mathcal{A}=\left\{  D_{\theta_{\lambda}}\left(  \boldsymbol{I}\right)
:\boldsymbol{I}\in\mathcal{T}\text{ and }\mathrm{\det}\left(  D_{\theta
_{\lambda}}\left(  \boldsymbol{I}\right)  \right)  \neq0\right\}  .
\]
Finally, define $V_{\boldsymbol{J}\left(  \lambda\right)  }=%
\mathbb{R}
$-span$\left\{  X_{j}:j\in\boldsymbol{J}\left(  \lambda\right)  \right\}  $
and
\[
\mathbf{T}^{\epsilon}=\left\{
\begin{array}
[c]{c}%
\mathcal{T}:\mathcal{T}\text{ is a full-rank lattice of }\mathbf{P}%
_{\boldsymbol{J}\left(  \lambda\right)  }^{\ast}\left(  \mathfrak{p}^{\ast
}\right)  \text{ and }\\
\sum_{\kappa\in\mathcal{T}}\mathbf{1}_{\beta_{\boldsymbol{J}\left(
\lambda\right)  }\left(  \digamma_{\epsilon}\right)  }\left(  \xi
+\kappa\right)  \leq1\text{ for every }\xi\in\mathbf{P}_{\boldsymbol{J}\left(
\lambda\right)  }^{\ast}\left(  \mathfrak{p}^{\ast}\right)
\end{array}
\right\}  .
\]
Moreover, for a linear operator $\mathcal{L}^{\epsilon}:V_{\boldsymbol{J}%
\left(  \lambda\right)  }\rightarrow V_{\boldsymbol{J}\left(  \lambda\right)
}$ we define $\Lambda_{\epsilon}\left(  \mathcal{L}^{\epsilon}\right)  =%
\mathbb{Z}
$-span$\left\{  \mathcal{L}^{\epsilon}X_{j}:j\in\boldsymbol{J}\left(
\lambda\right)  \right\}  $ and
\[
\Lambda_{\epsilon}^{\star}\left(  \mathcal{L}^{\epsilon}\right)  =%
\mathbb{Z}
\text{-span}\left\{  \left(  \mathcal{L}^{\epsilon}\right)  ^{\top}X_{j}%
^{\ast}:j\in\boldsymbol{J}\left(  \lambda\right)  \right\}  .
\]

\begin{lemma}
\label{GammaP}$\mathbf{T}^{\epsilon}$ is a non-empty set.\ In other words,
there exists an invertible linear operator $\mathcal{L}^{\epsilon
}:V_{\boldsymbol{J}\left(  \lambda\right)  }\rightarrow V_{\boldsymbol{J}%
\left(  \lambda\right)  }$ such that $\beta_{\boldsymbol{J}\left(
\lambda\right)  }\left(  \digamma_{\epsilon}\right)  $ is contained in a
measurable fundamental domain of a dual lattice $\Lambda_{\epsilon}^{\star
}\left(  \mathcal{L}^{\epsilon}\right)  .$
\end{lemma}

\begin{proof}
First, we claim that $\beta_{\boldsymbol{J}\left(  \lambda\right)  }\left(
\digamma_{\epsilon}\right)  $ is a bounded set. Indeed, since the compact set
\[
\sum_{k=1}^{n_{2}}\left[  -\frac{\epsilon}{2},\frac{\epsilon}{2}\right]  A_{k}%
\]
is properly contained in $\mathcal{O}$ and because the restriction of
$\beta_{\boldsymbol{J}\left(  \lambda\right)  }$ to $\mathcal{O}$ is a
diffeomorphism, it follows that the image of $\sum_{k=1}^{n_{2}}\left[
-\frac{\epsilon}{2},\frac{\epsilon}{2}\right]  A_{k}$ under $\beta
_{\boldsymbol{J}\left(  \lambda\right)  }$ is compact as well. Moreover, the
fact that $\beta_{\boldsymbol{J}\left(  \lambda\right)  }\left(
\digamma_{\epsilon}\right)  $ is a subset of
\[
\beta_{\boldsymbol{J}\left(  \lambda\right)  }\left(  \sum_{k=1}^{n_{2}%
}\left[  -\frac{\epsilon}{2},\frac{\epsilon}{2}\right]  A_{k}\right)
\]
implies that $\beta_{\boldsymbol{J}\left(  \lambda\right)  }\left(
\digamma_{\epsilon}\right)  $ is a bounded set. Next, let
\begin{equation}
\delta_{\epsilon}=\sup\left\{  \left\Vert \beta_{\boldsymbol{J}\left(
\lambda\right)  }\left(
{\displaystyle\sum\limits_{k=1}^{n_{2}}}
a_{k}A_{k}\right)  \right\Vert _{\max}:\left(  a_{1},\cdots,a_{n_{2}}\right)
\in\left[  -\frac{\epsilon}{2},\frac{\epsilon}{2}\right)  ^{n_{2}}\right\}
>0. \label{delta}%
\end{equation}
Secondly, put
\[
\mathfrak{F}=\sum_{j\in\boldsymbol{J}\left(  \lambda\right)  }\left[
-\delta_{\epsilon},\delta_{\epsilon}\right)  X_{j}^{\ast}.
\]
Note that $\mathfrak{F}$ is a fundamental domain for the lattice $\sum
_{j\in\boldsymbol{J}\left(  \lambda\right)  }2\delta_{\epsilon}%
\mathbb{Z}
X_{j}^{\ast}.$ Thirdly, from the definition of $\delta_{\epsilon},$ it is
clear that $\beta_{\boldsymbol{J}\left(  \lambda\right)  }\left(
\digamma_{\epsilon}\right)  \subseteq\mathfrak{F}$. Define the linear map
$\mathcal{L}^{\epsilon}:V_{\boldsymbol{J}\left(  \lambda\right)  }\rightarrow
V_{\boldsymbol{J}\left(  \lambda\right)  }$ such that
\[
\mathcal{L}^{\epsilon}\left(  \sum_{j\in\boldsymbol{J}\left(  \lambda\right)
}x_{j}X_{j}\right)  =\sum_{j\in\boldsymbol{J}\left(  \lambda\right)  }%
\frac{x_{j}}{2\delta_{\epsilon}}X_{j}.
\]
In other words, the matrix representation of $\mathcal{L}^{\epsilon}$ with
respect to the ordered basis $\left(  X_{j}\right)  _{j\in\boldsymbol{J}%
\left(  \lambda\right)  }$ is given by
\[
\left[  \mathcal{L}^{\epsilon}\right]  _{\left(  X_{j}\right)  _{j\in
\boldsymbol{J}\left(  \lambda\right)  }}=\left[
\begin{array}
[c]{ccc}%
\frac{1}{2\delta_{\epsilon}} &  & \\
& \ddots & \\
&  & \frac{1}{2\delta_{\epsilon}}%
\end{array}
\right]  ,\text{ and }\left[  \left(  \mathcal{L}^{\epsilon}\right)  ^{\top
}\right]  _{\left(  X_{j}^{\ast}\right)  _{j\in\boldsymbol{J}\left(
\lambda\right)  }}=\left[
\begin{array}
[c]{ccc}%
2\delta_{\epsilon} &  & \\
& \ddots & \\
&  & 2\delta_{\epsilon}%
\end{array}
\right]  .
\]
Thus, for every $\xi\in\mathbf{P}_{\boldsymbol{J}\left(  \lambda\right)
}^{\ast}\left(  \mathfrak{p}^{\ast}\right)  ,$
\[
\sum_{\kappa\in\Lambda_{\epsilon}^{\star}\left(  \mathcal{L}^{\epsilon
}\right)  }\mathbf{1}_{\beta_{\boldsymbol{J}\left(  \lambda\right)  }\left(
\digamma_{\epsilon}\right)  }\left(  \xi+\kappa\right)  \leq1.
\]

\end{proof}

\begin{example}
Let $G$ be a connected, simply connected solvable Lie group with Lie algebra
$\mathfrak{g}$ spanned by $\left\{  X_{1},X_{2},X_{3},A_{1},A_{2}\right\}  $
with non-trivial Lie brackets
\[
\left[  A_{1},X_{2}\right]  =X_{1},\left[  A_{2},X_{1}\right]  =X_{1},\left[
A_{2},X_{2}\right]  =X_{2},\left[  A_{2},X_{3}\right]  =-2X_{3}%
\]
and
\[
\left[  \mathrm{ad}X_{1}\right]  _{\left(  X_{1},X_{2},X_{3}\right)  }=\left[
\begin{array}
[c]{ccc}%
0 & 1 & 0\\
0 & 0 & 0\\
0 & 0 & 0
\end{array}
\right]  ,\left[  \mathrm{ad}X_{2}\right]  _{\left(  X_{1},X_{2},X_{3}\right)
}=\left[
\begin{array}
[c]{ccc}%
1 & 0 & 0\\
0 & 1 & 0\\
0 & 0 & -2
\end{array}
\right]  .
\]
Fix $\lambda=X_{1}^{\ast}+X_{2}^{\ast}+X_{3}^{\ast}.$ With straightforward
computations, we obtain
\[
\beta_{\boldsymbol{J}\left(  \lambda\right)  }\left(  t_{1},t_{2}\right)
=\left\{
\begin{array}
[c]{c}%
\left(  e^{-t_{2}},e^{-t_{2}}\left(  t_{1}-1\right)  \right)  \text{ if
}\boldsymbol{J}\left(  \lambda\right)  =\left(  1,2\right) \\
\left(  e^{-t_{2}}\left(  t_{1}-1\right)  ,e^{2t_{2}}\right)  \text{ if
}\boldsymbol{J}\left(  \lambda\right)  =\left(  2,3\right)
\end{array}
\right.  .
\]
Finally, the orbital data corresponding to the linear functional $\lambda$ is
given by $\left\{  \beta_{\left(  1,2\right)  },\beta_{\left(  2,3\right)
}\right\}  .$ Next, fix the map
\[
\beta_{\left(  2,3\right)  }\left(  t\right)  =\left(  e^{-t_{2}}\left(
t_{1}-1\right)  ,e^{2t_{2}}\right)
\]
in the orbital data corresponding to $\lambda.$ Clearly $\beta_{\left(
2,3\right)  }$ defines a diffeomorphism between $%
\mathbb{R}
^{2}$ and $\beta_{\left(  2,3\right)  }\left(
\mathbb{R}
^{2}\right)  .$ It follows that
\[
\left\Vert \beta_{\left(  2,3\right)  }\left(  t\right)  \right\Vert _{\max
}=\max\left\{  e^{2t_{2}},\frac{\left\vert t_{1}-1\right\vert }{e^{t_{2}}%
}\right\}
\]
and
\[
\delta_{\epsilon}=\sup\left\{  \max\left\{  e^{2t_{2}},\frac{\left\vert
t_{1}-1\right\vert }{e^{t_{2}}}\right\}  :t\in\left[  -\frac{\epsilon}%
{2},\frac{\epsilon}{2}\right)  ^{2}\right\}  .
\]
Thus, $\beta_{\left(  2,3\right)  }\left(  \left[  -\frac{\epsilon}{2}%
,\frac{\epsilon}{2}\right)  ^{2}\right)  $ is contained in a fundamental
domain of $2\delta_{\epsilon}%
\mathbb{Z}
X_{2}^{\ast}+2\delta_{\epsilon}%
\mathbb{Z}
X_{3}^{\ast}.$
\end{example}

\section{Proofs of main theorems}

\subsection{Proof of Theorem \ref{Main}}

Fix $\epsilon\in\mathbf{L}$ and%
\[
\Omega_{\epsilon}=\exp\left(  \left[  -\frac{\epsilon}{2},\frac{\epsilon}%
{2}\right)  A_{1}\right)  \exp\left(  \left[  -\frac{\epsilon}{2}%
,\frac{\epsilon}{2}\right)  A_{2}\right)  \cdots\exp\left(  \left[
\frac{\epsilon}{2},\frac{\epsilon}{2}\right)  A_{n_{2}}\right)  .
\]
Fix $\mathcal{L}^{\epsilon}\in\mathbf{T}^{\epsilon}$ and put $\Gamma
_{P}^{\epsilon}=\exp\left(  \Lambda_{^{\epsilon}}\right)  .$ For any given
positive function $r$ defined on $\Omega_{\epsilon},$ let
\[
\mathbf{f}_{r,\epsilon}=r\times\mathbf{1}_{\Omega_{\epsilon}}%
\]
and let
\[
b\left(  \xi\right)  =\left[  \beta_{\boldsymbol{J}\left(  \lambda\right)
}\right]  ^{-1}\left(  \xi\right)  .
\]
Recall that $\rho$ is the Radon-Nikodym derivative given by%
\begin{equation}
d\mu_{M}\left(  \mathbf{e}\left(
{\displaystyle\sum\limits_{k=1}^{n_{2}}}
a_{k}A_{k}\right)  \right)  =d\mu_{M}\left(  \exp\left(  a_{1}A_{1}\right)
\cdots\exp\left(  a_{n_{2}}A_{n_{2}}\right)  \right)  =\rho\left(  A\right)
dA
\end{equation}
where $d\mu_{M}$ is a left Haar measure on the solvable group $M$ and $dA$ is
the Lebesgue measure on $\mathfrak{m}=\mathbb{R}^{n_{2}}$.

\begin{lemma}
\label{sumconv} Let $\mathbf{h}\in L^{2}\left(  \Omega_{\epsilon},d\mu
_{M}\right)  \cap C\left(  \Omega_{\epsilon}\right)  $. If
\[
\xi\mapsto r\left(  \mathbf{e}\left(  b\left(  \xi\right)  \right)  \right)
\rho\left(  b\left(  \xi\right)  \right)  \Theta_{\lambda}\left(  \xi\right)
\]
is square-integrable over $\beta_{\boldsymbol{J}\left(  \lambda\right)
}\left(  \digamma_{\epsilon}\right)  $ with respect to the Lebesgue measure
then
\begin{align*}
&
{\displaystyle\sum\limits_{\alpha\in\Gamma_{P}^{\epsilon}}}
\left\vert \left\langle \mathbf{h},\pi_{\lambda}\left(  \alpha\right)
\mathbf{f}_{r,\epsilon}\right\rangle _{L^{2}\left(  \Omega_{\epsilon},d\mu
_{M}\right)  }\right\vert ^{2}\\
& =\left\vert \det\left(  \mathcal{L}^{\epsilon}\right)  ^{\top}\right\vert
\int_{\digamma_{\epsilon}}\left\vert \mathbf{h}\left(  \mathbf{e}\left(
A\right)  \right)  \right\vert ^{2}\left[  \left\vert r\left(  \mathbf{e}%
\left(  A\right)  \right)  \right\vert ^{2}\rho\left(  A\right)  ^{2}%
\Theta_{\lambda}\left(  \beta_{\boldsymbol{J}\left(  \lambda\right)  }\left(
A\right)  \right)  \right]  dA.
\end{align*}

\end{lemma}

\begin{proof}
Let $\mathbf{f}_{r,\epsilon},\mathbf{h}\ $be as defined in the statement of
the lemma. Then
\begin{align*}
&
{\displaystyle\sum\limits_{\alpha\in\Gamma_{P}^{\epsilon}}}
\left\vert \left\langle \mathbf{h},\pi_{\lambda}\left(  \alpha\right)
\mathbf{f}_{r,\epsilon}\right\rangle _{L^{2}\left(  \Omega_{\epsilon},d\mu
_{M}\right)  }\right\vert ^{2}\\
&  =%
{\displaystyle\sum\limits_{\exp\left(  X\right)  \in\Gamma_{P}^{\epsilon}}}
\left\vert \left\langle \mathbf{h},\pi_{\lambda}\left(  \exp X\right)
\mathbf{f}_{r,\epsilon}\right\rangle _{L^{2}\left(  \Omega_{\epsilon},d\mu
_{M}\right)  }\right\vert ^{2}\\
&  =%
{\displaystyle\sum\limits_{\exp\left(  X\right)  \in\Gamma_{P}^{\epsilon}}}
\left\vert \int_{M}\mathbf{h}\left(  m\right)  e^{-2\pi i\left\langle
\lambda,\log\left(  m^{-1}\exp\left(  X\right)  m\right)  \right\rangle
}\overline{\mathbf{f}_{r,\epsilon}\left(  m\right)  }d\mu_{M}\left(  m\right)
\right\vert ^{2}.
\end{align*}
Since the support of $\mathbf{f}_{r,\epsilon}$ is equal to $\Omega_{\epsilon
},$ it follows that
\begin{align*}
&
{\displaystyle\sum\limits_{\alpha\in\Gamma_{P}^{\epsilon}}}
\left\vert \left\langle \mathbf{h},\pi_{\lambda}\left(  \alpha\right)
\mathbf{f}_{r,\epsilon}\right\rangle _{L^{2}\left(  \Omega_{\epsilon},d\mu
_{M}\right)  }\right\vert ^{2}\\
&  =%
{\displaystyle\sum\limits_{\exp\left(  X\right)  \in\Gamma_{P}^{\epsilon}}}
\left\vert \int_{\Omega_{\epsilon}}\mathbf{h}\left(  m\right)  e^{-2\pi
i\left\langle \lambda,\log\left(  m^{-1}\exp\left(  X\right)  m\right)
\right\rangle }\overline{\mathbf{f}_{r,\epsilon}\left(  m\right)  }d\mu
_{M}\left(  m\right)  \right\vert ^{2}.
\end{align*}
Next, since
\[
\mathbf{e}\left(  \digamma_{\epsilon}\right)  =\Omega_{\epsilon},
\]
for a given $m\in\Omega_{\epsilon},$ there exists a unique $A\in
\digamma_{\epsilon}$ such that
\[
m=\mathbf{e}\left(  A\right)
\]
and
\begin{align*}
&
{\displaystyle\sum\limits_{\alpha\in\Gamma_{P}^{\epsilon}}}
\left\vert \left\langle \mathbf{h},\pi_{\lambda}\left(  \alpha\right)
\mathbf{f}_{r,\epsilon}\right\rangle _{L^{2}\left(  \Omega_{\epsilon},d\mu
_{M}\right)  }\right\vert ^{2}\\
&  =%
{\displaystyle\sum\limits_{\exp\left(  X\right)  \in\Gamma_{P}^{\epsilon}}}
\left\vert \int_{\digamma_{\epsilon}}\mathbf{h}\left(  \mathbf{e}\left(
A\right)  \right)  e^{-2\pi i\left\langle D_{\boldsymbol{J}\left(
\lambda\right)  }\mathbf{C}\left(  a\right)  ^{\ast}\lambda,X\right\rangle
}\overline{\mathbf{f}_{r,\epsilon}\left(  \mathbf{e}\left(  A\right)  \right)
}\underset{=\rho\left(  A\right)  dA}{\underbrace{d\mu_{M}\left(
\mathbf{e}\left(  A\right)  \right)  }}\right\vert ^{2}\\
&  =%
{\displaystyle\sum\limits_{\exp\left(  X\right)  \in\Gamma_{P}^{\epsilon}}}
\left\vert \int_{\digamma_{\epsilon}}\mathbf{h}\left(  \mathbf{e}\left(
A\right)  \right)  e^{-2\pi i\left\langle D_{\boldsymbol{J}\left(
\lambda\right)  }\mathbf{C}\left(  a\right)  ^{\ast}\lambda,X\right\rangle
}\overline{\mathbf{f}_{r,\epsilon}\left(  \mathbf{e}\left(  A\right)  \right)
}\rho\left(  A\right)  dA\right\vert ^{2}\\
&  =%
{\displaystyle\sum\limits_{\exp\left(  X\right)  \in\Gamma_{P}^{\epsilon}}}
\left\vert \int_{\digamma_{\epsilon}}e^{-2\pi i\left\langle \beta
_{\boldsymbol{J}\left(  \lambda\right)  }\left(  A\right)  ,X\right\rangle
}\left(  r\left(  \mathbf{e}\left(  A\right)  \right)  \times\mathbf{h}\left(
\mathbf{e}\left(  A\right)  \right)  \rho\left(  A\right)  \right)
dA\right\vert ^{2}=\left(  \ast\right)
\end{align*}
Set $\xi=\beta_{\boldsymbol{J}\left(  \lambda\right)  }\left(  A\right)  $ and
let
\[
\Psi_{r,\boldsymbol{J}\left(  \lambda\right)  }\left(  \xi\right)  =r\left(
\mathbf{e}\left(  b\left(  \xi\right)  \right)  \right)  \mathbf{h}\left(
\mathbf{e}\left(  b\left(  \xi\right)  \right)  \right)  \rho\left(  b\left(
\xi\right)  \right)  .
\]
The change of variable $A=b\left(  \xi\right)  $ gives%
\[
\left(  \ast\right)  =%
{\displaystyle\sum\limits_{\exp\left(  X\right)  \in\Gamma_{P}^{\epsilon}}}
\left\vert \int_{\beta_{\boldsymbol{J}\left(  \lambda\right)  }\left(
\digamma_{\epsilon}\right)  }e^{-2\pi i\left\langle \xi,X\right\rangle }%
\Psi_{r,\boldsymbol{J}\left(  \lambda\right)  }\left(  \xi\right)  d\left(
b\left(  \xi\right)  \right)  \right\vert ^{2}.
\]
Recall that $\Theta_{\lambda}\left(  \xi\right)  $ is the Radon-Nikodym
derivative given by
\[
\frac{d\left(  b\left(  \xi\right)  \right)  }{d\xi}=\Theta_{\lambda}\left(
\xi\right)  .
\]
Since $\beta_{\boldsymbol{J}\left(  \lambda\right)  }\left(  \digamma
_{\epsilon}\right)  $ is contained in a fundamental domain of
\[
\Lambda_{\epsilon}^{\star}\left(  \mathcal{L}^{\epsilon}\right)  =%
\mathbb{Z}
\text{-span}\left\{  \left(  \mathcal{L}^{\epsilon}\right)  ^{\top}X_{j}%
^{\ast}:j\in\boldsymbol{J}\left(  \lambda\right)  \right\}  ,
\]
the system
\begin{equation}
\left\{  \frac{e^{-2\pi i\left\langle \xi,X\right\rangle }\times
\mathbf{1}_{\beta_{\boldsymbol{J}\left(  \lambda\right)  }\left(
\digamma_{\epsilon}\right)  }\left(  \xi\right)  }{\left\vert \det\left(
\mathcal{L}^{\epsilon}\right)  ^{\top}\right\vert ^{1/2}}:X\in\Lambda
_{^{\epsilon}}\right\}
\end{equation}
is a Parseval frame for $L^{2}\left(  \beta_{\boldsymbol{J}\left(
\lambda\right)  }\left(  \digamma_{\epsilon}\right)  \right)  .$ Appealing to
H\"{o}lder's inequality,
\[
\Psi_{r,\boldsymbol{J}\left(  \lambda\right)  }\left(  \xi\right)
\Theta_{\lambda}\left(  \xi\right)  =\mathbf{h}\left(  \mathbf{e}\left(
b\left(  \xi\right)  \right)  \right)  r\left(  \mathbf{e}\left(  b\left(
\xi\right)  \right)  \right)  \times\rho\left(  b\left(  \xi\right)  \right)
\Theta_{\lambda}\left(  \xi\right)
\]
is an element of $L^{1}\left(  \beta_{\boldsymbol{J}\left(  \lambda\right)
}\left(  \digamma_{\epsilon}\right)  \right)  .$ Consequently,
\begin{align*}
\left(  \ast\right)   &  =%
{\displaystyle\sum\limits_{\exp\left(  X\right)  \in\Gamma_{P}^{\epsilon}}}
\left\vert \int_{\beta_{\boldsymbol{J}\left(  \lambda\right)  }\left(
\digamma_{\epsilon}\right)  }e^{-2\pi i\left\langle \xi,X\right\rangle }%
\Psi_{r,\boldsymbol{J}\left(  \lambda\right)  }\left(  \xi\right)
\Theta_{\lambda}\left(  \xi\right)  d\xi\right\vert ^{2}\\
&  =%
{\displaystyle\sum\limits_{\exp\left(  X\right)  \in\Gamma_{P}^{\epsilon}}}
\left\vert \int_{\beta_{\boldsymbol{J}\left(  \lambda\right)  }\left(
\digamma_{\epsilon}\right)  }\left\vert \det\left(  \mathcal{L}^{\epsilon
}\right)  ^{\top}\right\vert ^{1/2}\left[  \frac{e^{-2\pi i\left\langle
\xi,X\right\rangle }}{\left\vert \det\left(  \mathcal{L}^{\epsilon}\right)
^{\top}\right\vert ^{1/2}}\right]  \Psi_{r,\boldsymbol{J}\left(
\lambda\right)  }\left(  \xi\right)  \Theta_{\lambda}\left(  \xi\right)
d\xi\right\vert ^{2}=\left(  \ast\ast\right)
\end{align*}
Next,
\begin{align*}
\left(  \ast\ast\right)   &  =%
{\displaystyle\sum\limits_{\exp\left(  X\right)  \in\Gamma_{P}^{\epsilon}}}
\left\vert \int_{\beta_{\boldsymbol{J}\left(  \lambda\right)  }\left(
\digamma_{\epsilon}\right)  }\left\vert \det\left(  \mathcal{L}^{\epsilon
}\right)  ^{\top}\right\vert ^{1/2}\left[  \frac{e^{-2\pi i\left\langle
\xi,X\right\rangle }}{\left\vert \det\left(  \mathcal{L}^{\epsilon}\right)
^{\top}\right\vert ^{1/2}}\right]  r\left(  \mathbf{e}\left(  b\left(
\xi\right)  \right)  \right)  \right. \\
&  \times\left.  \mathbf{h}\left(  \mathbf{e}\left(  b\left(  \xi\right)
\right)  \right)  \rho\left(  b\left(  \xi\right)  \right)  \Theta_{\lambda
}\left(  \xi\right)  d\xi\right\vert ^{2}\\
&  =\int_{\beta_{\boldsymbol{J}\left(  \lambda\right)  }\left(  \digamma
_{\epsilon}\right)  }\left\vert \left\vert \det\left(  \mathcal{L}^{\epsilon
}\right)  ^{\top}\right\vert ^{1/2}r\left(  \mathbf{e}\left(  b\left(
\xi\right)  \right)  \right)  \right\vert ^{2}\left\vert \mathbf{h}\left(
\mathbf{e}\left(  b\left(  \xi\right)  \right)  \right)  \rho\left(  b\left(
\xi\right)  \right)  \Theta_{\lambda}\left(  \xi\right)  \right\vert ^{2}%
d\xi\\
&  =\left(  \ast\ast\ast\right)
\end{align*}
and
\[
\left(  \ast\ast\ast\right)  =\int_{\beta_{\boldsymbol{J}\left(
\lambda\right)  }\left(  \digamma_{\epsilon}\right)  }\left\vert \left\vert
\det\left(  \mathcal{L}^{\epsilon}\right)  ^{\top}\right\vert ^{1/2}%
\mathbf{h}\left(  \mathbf{e}\left(  b\left(  \xi\right)  \right)  \right)
\right\vert ^{2}\left\vert r\left(  \mathbf{e}\left(  b\left(  \xi\right)
\right)  \right)  \rho\left(  b\left(  \xi\right)  \right)  \Theta_{\lambda
}\left(  \xi\right)  \right\vert ^{2}d\xi.
\]
Finally, the change of variable
\[
A=b\left(  \xi\right)  =\left[  \beta_{\boldsymbol{J}\left(  \lambda\right)
}\right]  ^{-1}\left(  \xi\right)
\]
yields
\begin{align*}
\left(  \ast\ast\ast\right)   &  =\left\vert \det\left(  \mathcal{L}%
^{\epsilon}\right)  ^{\top}\right\vert \int_{\digamma_{\epsilon}}\left\vert
\mathbf{h}\left(  \mathbf{e}\left(  A\right)  \right)  \right\vert ^{2}\left[
\left\vert r\left(  \mathbf{e}\left(  A\right)  \right)  \rho\left(  A\right)
\Theta_{\lambda}\left(  \beta_{\boldsymbol{J}\left(  \lambda\right)  }\left(
A\right)  \right)  \right\vert ^{2}\right]  d\left(  \beta_{\boldsymbol{J}%
\left(  \lambda\right)  }\left(  A\right)  \right) \\
&  =\left\vert \det\left(  \mathcal{L}^{\epsilon}\right)  ^{\top}\right\vert
\int_{\digamma_{\epsilon}}\left\vert \mathbf{h}\left(  \mathbf{e}\left(
A\right)  \right)  \right\vert ^{2}\left[  \left\vert r\left(  \mathbf{e}%
\left(  A\right)  \right)  \rho\left(  A\right)  \Theta_{\lambda}\left(
\beta_{\boldsymbol{J}\left(  \lambda\right)  }\left(  A\right)  \right)
\right\vert \right]  \Theta_{\lambda}\left(  \beta_{\boldsymbol{J}\left(
\lambda\right)  }\left(  A\right)  \right)  d\left(  \beta_{\boldsymbol{J}%
\left(  \lambda\right)  }\left(  A\right)  \right) \\
&  =\left(  \ast\ast\ast\ast\right)
\end{align*}
Since $\Theta_{\lambda}\left(  \xi\right)  $ is the absolute value of the
determinant of the Jacobian of $\left[  \beta_{\boldsymbol{J}\left(
\lambda\right)  }|_{\mathcal{O}}\right]  ^{-1},$
\[
d\left(  \left[  \beta_{\boldsymbol{J}\left(  \lambda\right)  }|_{\mathcal{O}%
}\right]  ^{-1}\left(  \xi\right)  \right)  =\Theta_{\lambda}\left(
\xi\right)  d\xi.
\]
As such,
\[
dA=d\left(  \left[  \beta_{\boldsymbol{J}\left(  \lambda\right)
}|_{\mathcal{O}}\right]  ^{-1}\left(  \beta_{\boldsymbol{J}\left(
\lambda\right)  }\left(  A\right)  \right)  \right)  =\Theta_{\lambda}\left(
\beta_{\boldsymbol{J}\left(  \lambda\right)  }\left(  A\right)  \right)
d\left(  \beta_{\boldsymbol{J}\left(  \lambda\right)  }\left(  A\right)
\right)  ,
\]
and%
\[
\left(  \ast\ast\ast\ast\right)  =\left\vert \det\left(  \mathcal{L}%
^{\epsilon}\right)  ^{\top}\right\vert \int_{\digamma_{\epsilon}}\left\vert
\mathbf{h}\left(  \mathbf{e}\left(  A\right)  \right)  \right\vert ^{2}\left[
\left\vert r\left(  \mathbf{e}\left(  A\right)  \right)  \right\vert ^{2}%
\rho\left(  A\right)  ^{2}\Theta_{\lambda}\left(  \beta_{\boldsymbol{J}\left(
\lambda\right)  }\left(  A\right)  \right)  \right]  dA.
\]

\end{proof}

\begin{lemma}
\label{compact} The function $A\mapsto\left(  \Theta_{\lambda}\left(
\beta_{\boldsymbol{J}\left(  \lambda\right)  }\left(  A\right)  \right)
\right)  ^{-1}$ is integrable on every compact measurable subset of
$\mathcal{O}.$
\end{lemma}

\begin{proof}
We shall prove that if $\mathbf{K}$ is a compact Lebesgue measurable subset of
$\mathcal{O}$ then
\[
\int_{\mathbf{K}}\frac{1}{\Theta_{\lambda}\left(  \beta_{\boldsymbol{J}\left(
\lambda\right)  }\left(  A\right)  \right)  }dA<\infty.
\]
Letting $\xi=\beta_{\boldsymbol{J}\left(  \lambda\right)  }\left(  A\right)
,$%
\begin{align*}
\int_{\mathbf{K}}\frac{1}{\Theta_{\lambda}\left(  \beta_{\boldsymbol{J}\left(
\lambda\right)  }\left(  A\right)  \right)  }dA  &  =\int_{\beta
_{\boldsymbol{J}\left(  \lambda\right)  }\left(  \mathbf{K}\right)  }\frac
{1}{\Theta_{\lambda}\left(  \xi\right)  }d\left(  b\left(  \xi\right)  \right)
\\
&  =\int_{\beta_{\boldsymbol{J}\left(  \lambda\right)  }\left(  \mathbf{K}%
\right)  }\frac{\Theta_{\lambda}\left(  \xi\right)  }{\Theta_{\lambda}\left(
\xi\right)  }d\xi\\
&  =\int_{\beta_{\boldsymbol{J}\left(  \lambda\right)  }\left(  \mathbf{K}%
\right)  }d\xi\\
&  =\left\vert \beta_{\boldsymbol{J}\left(  \lambda\right)  }\left(
\mathbf{K}\right)  \right\vert .
\end{align*}
Using the fact that $\beta_{\boldsymbol{J}\left(  \lambda\right)  }$ is a
diffeomorphism between $\mathcal{O}$ and $\beta_{\boldsymbol{J}\left(
\lambda\right)  }\left(  \mathcal{O}\right)  $ together with the assumption
that $\mathbf{K}$ is a compact measurable subset of $\mathcal{O}\mathfrak{,}$
it follows that $\beta_{\boldsymbol{J}\left(  \lambda\right)  }\left(
\mathbf{K}\right)  $ is compact. Consequently,
\[
\int_{\mathbf{K}}\frac{dA}{\Theta_{\lambda}\left(  \beta_{\boldsymbol{J}%
\left(  \lambda\right)  }\left(  A\right)  \right)  }=\left\vert
\beta_{\boldsymbol{J}\left(  \lambda\right)  }\left(  \mathbf{K}\right)
\right\vert <\infty.
\]

\end{proof}

\begin{lemma}
If $r\left(  \mathbf{e}\left(  A\right)  \right)  =\left(  \rho\left(
A\right)  \Theta_{\lambda}\left(  \beta_{\boldsymbol{J}\left(  \lambda\right)
}\left(  A\right)  \right)  \right)  ^{-1/2}$ then $\mathbf{f}_{r,\epsilon
}=r\times\mathbf{1}_{\Omega_{\epsilon}}$ is square-integrable with respect to
the Haar measure on $M.$
\end{lemma}

\begin{proof}
We note that
\begin{align*}
&  \left.  \left\vert r\left(  \mathbf{e}\left(  A\right)  \right)
\right\vert ^{2}\rho\left(  A\right)  ^{2}\Theta_{\lambda}\left(
\beta_{\boldsymbol{J}\left(  \lambda\right)  }\left(  A\right)  \right)
=\rho\left(  A\right)  \right. \\
&  \Leftrightarrow\left\vert r\left(  \mathbf{e}\left(  A\right)  \right)
\right\vert ^{2}\Theta_{\lambda}\left(  \beta_{\boldsymbol{J}\left(
\lambda\right)  }\left(  A\right)  \right)  =\frac{1}{\rho\left(  A\right)
}\\
&  \Leftrightarrow\left\vert r\left(  \mathbf{e}\left(  A\right)  \right)
\right\vert =\left(  \frac{1}{\rho\left(  A\right)  \Theta_{\lambda}\left(
\beta_{\boldsymbol{J}\left(  \lambda\right)  }\left(  A\right)  \right)
}\right)  ^{1/2}.
\end{align*}
Next, it is clear that\ $\mathbf{f}_{r,\epsilon}\in L^{2}\left(
\Omega_{\epsilon},d\mu_{M}\right)  $ if and only if $\frac{1}{\theta_{\lambda
}\left(  \beta_{\boldsymbol{J}\left(  \lambda\right)  }\left(  A\right)
\right)  }$ is integrable over $\digamma_{\epsilon}.$ Indeed,
\[
\int_{\digamma_{\epsilon}}\left\vert \mathbf{f}_{r,\epsilon}\left(
\mathbf{e}\left(  A\right)  \right)  \right\vert ^{2}\rho\left(  A\right)
dA=\int_{\digamma_{\epsilon}}\frac{dA}{\Theta_{\lambda}\left(  \beta
_{\boldsymbol{J}\left(  \lambda\right)  }\left(  A\right)  \right)  }.
\]
Appealing to Lemma \ref{compact},
\[
\int_{\digamma_{\epsilon}}\frac{1}{\Theta_{\lambda}\left(  \beta
_{\boldsymbol{J}\left(  \lambda\right)  }\left(  A\right)  \right)  }%
dA<\infty.
\]

\end{proof}

\begin{lemma}
Let $\mathbf{h}\in L^{2}\left(  \Omega_{\epsilon},d\mu_{M}\right)  \cap
C\left(  \Omega_{\epsilon}\right)  $. If
\[
r\left(  \mathbf{e}\left(  A\right)  \right)  =\frac{1}{\sqrt{\rho\left(
A\right)  \Theta_{\lambda}\left(  \beta_{\boldsymbol{J}\left(  \lambda\right)
}\left(  A\right)  \right)  \text{ }}}%
\]
then
\[%
{\displaystyle\sum\limits_{\alpha\in\Gamma_{P}^{\epsilon}}}
\left\vert \left\langle \mathbf{h},\pi_{\lambda}\left(  \alpha\right)
\mathbf{f}_{r,\epsilon}\right\rangle _{L^{2}\left(  \Omega_{\epsilon},d\mu
_{M}\right)  }\right\vert ^{2}=\left\vert \det\left(  \mathcal{L}^{\epsilon
}\right)  ^{\top}\right\vert \times\left\Vert \mathbf{h}\right\Vert
_{L^{2}\left(  \Omega_{\epsilon},d\mu_{M}\right)  }^{2}.
\]

\end{lemma}

\begin{proof}
First, if $r\left(  \mathbf{e}\left(  A\right)  \right)  =\left(  \rho\left(
A\right)  \Theta_{\lambda}\left(  \beta_{\boldsymbol{J}\left(  \lambda\right)
}\left(  A\right)  \right)  \right)  ^{-1/2}$ then
\begin{align*}
r\left(  \mathbf{e}\left(  b\left(  \xi\right)  \right)  \right)  \rho\left(
b\left(  \xi\right)  \right)  \Theta_{\lambda}\left(  \xi\right)   &
=\frac{\rho\left(  b\left(  \xi\right)  \right)  \Theta_{\lambda}\left(
\xi\right)  }{\left(  \rho\left(  b\left(  \xi\right)  \right)  \Theta
_{\lambda}\left(  \beta_{\boldsymbol{J}\left(  \lambda\right)  }\left(
\left(  \beta_{\boldsymbol{J}\left(  \lambda\right)  }\right)  ^{-1}\left(
\xi\right)  \right)  \right)  \right)  ^{1/2}}\\
&  =\frac{\rho\left(  b\left(  \xi\right)  \right)  \Theta_{\lambda}\left(
\xi\right)  }{\rho\left(  b\left(  \xi\right)  \right)  ^{1/2}\Theta_{\lambda
}\left(  \xi\right)  ^{1/2}}\\
&  =\rho\left(  b\left(  \xi\right)  \right)  ^{1/2}\Theta_{\lambda}\left(
\xi\right)  ^{1/2}.
\end{align*}
Thus
\begin{align*}
\int_{\beta_{\boldsymbol{J}\left(  \lambda\right)  }\left(  \digamma
_{\epsilon}\right)  }\left\vert r\left(  \mathbf{e}\left(  b\left(
\xi\right)  \right)  \right)  \rho\left(  b\left(  \xi\right)  \right)
\Theta_{\lambda}\left(  \xi\right)  \right\vert ^{2}d\xi &  =\int
_{\beta_{\boldsymbol{J}\left(  \lambda\right)  }\left(  \digamma_{\epsilon
}\right)  }\rho\left(  b\left(  \xi\right)  \right)  \Theta_{\lambda}\left(
\xi\right)  d\xi\\
&  =\int_{\beta_{\boldsymbol{J}\left(  \lambda\right)  }\left(  \digamma
_{\epsilon}\right)  }\rho\left(  b\left(  \xi\right)  \right)  d\left(
b\left(  \xi\right)  \right) \\
&  =\int_{\digamma_{\epsilon}}\rho\left(  A\right)  d\left(  A\right)
=\left(  \ast\right)
\end{align*}
Next $d\mu_{M}\left(  \mathbf{e}\left(  A\right)  \right)  =\rho\left(
A\right)  dA$ implies that
\[
\left(  \ast\right)  =\int_{\digamma_{\epsilon}}d\mu_{M}\left(  \mathbf{e}%
\left(  A\right)  \right)  =\int_{\Omega_{\epsilon}}d\mu_{M}\left(  m\right)
.
\]
Consequently,
\[
\int_{\beta_{\boldsymbol{J}\left(  \lambda\right)  }\left(  \digamma
_{\epsilon}\right)  }\left\vert r\left(  \mathbf{e}\left(  b\left(
\xi\right)  \right)  \right)  \rho\left(  b\left(  \xi\right)  \right)
\Theta_{\lambda}\left(  \xi\right)  \right\vert ^{2}d\xi<\infty.
\]
Thus, the function
\[
\xi\mapsto r\left(  \mathbf{e}\left(  b\left(  \xi\right)  \right)  \right)
\rho\left(  b\left(  \xi\right)  \right)  \Theta_{\lambda}\left(  \xi\right)
\in L^{2}\left(  \beta_{\boldsymbol{J}\left(  \lambda\right)  }\left(
\digamma_{\epsilon}\right)  \right)  .
\]
Appealing to Lemma \ref{sumconv}, we obtain%
\[%
{\displaystyle\sum\limits_{\alpha\in\Gamma_{P}^{\epsilon}}}
\left\vert \left\langle \mathbf{h},\pi_{\lambda}\left(  \alpha\right)
\mathbf{f}_{r,\epsilon}\right\rangle \right\vert ^{2}=\left\vert \det\left(
\mathcal{L}^{\epsilon}\right)  ^{\top}\right\vert \int_{\digamma_{\epsilon}%
}\left\vert \mathbf{h}\left(  \mathbf{e}\left(  A\right)  \right)  \right\vert
^{2}\left[  \left\vert r\left(  \mathbf{e}\left(  A\right)  \right)
\right\vert ^{2}\rho\left(  A\right)  ^{2}\Theta_{\lambda}\left(
\beta_{\boldsymbol{J}\left(  \lambda\right)  }\left(  A\right)  \right)
\right]  dA.
\]
Next,
\[
\left\vert r\left(  \mathbf{e}\left(  A\right)  \right)  \right\vert ^{2}%
\rho\left(  A\right)  ^{2}\Theta_{\lambda}\left(  \beta_{\boldsymbol{J}\left(
\lambda\right)  }\left(  A\right)  \right)  =\frac{\rho\left(  A\right)
^{2}\Theta_{\lambda}\left(  \beta_{\boldsymbol{J}\left(  \lambda\right)
}\left(  A\right)  \right)  }{\rho\left(  A\right)  \Theta_{\lambda}\left(
\beta_{\boldsymbol{J}\left(  \lambda\right)  }\left(  A\right)  \right)
}=\rho\left(  A\right)  .
\]
Finally%
\begin{align*}%
{\displaystyle\sum\limits_{\alpha\in\Gamma_{P}^{\epsilon}}}
\left\vert \left\langle \mathbf{h},\pi_{\lambda}\left(  \alpha\right)
\mathbf{f}_{r,\epsilon}\right\rangle \right\vert ^{2}  &  =\left\vert
\det\left(  \mathcal{L}^{\epsilon}\right)  ^{\top}\right\vert \int
_{\digamma_{\epsilon}}\left\vert \mathbf{h}\left(  \mathbf{e}\left(  A\right)
\right)  \right\vert ^{2}\underset{=\rho\left(  A\right)  }{\underbrace
{\left\vert r\left(  \mathbf{e}\left(  A\right)  \right)  \right\vert ^{2}%
\rho\left(  A\right)  ^{2}\Theta_{\lambda}\left(  \beta_{\boldsymbol{J}\left(
\lambda\right)  }\left(  A\right)  \right)  }}dA\\
&  =\left\vert \det\left(  \mathcal{L}^{\epsilon}\right)  ^{\top}\right\vert
\int_{\digamma_{\epsilon}}\left\vert \mathbf{h}\left(  \mathbf{e}\left(
A\right)  \right)  \right\vert ^{2}\rho\left(  A\right)  dA\\
&  =\left\vert \det\left(  \mathcal{L}^{\epsilon}\right)  ^{\top}\right\vert
\times\left\Vert \mathbf{h}\right\Vert _{L^{2}\left(  \Omega_{\epsilon}%
,d\mu_{M}\right)  }^{2}.
\end{align*}

\end{proof}

Since $L^{2}\left(  \Omega_{\epsilon},d\mu_{M}\right)  \cap C\left(
\Omega_{\epsilon}\right)  $ is dense in $L^{2}\left(  \Omega_{\epsilon}%
,d\mu_{M}\right)  $ it follows that if $\mathbf{z}\in L^{2}\left(
\Omega_{\epsilon},d\mu_{M}\right)  $ and
\[
r\left(  \mathbf{e}\left(  A\right)  \right)  =\frac{1}{\sqrt{\rho\left(
A\right)  \Theta_{\lambda}\left(  \beta_{\boldsymbol{J}\left(  \lambda\right)
}\left(  A\right)  \right)  \text{ }}}%
\]
then
\[%
{\displaystyle\sum\limits_{\alpha\in\Gamma_{P}^{\epsilon}}}
\left\vert \left\langle \mathbf{z},\pi_{\lambda}\left(  \alpha\right)
\mathbf{f}_{r,\epsilon}\right\rangle _{L^{2}\left(  \Omega_{\epsilon},d\mu
_{M}\right)  }\right\vert ^{2}=\left\vert \det\left(  \mathcal{L}^{\epsilon
}\right)  ^{\top}\right\vert \times\left\Vert \mathbf{z}\right\Vert
_{L^{2}\left(  \Omega_{\epsilon},d\mu_{M}\right)  }^{2}.
\]
Next, we recall that
\[
\left\{  \gamma^{-1}\Omega_{\epsilon}:\gamma\in\Gamma_{M\text{ }}^{\epsilon
}=\exp\left(  \epsilon%
\mathbb{Z}
A_{1}\right)  \exp\left(  \epsilon%
\mathbb{Z}
A_{2}\right)  \cdots\exp\left(  \epsilon%
\mathbb{Z}
A_{n_{2}}\right)  \right\}
\]
is a partition of $M.$ Finally, put
\[
\Gamma^{\epsilon}=\left[  \Gamma_{M\text{ }}^{\epsilon}\right]  ^{-1}%
\Gamma_{P}^{\epsilon}.
\]
Given $r$ such that
\[
r\left(  \mathbf{e}\left(  A\right)  \right)  =\frac{1}{\sqrt{\rho\left(
A\right)  \Theta_{\lambda}\left(  \beta_{\boldsymbol{J}\left(  \lambda\right)
}\left(  A\right)  \right)  }}%
\]
it follows that
\[
\left\{  \pi_{\lambda}\left(  \kappa\right)  \left(  r\times\mathbf{1}%
_{\Omega_{\epsilon}}\right)  :\kappa\in\Gamma^{\epsilon}\right\}
\]
is a tight frame for $L^{2}\left(  M,\mu_{M}\right)  $ with frame bound
$\left\vert \det\left(  \mathcal{L}^{\epsilon}\right)  ^{\top}\right\vert .$

\subsection{Proof of Theorem \ref{smooth frames}}

Fix $\epsilon\in\mathbf{L}$ and define%
\[
\Omega_{\epsilon}^{\circ}=\exp\left(  \left(  -\frac{\epsilon}{2}%
,\frac{\epsilon}{2}\right)  A_{1}\right)  \cdots\exp\left(  \left(
-\frac{\epsilon}{2},\frac{\epsilon}{2}\right)  A_{n_{2}}\right)  \subset M
\]
such that
\[
\mathcal{O}_{\epsilon}=\left(  -\frac{\epsilon}{2},\frac{\epsilon}{2}\right)
A_{1}+\cdots+\left(  -\frac{\epsilon}{2},\frac{\epsilon}{2}\right)  A_{n_{2}%
}\subset\mathfrak{m}%
\]
is an open set around the neutral element in $\mathfrak{m}$. By assumption,
the restriction of $\beta_{\boldsymbol{J}\left(  \lambda\right)  }$ to
$\mathcal{O}_{\epsilon}$ defines a diffeomorphism between $\mathcal{O}%
_{\epsilon}$ and $\beta_{\boldsymbol{J}\left(  \lambda\right)  }\left(
\mathcal{O}_{\epsilon}\right)  .$ Next, we define the map
\[
\Phi_{\boldsymbol{J}\left(  \lambda\right)  }^{_{\epsilon}}:\Omega_{\epsilon
}^{\circ}\rightarrow\beta_{\boldsymbol{J}\left(  \lambda\right)  }\left(
\mathcal{O}_{\epsilon}\right)
\]
such that
\[
\Phi_{\boldsymbol{J}\left(  \lambda\right)  }^{_{\epsilon}}\left(
\mathbf{e}\left(
{\displaystyle\sum\limits_{k=1}^{n_{2}}}
a_{k}A_{k}\right)  \right)  =\beta_{\boldsymbol{J}\left(  \lambda\right)
}\left(
{\displaystyle\sum\limits_{k=1}^{n_{2}}}
a_{k}A_{k}\right)  .
\]
Thus,
\[
\Phi_{\boldsymbol{J}\left(  \lambda\right)  }^{_{\epsilon}}\circ
\mathbf{e}=\beta_{\boldsymbol{J}\left(  \lambda\right)  }%
\]
is a diffeomorphism (it is a composition of two diffeomorphisms:
$\beta_{\boldsymbol{J}\left(  \lambda\right)  }$ and $\mathbf{e}^{-1}.$) Let
$\mathbf{s}\in C_{c}^{\infty}\left(  M\right)  $ such that the support
$\mathrm{Supp}\left(  \mathbf{s}\right)  $ of $\mathbf{s}$ is a compact subset
of $\Omega_{\epsilon}^{\circ}.$ Put
\[
\Sigma_{\mathbf{s}}=\Phi_{\boldsymbol{J}\left(  \lambda\right)  }^{_{\epsilon
}}\left(  \mathrm{Supp}\left(  \mathbf{s}\right)  \right)  \subset
\beta_{\boldsymbol{J}\left(  \lambda\right)  }\left(  \mathcal{O}_{\epsilon
}\right)  .
\]
Since $\Phi_{\boldsymbol{J}\left(  \lambda\right)  }^{_{\epsilon}}$ is a
continuous function, it is clear that $\Sigma_{\mathbf{s}}$ is a compact
subset of $\beta_{\boldsymbol{J}\left(  \lambda\right)  }\left(
\mathcal{O}_{\epsilon}\right)  .$

\begin{lemma}
There exists an invertible linear operator $\mathcal{L}^{\epsilon}:\sum
_{j\in\boldsymbol{J}\left(  \lambda\right)  }%
\mathbb{R}
X_{j}\rightarrow\sum_{j\in\boldsymbol{J}\left(  \lambda\right)  }%
\mathbb{R}
X_{j}$ such that $\Sigma_{\mathbf{s}}$ is contained in a fundamental domain of
the lattice
\[
\Lambda_{\epsilon}^{\star}=\sum_{j\in\boldsymbol{J}\left(  \lambda\right)  }%
\mathbb{Z}
\left(  \mathcal{L}^{\epsilon}\right)  ^{\top}X_{j}^{\ast}\subset
\mathfrak{p}^{\ast}.
\]

\end{lemma}

\begin{proof}
Since $\Sigma_{\mathbf{s}}$ is a compact subset of $\sum_{j\in\boldsymbol{J}%
\left(  \lambda\right)  }%
\mathbb{R}
X_{j}$ of positive Lebesgue measure, there exists a positive real number
$\delta_{\epsilon}$ such that
\[
\Sigma_{\mathbf{s}}\subset\sum_{j\in\boldsymbol{J}\left(  \lambda\right)
}\left[  -\frac{\delta_{\epsilon}}{2},\frac{\delta_{\epsilon}}{2}\right)
X_{j}^{\ast}.
\]
Next, the desired result holds by letting $\Lambda_{\epsilon}^{\star}$ be
equal to $\sum_{j\in\boldsymbol{J}\left(  \lambda\right)  }%
\mathbb{Z}
\delta_{\epsilon}X_{j}^{\ast}.$
\end{proof}

Fix $\mathcal{L}^{\epsilon}$ such that $\Sigma_{\mathbf{s}}$ is contained in a
fundamental domain of the lattice
\[
\Lambda_{\epsilon}^{\star}=\sum_{j\in\boldsymbol{J}\left(  \lambda\right)  }%
\mathbb{Z}
\left(  \mathcal{L}^{\epsilon}\right)  ^{\top}X_{j}^{\ast}\subset
\mathfrak{p}^{\ast}.
\]
Put $\Lambda_{\epsilon}=\sum_{j\in\boldsymbol{J}\left(  \lambda\right)  }%
\mathbb{Z}
\mathcal{L}^{\epsilon}X_{j}$ and let $\Gamma_{P}^{\epsilon}=\exp\left(
\Lambda_{\epsilon}\right)  \subset P.$ Recall that
\[
w\left(  A\right)  =\left\vert \det\left(  \frac{id-e^{-\mathrm{ad}\left(
A\right)  }}{\mathrm{ad}\left(  A\right)  }\right)  \right\vert .
\]
Put
\[
\mathbf{t}\left(  \xi\right)  =\nu\left(  \mathbf{e}^{-1}\left(  \left[
\Phi_{\boldsymbol{J}\left(  \lambda\right)  }^{_{\epsilon}}\right]
^{-1}\left(  \xi\right)  \right)  \right)  =\left(  \mathbf{t}_{1}\left(
\xi\right)  ,\cdots,\mathbf{t}_{n_{2}}\left(  \xi\right)  \right)
\]
and
\begin{equation}
\mathbf{d}\left(  \xi\right)  =\left\vert \det\left[
\begin{array}
[c]{ccc}%
\dfrac{\partial\left(  \mathbf{t}_{1}\left(  \xi\right)  \right)  }%
{\partial\xi_{j_{1}}} & \cdots & \dfrac{\partial\left(  \mathbf{t}_{1}\left(
\xi\right)  \right)  }{\partial\xi_{j_{n_{2}}}}\\
\vdots & \ddots & \vdots\\
\dfrac{\partial\left(  \mathbf{t}_{n_{2}}\left(  \xi\right)  \right)
}{\partial\xi_{j_{1}}} & \cdots & \dfrac{\partial\left(  \mathbf{t}_{n_{2}%
}\left(  \xi\right)  \right)  }{\partial\xi_{j_{n_{2}}}}%
\end{array}
\right]  \right\vert . \label{dfunc}%
\end{equation}

\begin{lemma}
If $\Psi_{\boldsymbol{J}\left(  \lambda\right)  }^{\epsilon}:\beta
_{\boldsymbol{J}\left(  \lambda\right)  }\left(  \mathcal{O}_{\epsilon
}\right)  \rightarrow\left(  0,\infty\right)  $ is a positive function given
by $\Psi_{\boldsymbol{J}\left(  \lambda\right)  }^{\epsilon}\left(
\xi\right)  d\xi=d\mu_{M}\left(  \left[  \Phi_{\boldsymbol{J}\left(
\lambda\right)  }^{_{\epsilon}}\right]  ^{-1}\left(  \xi\right)  \right)
\ $then%
\[
\Psi_{\boldsymbol{J}\left(  \lambda\right)  }^{\epsilon}\left(  \xi\right)
=w\left(  \nu\left(  \mathbf{e}^{-1}\left(  \left[  \Phi_{\boldsymbol{J}%
\left(  \lambda\right)  }^{_{\epsilon}}\right]  ^{-1}\left(  \xi\right)
\right)  \right)  \right)  \mathbf{d}\left(  \xi\right)  .
\]

\end{lemma}

\begin{proof}
For a suitable positive function $\mathbf{f}$ such that $\mathbf{f}\circ
\Phi_{\boldsymbol{J}\left(  \lambda\right)  }^{_{\epsilon}}\in L^{1}\left(
\Omega_{\epsilon}^{\circ},d\mu_{M}\left(  m\right)  \right)  ,$
\begin{align*}
&  \int_{\Omega_{\epsilon}^{\circ}}\mathbf{f}\left(  \Phi_{\boldsymbol{J}%
\left(  \lambda\right)  }^{_{\epsilon}}\left(  m\right)  \right)  \text{ }%
d\mu_{M}\left(  m\right)  \\
&  =\int_{\mathcal{O}_{\epsilon}}\mathbf{f}\left(  \Phi_{\boldsymbol{J}\left(
\lambda\right)  }^{_{\epsilon}}\left(  \mathbf{e}\left(  A\right)  \right)
\right)  \text{ }d\mu_{M}\left(  \mathbf{e}\left(  A\right)  \right)  \\
&  =\int_{\mathcal{O}_{\epsilon}}\mathbf{f}\left(  \Phi_{\boldsymbol{J}\left(
\lambda\right)  }^{_{\epsilon}}\left(  \exp\left(  \nu\left(  A\right)
\right)  \right)  \right)  \text{ }d\mu_{M}\left(  \exp\left(  \nu\left(
A\right)  \right)  \right)  \\
&  =\int_{\nu\left(  \mathcal{O}_{\epsilon}\right)  }\mathbf{f}\left(
\Phi_{\boldsymbol{J}\left(  \lambda\right)  }^{_{\epsilon}}\left(  \exp
A\right)  \right)  \text{ }d\mu_{M}\left(  \exp\left(  A\right)  \right)  \\
&  =\int_{\nu\left(  \mathcal{O}_{\epsilon}\right)  }\mathbf{f}\left(
\Phi_{\boldsymbol{J}\left(  \lambda\right)  }^{_{\epsilon}}\left(  \exp
A\right)  \right)  \text{ }\left\vert \det\left(  \frac{id-e^{-\mathrm{ad}%
\left(  A\right)  }}{\mathrm{ad}\left(  A\right)  }\right)  \right\vert dA.
\end{align*}
The first and second equality above follows from the change of variables
$m=\mathbf{e}\left(  A\right)  ,$ and $\mathbf{e}\left(  A\right)
=\exp\left(  \nu\left(  A\right)  \right)  $ respectively. Next,
\begin{align*}
&  \int_{\Omega_{\epsilon}^{\circ}}\mathbf{f}\left(  \Phi_{\boldsymbol{J}%
\left(  \lambda\right)  }^{_{\epsilon}}\left(  m\right)  \right)  \text{ }%
d\mu_{M}\left(  m\right)  \\
&  =\int_{\nu\left(  \mathcal{O}_{\epsilon}\right)  }\mathbf{f}\left(
\Phi_{\boldsymbol{J}\left(  \lambda\right)  }^{_{\epsilon}}\left(  \exp
A\right)  \right)  \text{ }w\left(  A\right)  dA\\
&  =\int_{\mathcal{O}_{\epsilon}}\mathbf{f}\left(  \Phi_{\boldsymbol{J}\left(
\lambda\right)  }^{_{\epsilon}}\left(  \exp\left(  \nu\left(  A\right)
\right)  \right)  \right)  \text{ }w\left(  \nu\left(  A\right)  \right)
d\left(  \nu\left(  A\right)  \right)  \\
&  =\int_{\mathcal{O}_{\epsilon}}\mathbf{f}\left(  \Phi_{\boldsymbol{J}\left(
\lambda\right)  }^{_{\epsilon}}\left(  \mathbf{e}\left(  A\right)  \right)
\right)  \text{ }w\left(  \nu\left(  A\right)  \right)  d\left(  \nu\left(
A\right)  \right)  \\
&  =\int_{\Omega_{\epsilon}^{\circ}}\mathbf{f}\left(  \Phi_{\boldsymbol{J}%
\left(  \lambda\right)  }^{_{\epsilon}}\left(  m\right)  \right)  \text{
}w\left(  \nu\left(  \mathbf{e}^{-1}\left(  m\right)  \right)  \right)
d\left(  \nu\left(  \mathbf{e}^{-1}\left(  m\right)  \right)  \right)  .
\end{align*}
Finally, the change of variable $\xi=\Phi_{\boldsymbol{J}\left(
\lambda\right)  }^{_{\epsilon}}\left(  m\right)  $ yields
\begin{align*}
&  \int_{\Omega_{\epsilon}^{\circ}}\mathbf{f}\left(  \Phi_{\boldsymbol{J}%
\left(  \lambda\right)  }^{_{\epsilon}}\left(  m\right)  \right)  \text{ }%
d\mu_{M}\left(  m\right)  \\
&  =\int_{\beta_{\boldsymbol{J}\left(  \lambda\right)  }\left(  \mathcal{O}%
_{\epsilon}\right)  }\mathbf{f}\left(  \xi\right)  \text{ }w\left(  \nu\left(
\mathbf{e}^{-1}\left(  \left[  \Phi_{\boldsymbol{J}\left(  \lambda\right)
}^{_{\epsilon}}\right]  ^{-1}\left(  \xi\right)  \right)  \right)  \right)
d\left(  \nu\left(  \mathbf{e}^{-1}\left(  \left[  \Phi_{\boldsymbol{J}\left(
\lambda\right)  }^{_{\epsilon}}\right]  ^{-1}\left(  \xi\right)  \right)
\right)  \right)  \\
&  =\int_{\beta_{\boldsymbol{J}\left(  \lambda\right)  }\left(  \mathcal{O}%
_{\epsilon}\right)  }\mathbf{f}\left(  \xi\right)  \text{ }\underset
{=\Psi_{\boldsymbol{J}\left(  \lambda\right)  }^{\epsilon}\left(  \xi\right)
d\xi}{\underbrace{w\left(  \nu\left(  \mathbf{e}^{-1}\left(  \left[
\Phi_{\boldsymbol{J}\left(  \lambda\right)  }^{_{\epsilon}}\right]
^{-1}\left(  \xi\right)  \right)  \right)  \right)  d\left(  \nu\left(
\mathbf{e}^{-1}\left(  \left[  \Phi_{\boldsymbol{J}\left(  \lambda\right)
}^{_{\epsilon}}\right]  ^{-1}\left(  \xi\right)  \right)  \right)  \right)  }}%
\end{align*}
From the definition of $\mathbf{d}$, it follows that
\[
\Psi_{\boldsymbol{J}\left(  \lambda\right)  }^{\epsilon}\left(  \xi\right)
=w\left(  \nu\left(  \mathbf{e}^{-1}\left(  \left[  \Phi_{\boldsymbol{J}%
\left(  \lambda\right)  }^{_{\epsilon}}\right]  ^{-1}\left(  \xi\right)
\right)  \right)  \right)  \mathbf{d}\left(  \xi\right)  .
\]

\end{proof}

\begin{proposition}
Let $\mathfrak{k}^{\ast}$ be a measurable subset of $\beta_{\boldsymbol{J}%
\left(  \lambda\right)  }\left(  \mathcal{O}_{\epsilon}\right)  .$ Then
\[
1_{\mathfrak{k}^{\ast}}\in L^{1}\left(  \mathcal{O}_{\epsilon},\Psi
_{\boldsymbol{J}\left(  \lambda\right)  }^{\epsilon}\left(  \xi\right)
d\xi\right)  \Leftrightarrow1_{\left(  \Phi_{\boldsymbol{J}\left(
\lambda\right)  }^{_{\epsilon}}\right)  ^{-1}\left(  \mathfrak{k}^{\ast
}\right)  }\in L^{1}\left(  \Omega_{\epsilon}^{\circ},d\mu_{M}\right)
\]

\end{proposition}

\begin{proof}
First,
\[
\int_{\mathfrak{k}^{\ast}}\Psi_{\boldsymbol{J}\left(  \lambda\right)
}^{\epsilon}\left(  \xi\right)  d\xi=\int_{\mathfrak{k}^{\ast}}d\mu_{M}\left(
\left(  \Phi_{\boldsymbol{J}\left(  \lambda\right)  }^{_{\epsilon}}\right)
^{-1}\left(  \xi\right)  \right)  .
\]
Second, the change of variable $m=\left(  \Phi_{\boldsymbol{J}\left(
\lambda\right)  }^{_{\epsilon}}\right)  ^{-1}\left(  \xi\right)  $ yields%
\begin{align*}
\left\Vert 1_{\mathfrak{k}^{\ast}}\right\Vert _{L^{1}\left(  \mathcal{O}%
_{\epsilon},\Psi_{\boldsymbol{J}\left(  \lambda\right)  }^{\epsilon}\left(
\xi\right)  d\xi\right)  }  &  =\int_{\mathfrak{k}^{\ast}}\Psi_{\boldsymbol{J}%
\left(  \lambda\right)  }^{\epsilon}\left(  \xi\right)  d\xi\\
&  =\int_{\left(  \Phi_{\boldsymbol{J}\left(  \lambda\right)  }^{_{\epsilon}%
}\right)  ^{-1}\left(  \mathfrak{k}^{\ast}\right)  }d\mu_{M}\left(  m\right)
\\
&  =\left\Vert 1_{\left(  \Phi_{\boldsymbol{J}\left(  \lambda\right)
}^{_{\epsilon}}\right)  ^{-1}\left(  \mathfrak{k}^{\ast}\right)  }\right\Vert
_{L^{1}\left(  \Omega_{\epsilon}^{\circ},d\mu_{M}\right)  }.
\end{align*}

\end{proof}

\begin{example}
Let $G=\mathbb{R}\rtimes e^{\mathbb{R}}$ be a simply connected, connected
completely solvable Lie group with multiplication law
\[
\left(  x,e^{t}\right)  \left(  y,e^{s}\right)  =\left(  x+e^{t}%
y,e^{t+s}\right)  .
\]
Given a linear functional $\lambda=X_{1}^{\ast}$, we define the unitary
representation $\pi_{\lambda}$ as acting on $L^{2}\left(  \left(
0,\infty\right)  ,\frac{dh}{h}\right)  $ such that
\[
\left[  \pi_{\lambda}\left(  x,e^{t}\right)  \mathbf{f}\right]  \left(
h\right)  =e^{\frac{2\pi ix}{h}}\mathbf{f}\left(  \frac{h}{e^{t}}\right)  .
\]
Fix $\epsilon=2\ln2$ and define%
\[
\Omega_{2\ln2}^{\circ}=\left(  e^{-\ln2},e^{\ln2}\right)  =\left(  \frac{1}%
{2},2\right)
\]
and $\Phi_{\boldsymbol{J}\left(  \lambda\right)  }^{_{2\ln2}}\left(  h\right)
=h^{-1}.$ Next, let $\mathbf{s}\in C_{c}^{\infty}\left(  \left(
0,\infty\right)  \right)  $ such that the support of $\mathbf{s}$ is a compact
subset of $\left(  \frac{1}{2},2\right)  .$ Then
\[
\Sigma_{\mathbf{s}}=\Phi_{\boldsymbol{J}\left(  \lambda\right)  }^{_{\epsilon
}}\left(  \mathrm{Supp}\left(  \mathbf{s}\right)  \right)  \subset\left(
\frac{1}{2},2\right)
\]
and $\Sigma_{\mathbf{s}}$ is contained in a fundamental domain of
$\Lambda_{2\ln2}^{\star}=\frac{3}{2}%
\mathbb{Z}
$. Next, for a suitable function $\mathbf{f,}$
\[
\int_{\Omega_{\epsilon}^{\circ}}\mathbf{f}\left(  \Phi_{\boldsymbol{J}\left(
\lambda\right)  }^{_{\epsilon}}\left(  m\right)  \right)  \text{ }d\mu
_{M}\left(  m\right)  =\int_{\left(  \frac{1}{2},2\right)  }\mathbf{f}\left(
\frac{1}{h}\right)  \text{ }\frac{dh}{h}=\int_{\left(  \frac{1}{2},2\right)
}\mathbf{f}\left(  \xi\right)  \text{ }\frac{d\xi}{\xi}.
\]
Thus
\[
\Psi_{\boldsymbol{J}\left(  \lambda\right)  }^{2\ln2}\left(  \xi\right)
=\xi^{-1}%
\]
and
\[
\Psi_{\boldsymbol{J}\left(  \lambda\right)  }^{2\ln2}\left(  \Phi
_{\boldsymbol{J}\left(  \lambda\right)  }^{_{2\ln2}}\left(  h\right)  \right)
=\Psi_{\boldsymbol{J}\left(  \lambda\right)  }^{2\ln2}\left(  h^{-1}\right)
=h.
\]

\end{example}

\begin{proposition}
\label{integrable}If $\mathbf{K}$ is a compact subset of $\Omega_{\epsilon
}^{\circ}$ then
\[
m\mapsto\left[  \Psi_{\boldsymbol{J}\left(  \lambda\right)  }^{\epsilon
}\left(  \Phi_{\boldsymbol{J}\left(  \lambda\right)  }^{_{\epsilon}}\left(
m\right)  \right)  \right]  ^{-1}\in L^{1}\left(  \mathbf{K},d\mu_{M}\right)
.
\]

\end{proposition}

\begin{proof}
Let $\mathbf{K}$ be a compact subset of $\Omega_{\epsilon}^{\circ}.$ In order
to prove this result, it suffices to establish that
\[
\int_{\mathbf{K}}\frac{d\mu_{M}\left(  m\right)  }{\Psi_{\boldsymbol{J}\left(
\lambda\right)  }^{\epsilon}\left(  \Phi_{\boldsymbol{J}\left(  \lambda
\right)  }^{_{\epsilon}}\left(  m\right)  \right)  }=\int_{\Phi
_{\boldsymbol{J}\left(  \lambda\right)  }^{_{\epsilon}}\left(  \mathbf{K}%
\right)  }d\xi.
\]
The change of variable $\xi=\Phi_{\boldsymbol{J}\left(  \lambda\right)
}^{_{\epsilon}}\left(  m\right)  $ yields%
\begin{align*}
\int_{\mathbf{K}}\frac{d\mu_{M}\left(  m\right)  }{\Psi_{\boldsymbol{J}\left(
\lambda\right)  }^{\epsilon}\left(  \Phi_{\boldsymbol{J}\left(  \lambda
\right)  }^{_{\epsilon}}\left(  m\right)  \right)  }  &  =\int_{\Phi
_{\boldsymbol{J}\left(  \lambda\right)  }^{_{\epsilon}}\left(  \mathbf{K}%
\right)  }\frac{d\mu_{M}\left(  \left(  \Phi_{\boldsymbol{J}\left(
\lambda\right)  }^{_{\epsilon}}\right)  ^{-1}\left(  \xi\right)  \right)
}{\Psi_{\boldsymbol{J}\left(  \lambda\right)  }^{\epsilon}\left(  \xi\right)
}\\
&  =\int_{\Phi_{\boldsymbol{J}\left(  \lambda\right)  }^{_{\epsilon}}\left(
\mathbf{K}\right)  }\frac{1}{\Psi_{\boldsymbol{J}\left(  \lambda\right)
}^{\epsilon}\left(  \xi\right)  }\Psi_{\boldsymbol{J}\left(  \lambda\right)
}^{\epsilon}\left(  \xi\right)  d\xi\\
&  =\int_{\Phi_{\boldsymbol{J}\left(  \lambda\right)  }^{_{\epsilon}}\left(
\mathbf{K}\right)  }d\xi.
\end{align*}
Since $\Phi_{\boldsymbol{J}\left(  \lambda\right)  }^{_{\epsilon}}$ is
continuous and $\mathbf{K}$ is compact, we obtain
\[
\int_{\mathbf{K}}\frac{d\mu_{M}\left(  m\right)  }{\Psi_{\boldsymbol{J}\left(
\lambda\right)  }^{\epsilon}\left(  \Phi_{\boldsymbol{J}\left(  \lambda
\right)  }^{_{\epsilon}}\left(  m\right)  \right)  }=\left\vert \Phi
_{\boldsymbol{J}\left(  \lambda\right)  }^{_{\epsilon}}\left(  \mathbf{K}%
\right)  \right\vert <\infty.
\]

\end{proof}

\begin{lemma}
Let $\mathbf{s}\in C_{c}^{\infty}\left(  M\right)  $ such that the support
$\mathrm{Supp}\left(  \mathbf{s}\right)  $ of $\mathbf{s}$ is a compact subset
of $\Omega_{\epsilon}^{\circ}.$ Then $\mathbf{s}$ is square-integrable with
respect to the Haar measure $d\mu_{M}\left(  m\right)  .$
\end{lemma}

\begin{proof}
Recall that $\nu,\mathbf{e,}\Phi_{\boldsymbol{J}\left(  \lambda\right)
}^{_{\epsilon}}$ are smooth bijective functions and
\[
\Psi_{\boldsymbol{J}\left(  \lambda\right)  }^{\epsilon}\left(  \xi\right)
=w\left(  \nu\left(  \mathbf{e}^{-1}\left(  \left[  \Phi_{\boldsymbol{J}%
\left(  \lambda\right)  }^{_{\epsilon}}\right]  ^{-1}\left(  \xi\right)
\right)  \right)  \right)  \mathbf{d}\left(  \xi\right)  .
\]
As such, $\Psi_{\boldsymbol{J}\left(  \lambda\right)  }^{\epsilon}$ is a
positive smooth function defined on
\[
\beta_{\boldsymbol{J}\left(  \lambda\right)  }\left(  \mathcal{O}_{\epsilon
}\right)  \supset\Sigma_{\mathbf{s}}.
\]
Next,
\begin{align*}
\int_{M}\left\vert \mathbf{s}\left(  m\right)  \right\vert ^{2}d\mu_{M}\left(
m\right)   &  =\int_{\Sigma_{\mathbf{s}}}\left\vert \mathbf{s}\left(  \left[
\Phi_{\boldsymbol{J}\left(  \lambda\right)  }^{_{\epsilon}}\right]
^{-1}\left(  \xi\right)  \right)  \right\vert ^{2}d\mu_{M}\left(  \left[
\Phi_{\boldsymbol{J}\left(  \lambda\right)  }^{_{\epsilon}}\right]
^{-1}\left(  \xi\right)  \right) \\
&  =\int_{\Sigma_{\mathbf{s}}}\left\vert \mathbf{s}\left(  \left[
\Phi_{\boldsymbol{J}\left(  \lambda\right)  }^{_{\epsilon}}\right]
^{-1}\left(  \xi\right)  \right)  \right\vert ^{2}\Psi_{\boldsymbol{J}\left(
\lambda\right)  }^{\epsilon}\left(  \xi\right)  d\xi\\
&  =\int_{\Sigma_{\mathbf{s}}}\left\vert \mathbf{s}\left(  \left[
\Phi_{\boldsymbol{J}\left(  \lambda\right)  }^{_{\epsilon}}\right]
^{-1}\left(  \xi\right)  \right)  \left(  \Psi_{\boldsymbol{J}\left(
\lambda\right)  }^{\epsilon}\left(  \xi\right)  \right)  ^{1/2}\right\vert
^{2}d\xi\\
&  \leq\int_{\Sigma_{\mathbf{s}}}\left\Vert \xi\mapsto\mathbf{s}\left(
\left[  \Phi_{\boldsymbol{J}\left(  \lambda\right)  }^{_{\epsilon}}\right]
^{-1}\left(  \xi\right)  \right)  \left(  \Psi_{\boldsymbol{J}\left(
\lambda\right)  }^{\epsilon}\left(  \xi\right)  \right)  ^{1/2}\right\Vert
_{L^{\infty}\left(  \Sigma_{\mathbf{s}}\right)  }^{2}d\xi\\
&  =\left\Vert \xi\mapsto\mathbf{s}\left(  \left[  \Phi_{\boldsymbol{J}\left(
\lambda\right)  }^{_{\epsilon}}\right]  ^{-1}\left(  \xi\right)  \right)
\left(  \Psi_{\boldsymbol{J}\left(  \lambda\right)  }^{\epsilon}\left(
\xi\right)  \right)  ^{1/2}\right\Vert _{L^{\infty}\left(  \Sigma_{\mathbf{s}%
}\right)  }^{2}\left\vert \Sigma_{\mathbf{s}}\right\vert .
\end{align*}
The compactness of $\Sigma_{\mathbf{s}}$ together with the fact that the
function
\[
\xi\mapsto\mathbf{s}\left(  \left[  \Phi_{\boldsymbol{J}\left(  \lambda
\right)  }^{_{\epsilon}}\right]  ^{-1}\left(  \xi\right)  \right)  \left(
\Psi_{\boldsymbol{J}\left(  \lambda\right)  }^{\epsilon}\left(  \xi\right)
\right)  ^{1/2}%
\]
is uniformly continuous on $\Sigma_{\mathbf{s}}$ gives the desired result.
\end{proof}

\begin{lemma}
Let $\mathbf{s}_{1}$ be defined such that
\begin{equation}
\mathbf{s}_{1}\left(  m\right)  =\mathbf{s}\left(  m\right)  \sqrt
{\Psi_{\boldsymbol{J}\left(  \lambda\right)  }^{\epsilon}\left(
\Phi_{\boldsymbol{J}\left(  \lambda\right)  }^{_{\epsilon}}\left(  m\right)
\right)  }.
\end{equation}
Then $\mathbf{s}_{1}$ is a smooth function of compact support.
\end{lemma}

\begin{proof}
Since
\[
\Psi_{\boldsymbol{J}\left(  \lambda\right)  }^{\epsilon}\left(  \xi\right)
=w\left(  \nu\left(  \mathbf{e}^{-1}\left(  \left[  \Phi_{\boldsymbol{J}%
\left(  \lambda\right)  }^{_{\epsilon}}\right]  ^{-1}\left(  \xi\right)
\right)  \right)  \right)
\]
it is clear that
\[
\Psi_{\boldsymbol{J}\left(  \lambda\right)  }^{\epsilon}:\beta_{\boldsymbol{J}%
\left(  \lambda\right)  }\left(  \mathcal{O}_{\epsilon}\right)  \rightarrow
\left(  0,\infty\right)
\]
is a smooth positive function defined on $\beta_{\boldsymbol{J}\left(
\lambda\right)  }\left(  \mathcal{O}_{\epsilon}\right)  .$ Next, since
$\mathbf{s}_{1}$ is the product of the smooth functions $\sqrt{\Psi
_{\boldsymbol{J}\left(  \lambda\right)  }^{\epsilon}\left(  \Phi
_{\boldsymbol{J}\left(  \lambda\right)  }^{_{\epsilon}}\left(  m\right)
\right)  }$ and $\mathbf{s}\left(  m\right)  ,$ and since $\sqrt
{\Psi_{\boldsymbol{J}\left(  \lambda\right)  }^{\epsilon}\left(
\Phi_{\boldsymbol{J}\left(  \lambda\right)  }^{_{\epsilon}}\left(  m\right)
\right)  }$ is a positive function, it follows that $\mathbf{s}_{1}$ is a
smooth function which is supported on $\left[  \Phi_{\boldsymbol{J}\left(
\lambda\right)  }^{_{\epsilon}}\right]  ^{-1}\left(  \Sigma_{\mathbf{s}%
}\right)  .$
\end{proof}

Next, let $\Gamma_{M\text{ }}$ be a discrete subset of $M$. Next, let
$\mathbf{h}$ be a continuous function which is also square-integrable with
respect to the Haar measure. Then
\begin{align*}
&
{\displaystyle\sum\limits_{\exp\left(  X\right)  \in\Gamma_{P}^{\epsilon}}}
{\displaystyle\sum\limits_{\gamma\in\Gamma_{M\text{ }}}}
\left\vert \left\langle \mathbf{h},\pi_{\lambda}\left(  \gamma^{-1}\right)
\pi_{\lambda}\left(  \exp X\right)  \mathbf{s}\right\rangle \right\vert ^{2}\\
&  =%
{\displaystyle\sum\limits_{\exp\left(  X\right)  \in\Gamma_{P}^{\epsilon}}}
{\displaystyle\sum\limits_{\gamma\in\Gamma_{M\text{ }}}}
\left\vert \int_{M}\mathbf{h}\left(  m\right)  e^{-2\pi i\left\langle \left(
\left[  Ad\left(  \gamma m\right)  \right]  ^{-1}\right)  ^{\ast}%
\lambda,X\right\rangle }\overline{\mathbf{s}\left(  \gamma m\right)  }d\mu
_{M}\left(  m\right)  \right\vert ^{2}.
\end{align*}
The change of variable $n=\gamma m$ yields
\begin{align*}
&
{\displaystyle\sum\limits_{\exp\left(  X\right)  \in\Gamma_{P}^{\epsilon}}}
{\displaystyle\sum\limits_{\gamma\in\Gamma_{M\text{ }}}}
\left\vert \int_{M}\mathbf{h}\left(  m\right)  e^{-2\pi i\left\langle \left(
\left[  Ad\left(  \gamma m\right)  \right]  ^{-1}\right)  ^{\ast}%
\lambda,X\right\rangle }\overline{\mathbf{s}\left(  \gamma m\right)  }d\mu
_{M}\left(  m\right)  \right\vert ^{2}\\
&  =%
{\displaystyle\sum\limits_{\exp\left(  X\right)  \in\Gamma_{P}^{\epsilon}}}
{\displaystyle\sum\limits_{\gamma\in\Gamma_{M\text{ }}}}
\left\vert \int_{M}\mathbf{h}\left(  \gamma^{-1}n\right)  e^{-2\pi
i\left\langle \left(  \left[  Ad\left(  n\right)  \right]  ^{-1}\right)
^{\ast}\lambda,X\right\rangle }\overline{\mathbf{s}\left(  n\right)  }d\mu
_{M}\left(  \gamma^{-1}n\right)  \right\vert ^{2}.
\end{align*}
Secondly, since
\[
\left\langle \left(  \left[  Ad\left(  n\right)  \right]  ^{-1}\right)
^{\ast}\lambda,X\right\rangle =\left\langle \Phi_{\boldsymbol{J}\left(
\lambda\right)  }^{_{\epsilon}}\left(  n\right)  ,X\right\rangle
\]
it follows that%
\begin{align*}
&
{\displaystyle\sum\limits_{\gamma\in\Gamma_{M\text{ }}}}
{\displaystyle\sum\limits_{\exp\left(  X\right)  \in\Gamma_{P}^{\epsilon}}}
\left\vert \left\langle \mathbf{h},\pi_{\lambda}\left(  \gamma^{-1}\right)
\pi_{\lambda}\left(  \exp X\right)  \mathbf{s}\right\rangle \right\vert ^{2}\\
&  =%
{\displaystyle\sum\limits_{\gamma\in\Gamma_{M\text{ }}}}
{\displaystyle\sum\limits_{\exp\left(  X\right)  \in\Gamma_{P}^{\epsilon}}}
\left\vert \int_{M}\mathbf{h}\left(  \gamma^{-1}n\right)  e^{-2\pi
i\left\langle \Phi_{\boldsymbol{J}\left(  \lambda\right)  }^{_{\epsilon}%
}\left(  n\right)  ,X\right\rangle }\overline{\mathbf{s}\left(  n\right)
}d\mu_{M}\left(  \gamma^{-1}n\right)  \right\vert ^{2}\\
&  =%
{\displaystyle\sum\limits_{\gamma\in\Gamma_{M\text{ }}}}
{\displaystyle\sum\limits_{\exp\left(  X\right)  \in\Gamma_{P}^{\epsilon}}}
\left\vert \int_{M}\mathbf{h}\left(  \gamma^{-1}n\right)  e^{-2\pi
i\left\langle \Phi_{\boldsymbol{J}\left(  \lambda\right)  }^{_{\epsilon}%
}\left(  n\right)  ,X\right\rangle }\overline{\mathbf{s}\left(  n\right)
}d\mu_{M}\left(  n\right)  \right\vert ^{2}\\
&  =%
{\displaystyle\sum\limits_{\gamma\in\Gamma_{M\text{ }}}}
{\displaystyle\sum\limits_{\exp\left(  X\right)  \in\Gamma_{P}^{\epsilon}}}
\left\vert \int_{\Omega_{\epsilon}^{\circ}}\mathbf{h}\left(  \gamma
^{-1}n\right)  e^{-2\pi i\left\langle \Phi_{\boldsymbol{J}\left(
\lambda\right)  }^{_{\epsilon}}\left(  n\right)  ,X\right\rangle }%
\overline{\mathbf{s}\left(  n\right)  }d\mu_{M}\left(  n\right)  \right\vert
^{2}=\left(  \ast\right)
\end{align*}
Thirdly, let $\xi=\Phi_{\boldsymbol{J}\left(  \lambda\right)  }^{_{\epsilon}%
}\left(  n\right)  $ and define
\[
\mathbf{z}_{\gamma}\left(  \xi\right)  =\mathbf{h}\left(  \gamma^{-1}\left[
\Phi_{\boldsymbol{J}\left(  \lambda\right)  }^{_{\epsilon}}\right]
^{-1}\left(  \xi\right)  \right)  \overline{\mathbf{s}\left(  \left[
\Phi_{\boldsymbol{J}\left(  \lambda\right)  }^{_{\epsilon}}\right]
^{-1}\left(  \xi\right)  \right)  }.
\]
Since
\[
d\mu_{M}\left(  \left(  \Phi_{\boldsymbol{J}\left(  \lambda\right)
}^{_{\epsilon}}\right)  ^{-1}\left(  \xi\right)  \right)  =\Psi
_{\boldsymbol{J}\left(  \lambda\right)  }^{\epsilon}\left(  \xi\right)  d\xi
\]
it follows that
\begin{align*}
&  \int_{\Omega_{\epsilon}^{\circ}}\mathbf{h}\left(  \gamma^{-1}n\right)
e^{-2\pi i\left\langle \Phi_{\boldsymbol{J}\left(  \lambda\right)
}^{_{\epsilon}}\left(  n\right)  ,X\right\rangle }\overline{\mathbf{s}\left(
n\right)  }d\mu_{M}\left(  n\right) \\
&  =\int_{\beta_{\boldsymbol{J}\left(  \lambda\right)  }\left(  \mathcal{O}%
_{\epsilon}\right)  }\mathbf{h}\left(  \gamma^{-1}\left[  \Phi_{\boldsymbol{J}%
\left(  \lambda\right)  }^{_{\epsilon}}\right]  ^{-1}\left(  \xi\right)
\right)  e^{-2\pi i\left\langle \xi,X\right\rangle }\overline{\mathbf{s}%
\left(  \left[  \Phi_{\boldsymbol{J}\left(  \lambda\right)  }^{_{\epsilon}%
}\right]  ^{-1}\left(  \xi\right)  \right)  }d\mu_{M}\left(  \left[
\Phi_{\boldsymbol{J}\left(  \lambda\right)  }^{_{\epsilon}}\right]
^{-1}\left(  \xi\right)  \right) \\
&  =\int_{\beta_{\boldsymbol{J}\left(  \lambda\right)  }\left(  \mathcal{O}%
_{\epsilon}\right)  }\underset{=\mathbf{z}_{\gamma}\left(  \xi\right)
}{\underbrace{\left(  \mathbf{h}\left(  \gamma^{-1}\left[  \Phi
_{\boldsymbol{J}\left(  \lambda\right)  }^{_{\epsilon}}\right]  ^{-1}\left(
\xi\right)  \right)  \overline{\mathbf{s}\left(  \left[  \Phi_{\boldsymbol{J}%
\left(  \lambda\right)  }^{_{\epsilon}}\right]  ^{-1}\left(  \xi\right)
\right)  }\right)  }}e^{-2\pi i\left\langle \xi,X\right\rangle }%
\underset{=\Psi_{\boldsymbol{J}\left(  \lambda\right)  }^{\epsilon}\left(
\xi\right)  d\xi}{\underbrace{d\mu_{M}\left(  \left[  \Phi_{\boldsymbol{J}%
\left(  \lambda\right)  }^{_{\epsilon}}\right]  ^{-1}\left(  \xi\right)
\right)  }}.
\end{align*}
As a consequence of the observations made above%
\begin{align*}
\left(  \ast\right)   &  =%
{\displaystyle\sum\limits_{\gamma\in\Gamma_{M\text{ }}}}
{\displaystyle\sum\limits_{\exp\left(  X\right)  \in\Gamma_{P}^{\epsilon}}}
\left\vert \int_{\beta_{\boldsymbol{J}\left(  \lambda\right)  }\left(
\mathcal{O}_{\epsilon}\right)  }\mathbf{z}_{\gamma}\left(  \xi\right)  \times
e^{-2\pi i\left\langle \xi,X\right\rangle }\times\Psi_{\boldsymbol{J}\left(
\lambda\right)  }^{\epsilon}\left(  \xi\right)  d\xi\right\vert ^{2}\\
&  =%
{\displaystyle\sum\limits_{\gamma\in\Gamma_{M\text{ }}}}
{\displaystyle\sum\limits_{\exp\left(  X\right)  \in\Gamma_{P}^{\epsilon}}}
\left\vert \int_{\beta_{\boldsymbol{J}\left(  \lambda\right)  }\left(
\mathcal{O}_{\epsilon}\right)  }\left[  \mathbf{z}_{\gamma}\left(  \xi\right)
\Psi_{\boldsymbol{J}\left(  \lambda\right)  }^{\epsilon}\left(  \xi\right)
\right]  e^{-2\pi i\left\langle \xi,X\right\rangle }d\xi\right\vert ^{2}\\
&  =\left(  \ast\ast\right)
\end{align*}
Next, note that the support of the function
\[
\xi\mapsto\mathbf{s}\left(  \left[  \Phi_{\boldsymbol{J}\left(  \lambda
\right)  }^{_{\epsilon}}\right]  ^{-1}\left(  \xi\right)  \right)
\]
is a compact subset $\Sigma_{\mathbf{s}}$ of $\beta_{\boldsymbol{J}\left(
\lambda\right)  }\left(  \mathcal{O}_{\epsilon}\right)  $ which is contained
in a fundamental domain of $\Lambda_{\epsilon}^{\star}.$ Moreover, the
trigonometric system
\[
\left\{  \xi\mapsto\frac{e^{-2\pi i\left\langle \xi,X\right\rangle }%
}{\left\vert \mathrm{\det}\left(  \mathcal{L}^{\epsilon}\right)  ^{\top
}\right\vert ^{1/2}}:X\in\Lambda_{\epsilon}\right\}
\]
is an orthonormal basis for $L^{2}\left(  \mathfrak{p}^{\ast}/\Lambda
_{\epsilon}^{\star}\right)  .$ Next,
\begin{align*}
&  \left(  \ast\ast\right) \\
&  =%
{\displaystyle\sum\limits_{\gamma\in\Gamma_{M\text{ }}}}
{\displaystyle\sum\limits_{\exp\left(  X\right)  \in\Gamma_{P}^{\epsilon}}}
\left\vert \int_{\Sigma_{\mathbf{s}}}\left[  \left\vert \mathrm{\det}\left(
\mathcal{L}^{\epsilon}\right)  ^{\top}\right\vert ^{1/2}\mathbf{z}_{\gamma
}\left(  \xi\right)  \Psi_{\boldsymbol{J}\left(  \lambda\right)  }^{\epsilon
}\left(  \xi\right)  \right]  \frac{e^{-2\pi i\left\langle \xi,X\right\rangle
}}{\left\vert \mathrm{\det}\left(  \mathcal{L}^{\epsilon}\right)  ^{\top
}\right\vert ^{1/2}}d\xi\right\vert ^{2}\\
&  =%
{\displaystyle\sum\limits_{\gamma\in\Gamma_{M\text{ }}}}
\int_{\Sigma_{\mathbf{s}}}\left\vert \left\vert \mathrm{\det}\left(
\mathcal{L}^{\epsilon}\right)  ^{\top}\right\vert ^{1/2}\mathbf{z}_{\gamma
}\left(  \xi\right)  \Psi_{\boldsymbol{J}\left(  \lambda\right)  }^{\epsilon
}\left(  \xi\right)  \right\vert ^{2}d\xi\\
&  =%
{\displaystyle\sum\limits_{\gamma\in\Gamma_{M\text{ }}}}
\int_{\Sigma_{\mathbf{s}}}\left\vert \left\vert \mathrm{\det}\left(
\mathcal{L}^{\epsilon}\right)  ^{\top}\right\vert ^{1/2}\mathbf{z}_{\gamma
}\left(  \xi\right)  \Psi_{\boldsymbol{J}\left(  \lambda\right)  }^{\epsilon
}\left(  \xi\right)  ^{1/2}\right\vert ^{2}\Psi_{\boldsymbol{J}\left(
\lambda\right)  }^{\epsilon}\left(  \xi\right)  d\xi\\
&  =\left(  \ast\ast\ast\right)
\end{align*}
Setting
\[
\left[  \Phi_{\boldsymbol{J}\left(  \lambda\right)  }^{_{\epsilon}}\right]
^{-1}\left(  \xi\right)  =n
\]
yields%
\begin{align*}
\mathbf{z}_{\gamma}\left(  \xi\right)   &  =\mathbf{h}\left(  \gamma
^{-1}\left[  \Phi_{\boldsymbol{J}\left(  \lambda\right)  }^{_{\epsilon}%
}\right]  ^{-1}\left(  \xi\right)  \right)  \overline{\mathbf{s}\left(
\left[  \Phi_{\boldsymbol{J}\left(  \lambda\right)  }^{_{\epsilon}}\right]
^{-1}\left(  \xi\right)  \right)  }\\
&  =\mathbf{h}\left(  \gamma^{-1}\left[  \Phi_{\boldsymbol{J}\left(
\lambda\right)  }^{_{\epsilon}}\right]  ^{-1}\left(  \Phi_{\boldsymbol{J}%
\left(  \lambda\right)  }^{_{\epsilon}}\left(  n\right)  \right)  \right)
\overline{\mathbf{s}\left(  \left[  \Phi_{\boldsymbol{J}\left(  \lambda
\right)  }^{_{\epsilon}}\right]  ^{-1}\left(  \Phi_{\boldsymbol{J}\left(
\lambda\right)  }^{_{\epsilon}}\left(  n\right)  \right)  \right)  }\\
&  =\mathbf{h}\left(  \gamma^{-1}n\right)  \overline{\mathbf{s}\left(
n\right)  }%
\end{align*}
and%
\begin{align*}
\left(  \ast\ast\ast\right)   &  =%
{\displaystyle\sum\limits_{\gamma\in\Gamma_{M\text{ }}^{\epsilon}}}
\int_{\left[  \Phi_{\boldsymbol{J}\left(  \lambda\right)  }^{_{\epsilon}%
}\right]  ^{-1}\left(  \Sigma_{\mathbf{s}}\right)  }\left\vert \left\vert
\mathrm{\det}\left(  \mathcal{L}^{\epsilon}\right)  ^{\top}\right\vert
^{1/2}\mathbf{h}\left(  \gamma^{-1}n\right)  \overline{\mathbf{s}\left(
n\right)  }\left[  \Psi_{\boldsymbol{J}\left(  \lambda\right)  }^{\epsilon
}\left(  \Phi_{\boldsymbol{J}\left(  \lambda\right)  }^{_{\epsilon}}\left(
n\right)  \right)  \right]  ^{1/2}\right\vert ^{2}\\
&  \times\left(  \Psi_{\boldsymbol{J}\left(  \lambda\right)  }^{\epsilon
}\left(  \Phi_{\boldsymbol{J}\left(  \lambda\right)  }^{_{\epsilon}}\left(
n\right)  \right)  \right)  d\left(  \Phi_{\boldsymbol{J}\left(
\lambda\right)  }^{_{\epsilon}}\left(  n\right)  \right)  .
\end{align*}
Since
\[
d\mu_{M}\left(  \left[  \Phi_{\boldsymbol{J}\left(  \lambda\right)
}^{_{\epsilon}}\right]  ^{-1}\left(  \xi\right)  \right)  =\Psi
_{\boldsymbol{J}\left(  \lambda\right)  }^{\epsilon}\left(  \xi\right)  d\xi
\]
it follows that
\begin{align*}
&
{\displaystyle\sum\limits_{\gamma\in\Gamma_{M\text{ }}}}
\int_{\left[  \Phi_{\boldsymbol{J}\left(  \lambda\right)  }^{_{\epsilon}%
}\right]  ^{-1}\left(  \Sigma_{\mathbf{s}}\right)  }\left\vert \left\vert
\mathrm{\det}\left(  \mathcal{L}^{\epsilon}\right)  ^{\top}\right\vert
^{1/2}\mathbf{h}\left(  \gamma^{-1}n\right)  \overline{\mathbf{s}\left(
n\right)  }\left[  \Psi_{\boldsymbol{J}\left(  \lambda\right)  }^{\epsilon
}\left(  \Phi_{\boldsymbol{J}\left(  \lambda\right)  }^{_{\epsilon}}\left(
n\right)  \right)  \right]  ^{1/2}\right\vert ^{2}\\
&  \times\Psi_{\boldsymbol{J}\left(  \lambda\right)  }^{\epsilon}\left(
\Phi_{\boldsymbol{J}\left(  \lambda\right)  }^{_{\epsilon}}\left(  n\right)
\right)  d\left(  \Phi_{\boldsymbol{J}\left(  \lambda\right)  }^{_{\epsilon}%
}\left(  n\right)  \right)  .
\end{align*}
Thus,
\begin{align*}
&
{\displaystyle\sum\limits_{\exp\left(  X\right)  \in\Gamma_{P}^{\epsilon}}}
{\displaystyle\sum\limits_{\gamma\in\Gamma_{M\text{ }}}}
\left\vert \left\langle \mathbf{h},\pi_{\lambda}\left(  \gamma^{-1}\right)
\pi_{\lambda}\left(  \exp X\right)  \mathbf{s}\right\rangle \right\vert ^{2}\\
&  =%
{\displaystyle\sum\limits_{\gamma\in\Gamma_{M\text{ }}}}
\int_{\left[  \Phi_{\boldsymbol{J}\left(  \lambda\right)  }^{_{\epsilon}%
}\right]  ^{-1}\left(  \Sigma_{\mathbf{s}}\right)  }\left\vert \left\vert
\mathrm{\det}\left(  \mathcal{L}^{\epsilon}\right)  ^{\top}\right\vert
^{1/2}\mathbf{h}\left(  \gamma^{-1}n\right)  \overline{\mathbf{s}\left(
n\right)  }\left[  \Psi_{\boldsymbol{J}\left(  \lambda\right)  }^{\epsilon
}\left(  \Phi_{\boldsymbol{J}\left(  \lambda\right)  }^{_{\epsilon}}\left(
n\right)  \right)  \right]  ^{1/2}\right\vert ^{2}\\
&  \times d\mu_{M}\left(  \left[  \Phi_{\boldsymbol{J}\left(  \lambda\right)
}^{_{\epsilon}}\right]  ^{-1}\left(  \Phi_{\boldsymbol{J}\left(
\lambda\right)  }^{_{\epsilon}}\left(  n\right)  \right)  \right) \\
&  =%
{\displaystyle\sum\limits_{\gamma\in\Gamma_{M\text{ }}}}
\int_{\left[  \Phi_{\boldsymbol{J}\left(  \lambda\right)  }^{_{\epsilon}%
}\right]  ^{-1}\left(  \Sigma_{\mathbf{s}}\right)  }\left\vert \left\vert
\mathrm{\det}\left(  \mathcal{L}^{\epsilon}\right)  ^{\top}\right\vert
^{1/2}\mathbf{h}\left(  \gamma^{-1}n\right)  \overline{\mathbf{s}\left(
n\right)  }\left[  \Psi_{\boldsymbol{J}\left(  \lambda\right)  }^{\epsilon
}\left(  \Phi_{\boldsymbol{J}\left(  \lambda\right)  }^{_{\epsilon}}\left(
n\right)  \right)  \right]  ^{1/2}\right\vert ^{2}d\mu_{M}\left(  n\right) \\
&  =%
{\displaystyle\sum\limits_{\gamma\in\Gamma_{M\text{ }}}}
\int_{\left[  \Phi_{\boldsymbol{J}\left(  \lambda\right)  }^{_{\epsilon}%
}\right]  ^{-1}\left(  \Sigma_{\mathbf{s}}\right)  }\left\vert \mathbf{h}%
\left(  \gamma^{-1}n\right)  \right\vert ^{2}\left\vert \left\vert
\mathrm{\det}\left(  \mathcal{L}^{\epsilon}\right)  ^{\top}\right\vert
^{1/2}\mathbf{s}\left(  n\right)  \left[  \Psi_{\boldsymbol{J}\left(
\lambda\right)  }^{\epsilon}\left(  \Phi_{\boldsymbol{J}\left(  \lambda
\right)  }^{_{\epsilon}}\left(  n\right)  \right)  \right]  ^{1/2}\right\vert
^{2}d\mu_{M}\left(  n\right) \\
&  =\left(  \ast\ast\ast\ast\right)
\end{align*}
The change of variable $m=\gamma^{-1}n$ yields
\begin{align*}
&  \left\vert \mathbf{h}\left(  \gamma^{-1}n\right)  \right\vert
^{2}\left\vert \left\vert \mathrm{\det}\left(  \mathcal{L}^{\epsilon}\right)
^{\top}\right\vert ^{1/2}\mathbf{s}\left(  n\right)  \left[  \Psi
_{\boldsymbol{J}\left(  \lambda\right)  }^{\epsilon}\left(  \Phi
_{\boldsymbol{J}\left(  \lambda\right)  }^{_{\epsilon}}\left(  n\right)
\right)  \right]  ^{1/2}\right\vert ^{2}\\
&  =\left\vert \mathbf{h}\left(  m\right)  \right\vert ^{2}\left\vert
\left\vert \mathrm{\det}\left(  \mathcal{L}^{\epsilon}\right)  ^{\top
}\right\vert ^{1/2}\mathbf{s}\left(  \gamma m\right)  \left[  \Psi
_{\boldsymbol{J}\left(  \lambda\right)  }^{\epsilon}\left(  \Phi
_{\boldsymbol{J}\left(  \lambda\right)  }^{_{\epsilon}}\left(  \gamma
m\right)  \right)  \right]  ^{1/2}\right\vert ^{2}.
\end{align*}
Since $d\mu_{M}$ is a Haar measure, then
\begin{align*}
&  \left(  \ast\ast\ast\ast\right) \\
&  =%
{\displaystyle\sum\limits_{\gamma\in\Gamma_{M\text{ }}}}
\int_{\gamma^{-1}\left[  \Phi_{\boldsymbol{J}\left(  \lambda\right)
}^{_{\epsilon}}\right]  ^{-1}\left(  \Sigma_{\mathbf{s}}\right)  }\left\vert
\mathbf{h}\left(  m\right)  \right\vert ^{2}\left\vert \left\vert
\mathrm{\det}\left(  \mathcal{L}^{\epsilon}\right)  ^{\top}\right\vert
^{1/2}\mathbf{s}\left(  \gamma m\right)  \left[  \Psi_{\boldsymbol{J}\left(
\lambda\right)  }^{\epsilon}\left(  \Phi_{\boldsymbol{J}\left(  \lambda
\right)  }^{_{\epsilon}}\left(  \gamma m\right)  \right)  \right]
^{1/2}\right\vert ^{2}d\mu_{M}\left(  m\right) \\
&  =%
{\displaystyle\sum\limits_{\gamma\in\Gamma_{M\text{ }}}}
\int_{\gamma^{-1}\left[  \Phi_{\boldsymbol{J}\left(  \lambda\right)
}^{_{\epsilon}}\right]  ^{-1}\left(  \Sigma_{\mathbf{s}}\right)  }\left\vert
\mathbf{h}\left(  m\right)  \right\vert ^{2}\left\vert \sqrt{\Psi
_{\boldsymbol{J}\left(  \lambda\right)  }^{\epsilon}\left(  \Phi
_{\boldsymbol{J}\left(  \lambda\right)  }^{_{\epsilon}}\left(  \gamma
m\right)  \right)  \left\vert \mathrm{\det}\left(  \mathcal{L}^{\epsilon
}\right)  ^{\top}\right\vert }\mathbf{s}\left(  \gamma m\right)  \right\vert
^{2}d\mu_{M}\left(  m\right) \\
&  =\int_{M}\left\vert \mathbf{h}\left(  m\right)  \right\vert ^{2}\left(
{\displaystyle\sum\limits_{\gamma\in\Gamma_{M\text{ }}}}
\left\vert \sqrt{\Psi_{\boldsymbol{J}\left(  \lambda\right)  }^{\epsilon
}\left(  \Phi_{\boldsymbol{J}\left(  \lambda\right)  }^{_{\epsilon}}\left(
\gamma m\right)  \right)  \left\vert \mathrm{\det}\left(  \mathcal{L}%
^{\epsilon}\right)  ^{\top}\right\vert }\mathbf{s}\left(  \gamma m\right)
\right\vert ^{2}\right)  d\mu_{M}\left(  m\right)  .
\end{align*}
Next, let
\[
\mathcal{Z}^{\epsilon}=%
\begin{array}
[c]{c}%
\text{ }\\
\left\{  \left.  \Gamma\text{ is a discrete subset of }M\right\vert \text{
}\inf_{m\in M}\left(
{\displaystyle\sum\limits_{\gamma\in\Gamma_{\text{ }}}}
\left\vert \sqrt{\Psi_{\boldsymbol{J}\left(  \lambda\right)  }^{\epsilon
}\left(  \Phi_{\boldsymbol{J}\left(  \lambda\right)  }^{_{\epsilon}}\left(
\gamma m\right)  \right)  \left\vert \mathrm{\det}\left(  \mathcal{L}%
^{\epsilon}\right)  ^{\top}\right\vert }\mathbf{s}\left(  \gamma m\right)
\right\vert ^{2}\right)  >0\text{ }\right. \\
\left.  \sup_{m\in M}\left(
{\displaystyle\sum\limits_{\gamma\in\Gamma_{M\text{ }}}}
\left\vert \sqrt{\Psi_{\boldsymbol{J}\left(  \lambda\right)  }^{\epsilon
}\left(  \Phi_{\boldsymbol{J}\left(  \lambda\right)  }^{_{\epsilon}}\left(
\gamma m\right)  \right)  \left\vert \mathrm{\det}\left(  \mathcal{L}%
^{\epsilon}\right)  ^{\top}\right\vert }\mathbf{s}\left(  \gamma m\right)
\right\vert ^{2}\right)  <\infty\right\}
\end{array}
.
\]
Since
\[
\mathbf{s}\left(  m\right)  \sqrt{\Psi_{\boldsymbol{J}\left(  \lambda\right)
}^{\epsilon}\left(  \Phi_{\boldsymbol{J}\left(  \lambda\right)  }^{_{\epsilon
}}\left(  m\right)  \right)  }%
\]
is a smooth function of compact support, $\mathcal{Z}^{\epsilon}$ is
non-empty. Fixing $\Gamma_{M\text{ }}^{\epsilon}\in\mathcal{Z}^{\epsilon}$, it
is clear that there exists $A_{\mathbf{s,}\Gamma_{M\text{ }}^{\epsilon
},\Lambda^{\star}}>0,$ and $B_{\mathbf{s,}\Gamma_{M\text{ }}^{\epsilon
},\Lambda^{\star}}<\infty$ such that
\begin{align*}
A_{\mathbf{s,}\Gamma_{M\text{ }}^{\epsilon},\Lambda^{\star}}\times\left\Vert
\mathbf{h}\right\Vert _{L^{2}\left(  M,d\mu_{M}\right)  }^{2}  &  \leq%
{\displaystyle\sum\limits_{\exp\left(  X\right)  \in\Gamma_{P}^{\epsilon}}}
{\displaystyle\sum\limits_{\gamma\in\Gamma_{M\text{ }}^{\epsilon}}}
\left\vert \left\langle \mathbf{h},\pi_{\lambda}\left(  \gamma^{-1}\right)
\pi_{\lambda}\left(  \exp X\right)  \mathbf{s}\right\rangle \right\vert ^{2}\\
&  \leq B_{\mathbf{s,}\Gamma_{M\text{ }}^{\epsilon},\Lambda^{\star}}%
\times\left\Vert \mathbf{h}\right\Vert _{L^{2}\left(  M,d\mu_{M}\right)  }^{2}%
\end{align*}
and
\[
\left\{  \pi_{\lambda}\left(  \kappa\right)  \mathbf{s}:\kappa\in\left(
\Gamma_{M\text{ }}^{\epsilon}\right)  ^{-1}\Gamma_{P}^{\epsilon}\right\}
\]
is a frame generated by a smooth function of compact support.

\begin{acknowledgement}
I am thankful and grateful to Professors Hartmut Fuehr and Azita Mayeli for
proofreading and for providing valuable feedback and corrections that helped
me improve earlier versions of this manuscript.
\end{acknowledgement}

\end{document}